\documentclass[11pt,reqno]{amsart}
\usepackage{amssymb,bm,mathrsfs,stmaryrd, enumerate,bbm,amsmath,amsthm}
\usepackage[all]{xy}
\usepackage{hyperref}

\newcommand{\map}[1]{\xrightarrow{#1}}

\newcommand{\iso}{\cong}
\newcommand{\define}{\stackrel{\mathrm{def}}{=}}
\newcommand{\semi}{\rtimes}

\newcommand{\Gal}{\mathrm{Gal}}
\newcommand{\Hom}{\mathrm{Hom}}
\newcommand{\Aut}{\mathrm{Aut}}
\newcommand{\End}{\mathrm{End}}
\newcommand{\Spec}{\mathrm{Spec}}
\newcommand{\Q}{\mathbb Q}
\newcommand{\Z}{\mathbb Z}
\newcommand{\R}{\mathbb R}
\newcommand{\C}{\mathbb C}

\newcommand{\m}{\mathfrak m}
\newcommand{\A}{\mathbb A}
\newcommand{\co}{\mathcal O}
\newcommand{\alg}{\mathrm{alg}}
\newcommand{\action}{\bullet}

\newcommand{\GL}{\mathrm{GL}}
\newcommand{\PGL}{\mathrm{PGL}}
\newcommand{\Nm}{\mathrm{Nm}}
\newcommand{\Tr}{\mathrm{Tr}}
\newcommand{\Div}{\mathrm{Div}}
\newcommand{\Frob}{\mathrm{Frob}}
\newcommand{\Bun}{\mathrm{Bun}}
\newcommand{\Fr}{\mathrm{Fr}}

\newcommand{\orr}[1]{\overrightarrow{#1}}

\newcommand{\reg}{\mathrm{reg}}

\newcommand{\inv}{\mathrm{inv}}
\newcommand{\kk}{{\bm{k}}}
\newcommand{\cusp}{\mathrm{cusp}}
\newcommand{\Pic}{\mathrm{Pic}}
\newcommand{\newdet}{\mathrm{det}^\sharp}

\newcommand{\K}{\mathbb K}
\newcommand{\J}{\mathbb J}
\newcommand{\I}{\mathbb I}

\newcommand{\G}{\mathbb G}
\newcommand{\Sht}{\mathrm{Sht}}
\newcommand{\Ch}{\mathrm{Ch}}
\newcommand{\Hk}{\mathrm{Hk}}
\newcommand{\Eis}{\mathrm{Eis}}

\begin{document}
\author{Benjamin Howard and Ari Shnidman}
\title{A Gross-Kohnen-Zagier formula for Heegner-Drinfeld cycles}

\thanks{B.H.~was supported in part by NSF grant DMS-0901753.}

\address{Department of Mathematics\\Boston College\\ 140 Commonwealth Ave. \\Chestnut Hill, MA 02467, USA}
\email{howardbe@bc.edu}

\address{Department of Mathematics\\Boston College\\ 140 Commonwealth Ave. \\Chestnut Hill, MA 02467, USA}
\curraddr{Einstein Institute of Mathematics\\ Hebrew University of Jerusalem\\ Israel}
\email{ariel.shnidman@mail.huji.ac.il}

\begin{abstract}
Let $F$ be the field of rational functions on a smooth projective curve over a finite field, and let $\pi$ be an unramified cuspidal automorphic representation for $\PGL_2$ over $F$.  
We prove a variant of the formula of Yun and Zhang relating derivatives of the $L$-function of $\pi$ to  the self-intersections of Heegner-Drinfeld cycles on  moduli spaces of shtukas.  

In our variant, instead of a self-intersection,  we compute the  intersection pairing of  Heegner-Drinfeld cycles coming from two different quadratic extensions of $F$, and relate the intersection to   the $r$-th derivative of a product of two toric period integrals.  
\end{abstract}

\maketitle

\theoremstyle{plain}
\newtheorem{theorem}{Theorem}[section]
\newtheorem{bigtheorem}{Theorem}[section]
\newtheorem{proposition}[theorem]{Proposition}
\newtheorem{lemma}[theorem]{Lemma}
\newtheorem{corollary}[theorem]{Corollary}

\theoremstyle{definition}
\newtheorem{definition}[theorem]{Definition}

\theoremstyle{remark}
\newtheorem{remark}[theorem]{Remark}
\newtheorem{question}[theorem]{Question}

\numberwithin{equation}{section}
\renewcommand{\thebigtheorem}{\Alph{bigtheorem}}

\setcounter{tocdepth}{1}
\tableofcontents


\section{Introduction}


Let $X$ be a smooth, projective, geometrically connected curve over a finite field $\kk$, and let $F = \kk(X)$ be the field of rational functions on $X$.  
Any finite \'etale cover $f : X' \to X$ of degree 2, with $X'$ geometrically connected, determines a quadratic extension $F' = \kk(X')$ of $F$, and hence determines a nonsplit torus $T = (\mathrm{Res}_{F'/F}\mathbb{G}_m)/\G_m$ of rank one.

Let $r \geq 0$ be an even integer, and suppose that $\pi$ is an unramified cuspidal automorphic representation of the adelic points of $G = {\PGL_2}_{/F}$.    
Yun and Zhang \cite{YZ} have recently proved a remarkable formula for the $r$-th central derivative of the normalized base-change $L$-function $\mathscr{L}(\pi_{F'}, s)$ in terms of the self-intersection of a Heegner-Drinfeld cycle $[\Sht_T^{\prime r}]_\pi$.  The latter is an algebraic cycle of dimension $r$ on the $2r$-dimensional $\kk$-stack 
\[
\Sht_G^{\prime r} = \Sht_G^r \times_{X^r} X^{\prime r},
\]  
where $\Sht_{G}^r \to X^r$ is the moduli stack of $G$-shtukas with $r$ modifications.  In particular, they prove that 
\[
\mathscr{L}^{(r)}(\pi_{F'}, 1/2) 
=  \kappa \cdot \langle [\Sht_T^{\prime r}]_\pi ,  [\Sht_T^{\prime r}]_\pi \rangle,
\]
for a certain non-zero real number $\kappa$.  The right hand side is the intersection pairing on $\Ch_{c,r}(\Sht_G^{\prime r})_\R$, the middle dimensional Chow group with proper support and $\R$-coefficients.   


Our goal is to prove a variant of this result, where the self-intersection of a Heegner-Drinfeld cycle is replaced by the intersection pairing of two \emph{different} Heeger-Drinfeld cycles, in the spirit of the Gross-Kohnen-Zagier theorem \cite{GKZ}.

\subsection{Statement of the results}

Suppose we are instead given two nonisomorphic finite  \'etale double covers
\[
f_1 : Y_1 \to X ,\quad f_2: Y_2 \to X,
\] 
where $Y_1$ and $Y_2$ are  projective and geometrically connected.  It is  natural to wonder how the corresponding Heegner-Drinfeld cycles are related, and whether their intersection pairing is also related to $L$-values.  We answer these questions by giving a formula (Theorem \ref{thm:main intro}) for the intersection pairing of the corresponding Heegner-Drinfeld cycles on $\Sht_{G}^r$ in terms of certain toric period integrals.



To state this result, we must introduce some notation.  Denote by $\sigma_i \in \Aut(Y_i/X)$  the nontrivial automorphism.  
The fiber product  \[ Y=Y_1\times_X Y_2 \] is again a smooth geometrically connected curve, and $X$ is its quotient by the action of the Klein four group
$
\{ \mathrm{id} ,  \tau_1,\tau_2 , \tau_3 \}=  \Aut(Y/X)
$
with nontrivial elements 
\[
(y_1,y_2)^{\tau_1} = (y_1 , y_2^{\sigma_2}),\quad 
(y_1,y_2)^{\tau_2} = (y_1 ^{\sigma_1}, y_2), \quad
(y_1,y_2)^{\tau_3} = (y_1 ^{\sigma_1}, y_2^{\sigma_2}).
\]

The quotient of $Y$ by the action of $\tau_3$ is a geometrically connected \'etale double cover of $X$, which we denote by $Y_3$.
The picture is
\begin{equation}\label{biquad curves}
\xymatrix{
&  {  Y } \ar[dl] \ar[d] \ar[dr]  \\
{    Y_1 }  \ar[dr]_{f_1}   &     {   Y_2  }  \ar[d]^{f_2}  &  {  Y_3 } \ar[dl]^{f_3}  \\
&  {   X   } .
}
\end{equation}
Taking fields of rational functions, the corresponding picture is
\begin{equation}\label{biquad diagram}
\xymatrix{
&  { K}  \\
{   K_1 }  \ar@{-}[ur]^{  \tau_1}  &     {   K_2   }  \ar@{-}[u]_{\tau_2} &  {  K_3  } \ar@{-}[ul]_{\tau_3 }  \\
&  {   F   }  \ar@{-}[ul]^{\sigma_1}  \ar@{-}[u]_{\sigma_2} \ar@{-}[ur]_{\sigma_3},
}
\end{equation}
where the labels  indicate the nontrivial automorphisms.  

There is a fourth quadratic algebra lurking in the background, corresponding to the trivial double cover $Y_0 = X\sqcup X$. Let $f_0 : Y_0 \to X$ be the natural morphism, and let $\sigma_0 \in \Aut(Y_0/X)$ be the nontrivial automorphism.  
The algebra of rational functions on $Y_0$ is $K_0=F\oplus F$.

For $i\in \{0,1,2,3\}$ there is a natural  closed immersion 
\[
\xymatrix{
{  \tilde{T}_i =\underline{\Aut}_{ f_{i*} \co_{Y_i}  } ( f_{i*} \co_{Y_i} )  }   \ar[d]   \\  
{  \tilde{G}_i= \underline{\Aut}_{\co_X} ( f_{i*} \co_{Y_i} )  } 
}
\] 
of group schemes over $X$.   The group scheme $\tilde{G}_i$ is Zariski locally isomorphic to $\GL_2$, and  $\tilde{T}_i \subset \tilde{G}_i$  is a maximal torus.
Let $T_i\subset G_i$ be obtained from this pair by quotienting out by the central $\G_m$.  It is Zariski locally isomorphic to a maximal torus in $\PGL_2$ (over $X$).     On  $F$-points the picture is
\[
\xymatrix{
{  \tilde{T}_i(F) = K_i^\times } \ar[r]\ar[d] & { T_i(F) = K_i^\times / F^\times  }   \ar[d]  \\
  {   \tilde{G}_i(F)   =  \Aut_F(K_i)  }    \ar[r]&  {   G_i(F)  =  \Aut_F(K_i) / F^\times  . } 
}
\]
Note that the canonical isomorphism $f_{0*}\co_{Y_0} \iso \co_X\oplus \co_X$ identifies 
\[
\tilde{G}_0 \iso \GL_2 ,\quad G_0 \iso \PGL_2, 
\]
and identifies $\tilde{T}_0$ and $T_0$ with their diagonal tori.

Let $\mathbb{O} \subset \A$ be the subring of integral elements in the adele ring of $F$.
Similarly,   for $i\in \{ 0,1,2,3\}$, denote by  $\A_i$ the adele ring of $K_i$, and by $\mathbb{O}_i \subset \A_i$ its subring of integral elements.
Define $U_i = G_i(\mathbb{O})$. The pair $U_i \subset G_i(\A)$ is isomorphic to $\PGL_2(\mathbb{O}) \subset \PGL_2(\A)$, but there is no canonical such isomorphism except when $i=0$.

There is, however, a canonical isomorphism of spaces of unramified cuspidal automorphic forms  \[\mathcal{A}_\cusp(G_i)^{U_i} \cong \mathcal{A}_\cusp(G_0)^{U_0}\]  by Lemma \ref{lem:naive transfer}.  These  are finite dimensional $\C$-vector spaces, and the space on the right carries a natural action of the  Hecke algebra $\mathscr{H}$ of $\Q$-valued compactly supported $U_0$-bi-invariant functions on $G_0(\A)$.  

\begin{remark}
Note that our automorphic forms are complex valued, as in \cite[\S 5]{Borel-Jacquet}, as opposed to the $\Q$-valued automorphic forms of \cite[\S 1.2]{YZ}.
\end{remark}

To set up the analytic side of our formula,  for any $\phi \in \mathcal{A}_\cusp(G_0)^{U_0}$ we define the $T_0$-period integral
\[
\mathscr{P}_0(\phi,s) = \int_{T_0(F)\backslash T_0(\A)} \phi(t_0)|t_0|^{2s} dt_0,
\]
which is absolutely convergent for all $s \in \C$.  Viewing $\phi$ in $\mathcal{A}_\cusp(G_3)^{U_3}$, we also define a $T_3$-period integral
\[
\mathscr{P}_3(\phi,\eta) = \int_{T_3(F) \backslash T_3(\A)} \phi(t_3)\eta(t_3)dt_3.
\]
Here  $\eta:\A_3^\times \to \{\pm 1\}$ is the quadratic character determined by the extension $K/K_3$.  
It descends to the quotient $T_3(\A) = \A_3^\times / \A^\times$ by  Lemma \ref{lem:eta}.  In both of these integrals, the Haar measure on $T_i(\A)$ is chosen so that the volume of $T_i(\mathbb{O})$ is 1.

Now suppose that $\pi$ is an unramified cuspidal automorphic representation of $G_0(\A)$, and define
\[
C(\pi, s) =   \frac{   \mathscr{P}_0 (  \phi ,s) \mathscr{P}_3(\overline{\phi} ,\eta)}{   \langle \phi, \phi \rangle_\mathrm{Pet} } 
\]
for any nonzero $\phi \in \pi^{U_0}$.  
If   $q$ denotes  the cardinality of $\kk$ and  $\lambda_\pi \colon \mathscr{H} \to \C$  denotes the character giving the action of $\mathscr{H}$ on the line $\pi^{U_0}$, we prove in Proposition \ref{prop:period reciprocity} that 
\[
C_r(\pi)  \define ( \log q)^{-r} \cdot  \frac{d^r}{ds^r}  C (  \pi ,s) \big|_{s=0}
\]  
 lies in the totally real number field $E_\pi=\lambda_\pi ( \mathscr{H}).$

To define the geometric side of our formula, recall from \cite{YZ} the stacks 
\[
\Sht_{T_1}^r \to Y_1^r ,\quad 
\Sht_{T_2}^r \to Y_2^r
\] 
of $T_i$-shtukas with $r$ modifications, and the stack \[\Sht_{G_0}^r \to X^r\] of $\PGL_2$-Shtukas with $r$ modifications.   
For $i\in \{1,2\}$ the stacks $\Sht_{T_i}^r$ are proper over $\kk$, and admit  finite morphisms $\Sht_{T_i}^r \to \Sht_{G_0}^r$.  
Denote by 
\[
[\Sht_{T_i}^r]\in  \Ch_{c,r}(\Sht_{G_0}^r)
\]
 the pushforward of the fundamental class under the induced map  on Chow groups with $\Q$-coefficients.

\begin{remark}
The definitions of these stacks require the choice of a tuple $\mu=(\mu_i)  \in \{\pm 1\}^r$ satisfying the parity condition $\sum_{i=1}^r \mu_i =0$, and in particular we must assume that $r$ is even.   We suppress the choice of $\mu$  from the notation. 
\end{remark}

Denote by $\tilde W_i \subset \Ch_{c,r}(\Sht_{G_0}^r)$  the $\mathscr{H}$-submodule generated by the class $[\Sht_{T_i}^r]$.
Restricting the intersection pairing on the Chow group defines a pairing
$
\langle \cdot , \cdot\rangle : \tilde{W}_1 \times \tilde{W}_2 \to \Q.
$
If we define  
\begin{align*}
W_1 & = \tilde{W}_1/  \{ c \in \tilde{W}_1 : \langle c, \tilde{W}_2\rangle =0 \}  \\
W_2 & = \tilde{W}_2/  \{ c \in \tilde{W}_2 : \langle c, \tilde{W}_1\rangle =0 \},
\end{align*}
this pairing obviously descends to $\langle \cdot , \cdot \rangle :  W_1 \times W_2 \to \Q$, and we extend it to an $\R$-bilinear pairing
\[
W_1(\R)  \times W_2(\R) \to \R.
\]

We prove in \S \ref{ss:main proofs} that there is a decomposition into isotypic components
\[
W_i (\R)= W_{i,\mathrm{Eis}} \oplus \left(\bigoplus _\pi W_{i, \pi}\right),
\] 
where the sum is over all  unramified cuspidal $\pi$, and $\mathscr{H}$ acts on the summand $W_{i, \pi}$ via $\lambda_\pi : \mathscr{H} \to \R$. Denote by
$
[\Sht_{T_i}^r]_\pi \in  W_{i, \pi} ,
$ 
the projection of   $[\Sht_{T_i}^r]\in W_i(\R)$ to this summand. Our main result is the following intersection formula.

\begin{bigtheorem}\label{thm:main intro}
For any unramified cuspidal automorphic representation $\pi$, and for any even $r \geq 0$,
\[
\langle [\Sht_{T_1}^r]_\pi, [\Sht_{T_2}^r]_\pi\rangle = C_r(\pi). 
\]  
\end{bigtheorem}

Theorem \ref{thm:main intro} is the function field analogue of the Gross-Kohnen-Zagier formula \cite[Theorem B]{GKZ}, but for higher order derivatives.  The original Gross-Kohnen-Zagier formula corresponds to the case $r = 1$ over $\Q$.  Their result allows for ramified $\pi$, and we expect our formula can be extended to mildly ramified $\pi$ as well.  Indeed, the case of Iwahori level structure was recently announced by Hao Li in \cite{hao}.      

The proof of Theorem \ref{thm:main intro} is very different from that of \cite{GKZ}, and is much closer  to the relative trace formula arguments of  \cite{Jacquet, YZ}.  

As an application of Theorem \ref{thm:main intro}, we obtain a criterion for the non-vanishing of certain special values of $L$-functions and their derivatives.  Let $L(\pi,s)$ be the standard $L$-function and set 
\[\mathscr{L}(\pi, s) = q^{2(g-1)(s-1/2)}L(\pi,s),\] 
where $g$ is the genus of $X$.  Then $\mathscr{L}(\pi,s) = \mathscr{L}(\pi,1-s)$, and  moreover, there is a choice of spherical vector $\phi$ so that $\mathscr{P}_0(\phi,s) = \mathscr{L}(\pi,2s + 1/2) $.

\begin{bigtheorem}\label{thm:nonvanishing}
Let $\chi_1$ and $\chi_2$ be the quadratic characters of $\A^\times$ corresponding to the extensions $K_1/F$ and $K_2/F$. The Heegner-Drinfeld cycles satisfy
 \[
\langle [\Sht_{T_1}^r]_{\pi} ,  [\Sht_{T_2}^r]_{\pi }\rangle  = 0
\] 
if and only if 
\[
\mathscr{L}^{(r)} (\pi, 1/2 ) \cdot \mathscr{L}(\pi \otimes \chi_1, 1/2 )  \cdot \mathscr{L}(\pi \otimes \chi_2, 1/2) = 0,
\]
where $\mathscr{L}^{(r)} (\pi, s )$ is the $r^\mathrm{th}$ derivative of $\mathscr{L}(\pi,s)$.
\end{bigtheorem}

There are three main differences between our results and the main result of \cite{YZ}.  First, of course, we are computing the intersection of two distinct cycles, as opposed to a self-intersection.  Second, our intersection pairing is on the $X^r$-stack $\Sht_{G_0}^r$, not its base-change to the $r$-fold product of a double cover as in \cite{YZ}.  More precisely, our Heegner-Drinfeld cycle  $[\Sht_{T_i}^r]$ is the pushforward via the $2^r$-fold cover
\[
\Sht_{G_0}^r \times_{X^r} Y_i^r  \to \Sht_{G_0}^r
\]
of the Heegner-Drinfeld cycle of \cite{YZ}.  Finally, our formula relates the intersection pairing to the standard $L$-function $L(\pi,s)$, not the base-change to a quadratic extension.

We also prove the following result on the intersection of the ``bare'' Heegner-Drinfeld cycles without any projection.

\begin{bigtheorem}\label{thm:bare intersection}
Let $r \geq 0$ be an even integer, and let 
\[
\langle\cdot,\cdot\rangle :\Ch_{c,r}(\Sht_{G_0}^r) \times \Ch_{c,r}(\Sht_{G_0}^r) \to \Q
\] 
be the intersection pairing.
\begin{enumerate}[$(a)$]
\item If $r>0$ then $\langle  [\Sht_{T_1}^r ]  ,   \, [\Sht_{T_2}^r ] \rangle  = 0$.
\item If $r=0$ then  
\[
\langle  [\Sht_{T_1}^0 ]  , \,   [\Sht_{T_2}^0 ] \rangle = 
\begin{cases} 1 &  \mbox{ if }  \mathrm{char}(\kk)>2 \\
0 & \mbox{ if } \mathrm{char}(\kk)=2 .
\end{cases}\]
\end{enumerate}
\end{bigtheorem}

When $r=0$, the intersection pairings in the above theorems can be reinterpreted as period integrals of automorphic forms.  This is spelled out explicitly in \S \ref{ss:r=0}, where the reader can also find a reinterpretation of the $r=0$ case of Theorem \ref{thm:bare intersection}, in a way that perhaps clarifies why the characteristic of $\kk$ plays a role.  See especially Remark \ref{rem:no optimal}.

Suppose that $\pi$ is an unramified cuspidal automorphic representation of $G_0(\A)$, and that $\phi \in \pi^{U_0}$ is a spherical vector.
As noted earlier, Lemma \ref{lem:naive transfer} provides  canonical isomorphisms
\[
 \mathcal{A}_\mathrm{cusp}(G_0)^{U_0} \iso  \mathcal{A}_\mathrm{cusp}(G_i)^{U_i}
\]
for $i\in \{1,2,3\}$, and we view $\phi$ as an automorphic form on any of the four groups $G_i$.
In \S \ref{ss:r=0} we show that the following result, relating the periods of this form along the four tori $T_i\subset G_i$,  is equivalent to the $r=0$ case of  Theorem \ref{thm:main intro}.

\begin{bigtheorem}\label{thm:periods}
For  any $\phi \in \pi^{U_0}$ as above, we have
\begin{multline*}
   \left( \int_{T_1(F)\backslash T_1(\A)} \phi(t_1) dt_1 \right)  \left( \int_{T_2(F) \backslash T_2(\A)} \overline{\phi}(t_2) dt_2\right) \\
    = 
    \left( \int_{T_0(F)\backslash T_0(\A)} \phi(t_0) dt_0 \right)  \left( \int_{T_3(F) \backslash T_3(\A)} \overline{\phi}(t_3)\eta(t_3)dt_3\right).
 \end{multline*}
\end{bigtheorem}

\begin{remark}
If one takes the absolute value of both sides of the equality,  then Waldspurger's formula \cite{Waldspurger, CW} relates all four toric periods to $L$-values, and the equality of absolute values follows from a formal manipulation of Euler products.   Thus, as both sides of the stated equality are easily seen to be real numbers (the $\C$-line $\pi^{U_0}$ is generated by some $\R$-valued $\phi$), the only new information is  that the signs agree.
\end{remark}

Just as Waldspurger's period formula generalizes to higher rank unitary and orthogonal groups, as in the conjectures of Gan-Gross-Prasad and the work of Zhang \cite{zhang}, one could hope that there are analogues of Theorem \ref{thm:periods} also on higher rank groups.  
In particular, one might expect that period relations similar to those of Theorem \ref{thm:periods} should appear in the context of  \cite{FMW} and \cite{guo}.


\subsection{The case of self-intersection}


Our methods  apply to the case $Y_1\iso Y_2$ with relatively minor adjustments, and the analogues of  Theorems \ref{thm:main intro} and \ref{thm:nonvanishing} hold verbatim (but not Theorem \ref{thm:bare intersection}).  Indeed, while this paper was under review, such analogues appeared in the sequel \cite{YZ2} of Yun-Zhang.  

For example, if we denote by $X'$ the curve $Y_1\iso Y_2$, and by $T$ the torus $T_1\iso T_2$, then (\ref{biquad curves}) degenerates to
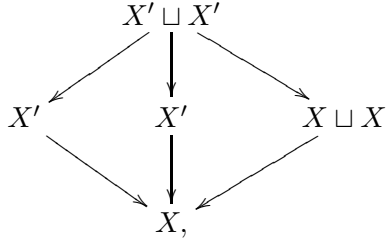
\[
\xymatrix{
&  {  X'\sqcup X' } \ar[dl] \ar[d] \ar[dr]  \\
{    X' }  \ar[dr]   &     {   X'  }  \ar[d]  &  {  X\sqcup X } \ar[dl]  \\
&  {   X   } ,
}
\]
 and our Heegner-Drinfeld cycle $[\Sht_T^r] \in  \Ch_{c,r}(\Sht_{G_0}^r)$  is, as was noted before, the pushforward via the $2^r$-fold cover
\[
\Sht_{G_0}^r \times_{X^r} X^{\prime r}  \to \Sht_{G_0}^r
\]
of the Heegner-Drinfeld cycle of \cite{YZ}.  
The analogue of Theorem \ref{thm:nonvanishing} in this setting says that 
\[
\langle [\Sht_T^r]_{\pi} ,  [\Sht_T^r]_{\pi }\rangle = 0 \iff \mathscr{L}^{(r)} (\pi, 1/2 ) \cdot \mathscr{L}(\pi \otimes \chi , 1/2 ) = 0,
\]
where $\chi$ is the quadratic character  of $\A^\times$ determined by  $X'\to X$.  

As the assumption $Y_1 \not\iso Y_2$ leads to some technical simplifications, we keep it throughout the paper.

%


\subsection{Notation}
\label{ss:notation}


After \S \ref{s:matching}, which contains  results about quaternion algebras over arbitrary  fields, the curves and functions fields of (\ref{biquad curves}) and (\ref{biquad diagram}) will remain fixed,  as will the group schemes $T_i \to G_i$ over $X$.

Let $|X|$ be the set of closed points of $X$. We normalize the absolute value 
\[
| \cdot |= \prod_{x \in |X|} |\cdot|_x   \colon \A^\times \to \Z
\] 
 so that a uniformizer $\pi_x \in F_x$ with residue field $\kk_x$ satisfies
$| \pi_x|_x   = q^{ - [\kk_x \colon \kk]}$.

 The Haar measures on $T_i(\A)$ and $G_i(\A)$ are normalized  so that $T_i(\mathbb{O})$ and $G_i(\mathbb{O})$ have volume 1.

Unless otherwise specified, all Chow groups have $\Q$-coefficients.  


\subsection{Acknowledgements}


We thank the referee for a careful reading of the paper, and for suggesting several improvements and simplifications.  We also thank Spencer Leslie for helpful conversations. 


\section{Biquadratic algebras and quaternion algebras}
\label{s:matching}


In this section alone we allow  $F$ to be any field whatsoever, and let $K_1$ and $K_2$ be  quadratic \'etale $F$-algebras.    
In other words,  each $K_i$ is either a degree two Galois extension, or   $K_i\iso F\oplus F$.  In particular,  the results below apply both to the global fields of (\ref{biquad diagram}) and to their completions.


\subsection{Invariants of quaternion embeddings}
\label{ss:invariants}


As usual, a \emph{quaternion algebra} over $F$ is a central simple $F$-algebra of dimension four.

\begin{definition}
A \emph{quaternion embedding of $(K_1,K_2)$}  is a triple $(B, \alpha_1,\alpha_2)$ consisting of a quaternion algebra $B$ over $F$, and $F$-algebra embeddings $\alpha_1 : K_1 \to B$ and $\alpha_2 : K_2 \to B$.
Such a  triple is  \emph{regular} if $\alpha_1(K_1)\cup \alpha_2(K_2)$ generates $B$ as an $F$-algebra.   
\end{definition}

Our goal is to parametrize the set
\[
\mathcal{Q}(K_1,K_2)  = \{\mbox{isomorphism classes of quaternion embeddings } (B,\alpha_1,\alpha_2)  \},
\]  
along with its subset   
\begin{equation}\label{regular subset}
\mathcal{Q}_\reg(K_1,K_2) \subset \mathcal{Q}(K_1,K_2)
\end{equation}
 of regular  quaternion embeddings.

For $i\in \{1,2\}$, denote by $\sigma_i \in \Aut(K_i/F)$ the  nontrivial automorphism.
The automorphism group of the quartic $F$-algebra $K = K_1\otimes_F K_2$ contains the Klein four subgroup
\[
\{ \mathrm{id}, \tau_1,\tau_2,\tau_3\} \subset \Aut(K/F)
\]
with nontrivial elements 
\begin{align*}
(x \otimes y)^{\tau_1}  = x \otimes y^{\sigma_2} ,\quad
( x \otimes y) ^{\tau_2} =  x^{\sigma_1} \otimes y ,\quad 
(x \otimes y) ^{\tau_3}  =  x^{\sigma_1} \otimes y^{\sigma_2} .
\end{align*}
Denote by $K_3\subset K$ the subalgebra  of elements fixed by $\tau_3$. 
 Elementary algebra shows that  $K_1\iso K_2$ if and only if $K_3\iso F\oplus F$.
The picture, along with generators of the automorphism groups, is (\ref{biquad diagram}).

 Fix a  triple $(B,\alpha_1,\alpha_2) \in \mathcal{Q}(K_1,K_2)$, and define an $F$-linear map $\alpha : K \to B$ by 
\[
\alpha(x_1\otimes x_2) =\alpha_1(x_1) \alpha_2(x_2).
\]
  Clearly $\alpha|_{K_1} = \alpha_1$ and $\alpha|_{K_2}=\alpha_2$
are $F$-algebra homomorphisms, but $\alpha$ is an $F$-algebra homomorphism if and only if $\alpha_1(K_1)=\alpha_2(K_2)$.

\begin{lemma}\label{lem:regular lemma}
The image of $\alpha$ is the smallest $F$-subalgebra of $B$ containing both $\alpha_1(K_1)$ and $\alpha_2(K_2)$.  
In particular, $(B,\alpha_1,\alpha_2)$ is regular if and only if $\alpha(K)=B$.
\end{lemma}

\begin{proof}
It is clear from the definition that 
\[
\alpha(K)=\mathrm{Span}_F \{ \alpha_1(x)\alpha_2(y) : x\in K_1, y\in K_2 \},
\]
and so it suffices to prove that $\alpha(K)$ is a subalgebra of $B$.

If $\alpha(K)$ has dimension $2$, then $\alpha_1(K_1)=\alpha(K)=\alpha_2(K_2)$, and we are done.   If $\alpha(K)$ has dimension $4$, then $\alpha(K)=B$, and again we are done.  
Now suppose that $\alpha(K)$  has dimension $3$.    As $\alpha(K)$ is stable under left multiplication by $\alpha_1(K_1)$ and right multiplication by $\alpha_2(K_2)$,  these quadratic algebras  cannot be fields.  Thus each is  isomorphic to $F\times F$, and also $B\iso M_2(F)$.  This latter isomorphism may be chosen so that  $\alpha_1(K_1) \subset M_2(F)$ is the algebra of diagonal matrices.  This implies  that $\alpha(K) \subset M_2(F)$ is either the space of upper triangular matrices or the space of lower triangular matrices, as these are the only three dimensional subspaces of $M_2(F)$ stable under left multiplication by the diagonal subalgebra.  In either case $\alpha(K)$ is  a subalgebra of $B$.
\end{proof}

Denote by $b\mapsto b^\iota$ the main involution on $B$, and by $\mathrm{Nrd}_B(b)=bb^\iota$ and $\mathrm{Trd}_B(b)=b+b^\iota$ the reduced norm and reduced trace.

\begin{proposition}\label{prop:xi construct}
There is a unique $\xi \in K_3$ satisfying 
\begin{equation}\label{nrdquad}
 \mathrm{Nrd}_B( \alpha(x) )  = \Tr_{K_3/F} ( \xi x x^{\tau_3})  
\end{equation}
 for all $x\in K$.   It further satisfies $ \Tr_{K_3/F} ( \xi )=1$ and
\begin{equation}\label{trdquad}
\mathrm{Trd}_B \big(  \alpha(x)  \alpha(y)^\iota \big)    =   \Tr_{K /F}  (   \xi xy^{\tau_3} ) 
 \end{equation}
for all $x,y\in K$.  Moreover, $(B, \alpha_1,\alpha_2)$ is regular if and only if $\xi\in K_3^\times$.
 \end{proposition}

\begin{proof}
Endow $B$ with the structure of a $K$-module  via the action  
\[
(x_1 \otimes x_2) \action b = \alpha_1(x_1) \cdot b \cdot \alpha_2(x_2),
\]
 so that the $F$-linear map $\alpha : K \to B$ defined above is  $\alpha(x) = x\action 1$.

The action of $K$ on $B$ induces an action of $\mathbb{K}=K\otimes_F K_3$ on  $\mathbb{B}=B\otimes_F K_3$, and the orthogonal idempotents in  $\mathbb{K}\iso K\times K$ determines a splitting $\mathbb{B} = \mathbb{B}_+ \oplus \mathbb{B}_-$. More precisely, 
\begin{align*}
\mathbb{B}_+ & = \{ b \in \mathbb{B} : \forall x\in K_3,  x\action   b = b\cdot x  \} \\
\mathbb{B}_- & = \{ b \in \mathbb{B} : \forall x\in K_3, x\action  b =   b\cdot x^{\sigma_3}  \} .
\end{align*}
The projection map $\mathbb{B} \to \mathbb{B}_\pm$ is denoted $b\mapsto b_\pm$.   Each summand  is $K$-stable,  the projections are $K$-linear, and the decomposition is orthogonal with respect to the reduced norm $\mathrm{Nrd}_\mathbb{B} : \mathbb{B} \to K_3$.

By restricting the action of $K$, we view $B$ as a free $K_3$-module of rank $2$, and define a  quadratic form 
$
\mathrm{Nrd}_B^\dagger : B \to K_3
$
 by $\mathrm{Nrd}_B^\dagger (b) = \mathrm{Nrd}_\mathbb{B}( b_+)$.  We leave it as an exercise to the reader to verify the relations
\begin{align*}
\mathrm{Nrd}_B^\dagger( x\action b) & = x x^{\tau_3} \mathrm{Nrd}_B^\dagger(b)  \\
 \mathrm{Nrd}_B(b)&= \Tr_{K_3/F} ( \mathrm{Nrd}_B^\dagger(b) )  .
\end{align*}
for all $x\in K$ and $b\in B$.

Setting $\xi = \mathrm{Nrd}_B^\dagger (1)$, the first of these relations implies 
\[
 \xi xx^{\tau_3}    =  x x^{\tau_3} \mathrm{Nrd}_B^\dagger (1)   =  \mathrm{Nrd}_B^\dagger (x\action 1)    = \mathrm{Nrd}_B^\dagger ( \alpha(x)) ,
\]
and hence the second implies  
\[
\mathrm{Nrd}_B(\alpha(x) )     = \Tr_{K_3/F} ( \xi xx^{\tau_3} ).
\]
This proves the equality of quadratic forms (\ref{nrdquad}), which then implies the equality (\ref{trdquad}) of corresponding bilinear forms.
Taking $x=1$ in  (\ref{nrdquad})  shows that  $ \Tr_{K_3/F} ( \xi )=1$.

  If $\xi'\in K_3$ also satisfies (\ref{nrdquad}), then it also satisfies (\ref{trdquad}), and hence $\zeta=\xi-\xi'$ satisfies
$\Tr_{K /F}  (   \zeta xy ) =0$ for all $x,y\in K$.  Thus $\zeta=0$ and  $\xi=\xi'$.

Finally, it is easy to see  from Lemma \ref{lem:regular lemma} that each of the following statements is equivalent to the next.
\begin{itemize}
\item
The quaternion embedding $(B, \alpha_1,\alpha_2)$ is regular.
\item
The $F$-linear map $\alpha : K \to B$ is an isomorphism.
\item
The radical of the bilinear form (\ref{trdquad}) is trivial.
\item
The element $\xi\in K_3$ is a unit.
\end{itemize}
This completes the proof of Proposition \ref{prop:xi construct}.
\end{proof}

Proposition \ref{prop:xi construct} attaches to each $(B,\alpha_1,\alpha_2) \in \mathcal{Q}(K_1,K_2)$ an element $\xi\in K_3$.  
This defines the \emph{invariant}
\begin{equation}\label{basic invariant}
\inv : \mathcal{Q} (K_1,K_2) \to \{ \xi \in K_3 :  \Tr_{K_3/F}( \xi) =1 \},
\end{equation}
which satisfies
\[
 \mathcal{Q}_\reg(K_1,K_2) = \{ (B,\alpha_1,\alpha_2) \in \mathcal{Q}(K_1,K_2) : \inv(B,\alpha_1,\alpha_2) \in K_3^\times \}.
\]

\begin{theorem}\label{thm:the invariant}
The invariant (\ref{basic invariant}) restricts to a bijection
\begin{equation}\label{regular invariant}
\inv : \mathcal{Q}_\reg(K_1,K_2)  \iso \{ \xi \in K_3^\times :  \Tr_{K_3/F}( \xi) =1 \}.
\end{equation}
\end{theorem}

\begin{proof}
The strategy is to reduce to the case where $K_1$ and $K_2$ are split.

Abbreviate (\ref{regular subset})  to $\mathcal{Q}_\reg \subset \mathcal{Q}$.
Let $\mathbf{F}/F$ be a Galois extension large enough that both $\mathbf{F}$-algebras $\mathbf{K}_1=K_1\otimes_F\mathbf{F}$ and  $\mathbf{K}_2=K_2\otimes_F\mathbf{F}$ are isomorphic to $\mathbf{F}\times\mathbf{F}$.  
Applying $\otimes_F \mathbf{F}$ to the diagram (\ref{biquad diagram}) gives a new diagram of $\mathbf{F}$-algebras 
\begin{equation}\label{biquad extension}
\xymatrix{
&  { \mathbf{K}}  \\
{   \mathbf{K}_1 }  \ar@{-}[ur]^{  \tau_1}  &     {   \mathbf{K}_2   }  \ar@{-}[u]_{\tau_2} &  {  \mathbf{K}_3  } \ar@{-}[ul]_{\tau_3 }  \\
&  {   \mathbf{F}  , }  \ar@{-}[ul]^{\sigma_1}  \ar@{-}[u]_{\sigma_2} \ar@{-}[ur]_{\sigma_3}
}
\end{equation}
and we consider the set $\bm{\mathcal{Q}}$  of isomorphism classes of quaternion embeddings of $(\mathbf{K}_1,\mathbf{K}_2)$.  That is to say, isomorphism classes of triples $(\mathbf{B} , \alpha_1 ,\alpha_2)$ where $\mathbf{B}$ is a  quaternion algebra over $\mathbf{F}$,  and 
$\alpha_1 : \mathbf{K}_1 \to \mathbf{B}$ and  $\alpha_2 : \mathbf{K}_2 \to \mathbf{B}$ are $\mathbf{F}$-algebra embeddings.  Of course such embeddings can only exist if $\mathbf{B}\iso M_2(\mathbf{F})$.  
Again let  $\bm{\mathcal{Q}}_\reg \subset \bm{\mathcal{Q}}$ denote the subset of regular triples.

The Galois group $\Gal(\mathbf{F}/F)$ acts on the set $\bm{\mathcal{Q}}_\reg$ and on the rings  (\ref{biquad extension}) in a natural way, and the construction of (\ref{regular invariant})  defines a $\Gal(\mathbf{F}/F)$-invariant function 
\begin{equation}\label{separable invariant}
\inv : \bm{\mathcal{Q}}_\reg   \to \{ \xi \in \mathbf{K}_3^\times : \Tr_{ \mathbf{K}_3 / \mathbf{F}}(\xi) =1\} .
\end{equation}
Extension of scalars from $F$ to $\mathbf{F}$ defines the vertical arrows in the commutative diagram
\begin{equation}\label{invariant diagram}
\xymatrix{
{  \mathcal{Q}_\reg    } \ar[rrr]^{\inv} \ar[d] & & &  {     \{ \xi \in  K_3^\times : \Tr_{ K_3 / F}(\xi) =1\}    }  \ar[d] \\
{  \bm{ \mathcal{Q} }_\reg ^{\Gal(\mathbf{F}/F) }   }  \ar[rrr]^{\inv} & & & {    \{ \xi \in \mathbf{K}_3^\times : \Tr_{ \mathbf{K}_3 / \mathbf{F}}(\xi) =1\}^{\Gal(\mathbf{F}/F) } .  }      
}
\end{equation}

\begin{lemma}\label{lem:inv galois descent}
The function
$
\mathcal{Q}  _\reg  \to \bm{\mathcal{Q}}_\reg^{\Gal(\mathbf{F}/F) } 
$
is a bijection.
\end{lemma}

\begin{proof}
This is immediate from the theory of Galois descent.  The only thing to check is that a regular quaternion embedding $(\mathbf{B},\alpha_1,\alpha_2)\in \bm{\mathcal{Q}}_\reg$   has no nontrivial automorphisms.  Indeed, any $\mathbf{F}$-algebra automorphism of $\mathbf{B}$ is given by conjugation by some $b\in \mathbf{B}^\times$.  If $b$ defines an automorphism of the triple $(\mathbf{B},\alpha_1,\alpha_2)$, then $b$ centralizes both $\alpha_1(\mathbf{K}_1)$ and $\alpha_2(\mathbf{K}_2)$.  These subalgebras are equal to their own centralizers, and so  $b$ lies in the intersection  $\alpha_1( \mathbf{K}_1)^\times \cap \alpha_2( \mathbf{K}_2)^\times$.  The regularity of $(\mathbf{B},\alpha_1,\alpha_2)$  implies that this intersection is equal to $\mathbf{F}^\times$, and so conjugation by $b$ is trivial.
\end{proof}

\begin{lemma}\label{lem:inv closure}
The function (\ref{separable invariant}) is a bijection.
\end{lemma}

\begin{proof}
Fix isomorphisms $\mathbf{K}_i\iso \mathbf{F}\times \mathbf{F}$ for $i\in \{1,2,3\}$,  and an isomorphism 
$
\mathbf{K} \iso \mathbf{F}\times \mathbf{F}\times\mathbf{F}\times\mathbf{F},
$
in such a way that 
\begin{align*}
(x_1,x_2,x_3,x_4)^{\tau_1}  & = ( x_3, x_4 , x_1, x_2     ) \\
(x_1,x_2,x_3,x_4)^{\tau_2}  & = ( x_4, x_3, x_2, x_1     ) \\
(x_1,x_2,x_3,x_4)^{\tau_3}  & = ( x_2, x_1, x_4, x_3     ),
\end{align*}
and the inclusions of $\mathbf{K}_1 \iso \mathbf{K}_2\iso \mathbf{K}_3$ into $\mathbf{K}$ are identified (respectively) with
\[
(u,v) \mapsto (u,v,u,v),\quad (u,v)\mapsto (u,v,v,u), \quad (u,v)\mapsto (u,u,v,v).
\]

Define embeddings   $\alpha_1:\mathbf{K}_1 \to M_2(\mathbf{F}) $ and $\alpha_2: \mathbf{K}_2 \to M_2(\mathbf{F}) $ by 
\[
(u,v) \mapsto \left(\begin{matrix} u & \\  & v \end{matrix}\right).
\]
If we let $A\subset M_2(\mathbf{F})$ denote the subalgebra of diagonal matrices,  there is an induced bijection
\begin{equation}\label{split embedding param}
A^\times \backslash \GL_2(\mathbf{F}) / A^\times  \iso  \bm{\mathcal{Q}}
\end{equation}
sending 
$
\gamma= \left(\begin{smallmatrix}  a & b \\ c& d \end{smallmatrix}\right)
$
to the quaternion embedding  $( M_2(\mathbf{F}) ,   \alpha_1  , \gamma \alpha_2\gamma ^{-1} )$.
In particular, $\gamma$ determines an $\mathbf{F}$-linear map 
\[
 \mathbf{F} \times \mathbf{F} \times \mathbf{F} \times \mathbf{F}   = \mathbf{K} \map{\alpha} M_2(\mathbf{F})
\] 
and an element $\inv(\gamma)  \in  \mathbf{K}_3$  of trace $1$.  Direct calculation shows that 
\[
\alpha(x_1, x_2,x_3,x_4)  
 =   \frac{1}{ad-bc}  \left(\begin{matrix}  adx_1 -bcx_3 &  ab (x_3 - x_1) \\   cd(x_4-x_2)    &    ad x_2-bc x_4 \end{matrix}\right) 
\]
and 
\[
\inv(\gamma) = \left( \frac{ad}{ad-bc}  , \frac{-bc}{ad-bc} \right) \in \mathbf{F}\times\mathbf{F} = \mathbf{K}_3.
\]

In \cite[\S 1.3]{Jacquet}  one can find a complete set of representatives for the double coset space (\ref{split embedding param}).
There are six non-regular elements, represented by 
\[
\left( \begin{smallmatrix} 1 & \\ & 1 \end{smallmatrix}\right),\quad 
\left( \begin{smallmatrix}  & 1 \\ 1&  \end{smallmatrix}\right),\quad 
\left( \begin{smallmatrix} 1 &  1\\ & 1 \end{smallmatrix}\right),\quad 
\left( \begin{smallmatrix} 1 & \\ 1& 1 \end{smallmatrix}\right),\quad 
\left( \begin{smallmatrix}  &  1  \\ 1 & 1 \end{smallmatrix}\right),\quad 
\left( \begin{smallmatrix} 1 & 1\\ 1 &  \end{smallmatrix}\right). 
\]
(Compare with the proof of Lemma \ref{lem:regular lemma}. For the first two double cosets the corresponding map $\alpha:\mathbf{K}\to M_2(\mathbf{F})$ has image of dimension two.   For the remaining four double cosets the image has dimension three.)  
A complete set of  representatives for the regular  elements is given by
\[
\left\{ \left(\begin{matrix}  1 & x \\ 1 & 1  \end{matrix}\right) : x\in \mathbf{F}\smallsetminus \{ 0,1\} \right\}
\iso \bm{\mathcal{Q}}_\reg .
\]
For this set of coset representatives,  the function 
\[
\inv: \bm{\mathcal{Q}}_\reg \to \{ ( u , v)  \in \mathbf{F}^\times \times\mathbf{F}^\times :  u+v=1\}
\]
of (\ref{separable invariant})  takes the explicit form
\[
\left(\begin{matrix}  1 & x \\ 1 & 1  \end{matrix}\right) \mapsto  \left( \frac{1}{1-x}  , \frac{-x}{1-x} \right) ,
\]
and the reader will have no difficulty in checking that this is a bijection.
\end{proof}

The bijectivity of   (\ref{regular invariant})   is clear from the  two lemmas above.  
Indeed, the left vertical arrow in (\ref{invariant diagram}) is  bijective by Lemma \ref{lem:inv galois descent}, the right vertical arrow is obviously bijective, and the bottom horizontal arrow is bijective by Lemma \ref{lem:inv closure}.  
Therefore the top horizontal arrow is a bijection.  This completes the proof of Theorem \ref{thm:the invariant}.
\end{proof}

\begin{remark}
Although we will not explicitly need to do so, one can understand Theorem \ref{thm:the invariant} as a parametrization of certain double coset spaces, in the spirit of \cite{Jacquet}.  Call a quaternion algebra $B$ \emph{relevant} if it admits embeddings of both $K_1$ and $K_2$.  For each relevant $B$, fix such embeddings once and for all. These  choices determine bijections
\[
B^\times / K_i^\times \iso \{ \mbox{embeddings } K_i \to B\}
\]
 and 
\[
\bigsqcup_{ \mathrm{relevant\, }B}  K_1^\times  \backslash B^\times / K_2^\times \iso 
\bigsqcup_{ \mathrm{relevant\, }B}  B^\times \backslash \big(  B^\times / K_1^\times \times B^\times/K_2^\times \big)
\iso \mathcal{Q}(K_1,K_2).
\]
With these identifications, the invariant (\ref{basic invariant}) determines a function
\[
\bigsqcup_{ \mathrm{relevant\, }B}  K_1^\times  \backslash B^\times / K_2^\times \to    \{ \xi \in K_3 : \Tr_{K_3/F}(\xi)=1 \} .
\]
The content of Theorem \ref{thm:the invariant} is that this is very nearly a bijection.  In fact, if $K_1\not\iso K_2$ are both fields, then $K_3$ is a field, $\mathcal{Q}(K_1,K_2)=\mathcal{Q}_\reg(K_1,K_2)$, and  Theorem \ref{thm:the invariant}  asserts the bijectivity of this function.
\end{remark}


\subsection{Invariants of double cosets}
\label{ss:coset invariants}


Let $\mathrm{Iso}(K_2,K_1)$ be the set of isomorphisms of $F$-vector spaces $K_2\to K_1$.  
The group $K_1^\times$ acts on the left by postcomposition, and the group $K_2^\times$ acts on the right by precomposition.  
We use similar notation if $K_1$ and $K_2$ are replaced by other $F$-algebras.  
Our goal is to define a function
\begin{equation}\label{new invariant}
K_1^\times \backslash \mathrm{Iso}(K_2,K_1)  /K_2^\times   \map{\inv}    \{ \xi \in K_3 : \Tr_{K_3/F}(\xi)=1 \} ,
\end{equation}
intimately related to the invariant (\ref{basic invariant}).

There are canonical $F$-algebra isomorphisms 
 \[
 K_1\otimes_F K \iso  K\oplus K ,\quad    K_2\otimes_F K \iso K\oplus K ,
 \]
the first defined by $x\otimes y \mapsto (xy,x^{\sigma_1}y)$,  and the second defined similarly.  
 Any $\phi \in \Hom(K_2,K_1)$ therefore induces a $K$-linear map 
 \[
 K\oplus K \iso K_2\otimes_F K \map{\phi \otimes \mathrm{id} } K_1\otimes_F K \iso  K\oplus K,
 \]
represented by a matrix $\left(\begin{smallmatrix} a & b \\ c & d \end{smallmatrix}\right)\in M_2(K)$ whose entries are related by
\begin{equation}\label{abcd}
b=a^{\tau_1},\quad c=a^{\tau_2},\quad d=a^{\tau_3}.
\end{equation}

Define a one-dimensional $F$-vector space  \[ \Delta= \Hom_F(\det\nolimits_F(K_2),\det\nolimits_F(K_1)),\]
and define 
$
\newdet(\phi) \in \Delta \otimes_F K 
$
by
\[
\xymatrix{
{  \det_F(K_2) \otimes_F K} \ar@{=}[r]  \ar[d]_{  \newdet(\phi) } & {  \det_K(K_2\otimes_F K) }   \ar@{=}[r]   &  {   \det_K(K\oplus K) } \ar[d]^{ ad } \\
{  \det_F(K_1) \otimes_F K}  \ar@{=}[r] & {  \det_K(K_1\otimes_F K) }   \ar@{=}[r] &  {   \det_K(K\oplus K) } 
}
\] 
where   $ad$ is the map  $e_1\wedge e_2 \mapsto (a e_1) \wedge (d e_2)$.

  Using the  relations (\ref{abcd}), one can check that  $\newdet(\phi)$ actually lies in  
$\Delta\otimes_F K_3$, and its trace to $\Delta$   is the usual determinant  $\det(\phi) \in \Delta$.
In other words, we have  a commutative diagram
\[
\xymatrix{
{  \Hom_F(K_2,K_1)  } \ar[drr]_{ \det} \ar[rr]^{ \newdet}  & &   {   \Delta\otimes_F K_3   } \ar[d]^{\Tr_{K_3/F} }\\
&  &  {  \Delta,} 
}
\]
and the desired function (\ref{new invariant}) is  defined by
\[
\inv(\phi) =  \newdet(\phi)  / \det(\phi)  \in \{ \xi \in K_3 : \Tr_{K_3/F}(\xi)=1\}.
\]

\begin{proposition}
If we define
\[
\mathrm{Iso}_\reg(K_2,K_1)   = \{ \phi\in \mathrm{Iso}(K_2,K_1) : \inv(\phi)\in K_3^\times \},
\]
the function (\ref{new invariant}) restricts to an injection
\[
K_1^\times \backslash \mathrm{Iso}_\reg(K_2,K_1)  /K_2^\times       \map{\inv}    \{ \xi \in K_3^\times : \Tr_{K_3/F}(\xi)=1 \} .
\]
\end{proposition}

\begin{proof}
Any pair 
\[
( a_1,a_2) \in \mathrm{Iso}(K_1, F\oplus F ) \times \mathrm{Iso}(K_2, F\oplus F)
\] 
determines  an $F$-linear isomorphism $a^{-1}_1\circ a_2 : K_2\to K_1$,  as well as $F$-algebra embeddings  
$\alpha_1 : K_1 \to M_2(F)$ and  $\alpha_2:K_2\to M_2(F)$.
  These constructions induce  canonical identifications
\begin{eqnarray}\lefteqn{
K_1^\times \backslash \mathrm{Iso}(K_2,K_1)  / K_2^\times  }  \label{coset to quaternion}  \\
 & \iso &
  \GL_2(F) \backslash  \big(  \mathrm{Iso}(K_1, F\oplus F) /K_1^\times \times \mathrm{Iso}(K_2, F\oplus F)/K_2^\times \big)  \nonumber \\
 & \iso &
 \GL_2(F) \backslash    \left\{  \begin{array}{c}  \mbox{pairs of embeddings }  \\     \alpha_1: K_1 \to M_2(F)  \\   \alpha_2 :K_2 \to M_2(F)  \end{array} \right\} , \nonumber
\end{eqnarray}
which realize 
\[
K_1^\times \backslash \mathrm{Iso} (K_2,K_1)  /K_2^\times   \subset \mathcal{Q}(K_1,K_2)
\]
as the set of all quaternion embeddings whose underlying quaternion algebra is  $M_2(F)$.  A tedious but elementary calculation shows that 
the invariant (\ref{new invariant}) is simply the restriction of (\ref{basic invariant})  to this subset,  and so the claim follows from the bijectivity of (\ref{regular invariant}).
\end{proof}


\subsection{The dual picture}
\label{ss:split coset invariant}


In this subsection and the next, we let 
\[
K_0=F\oplus F.
\]
The initial input to the constructions of \S \ref{ss:invariants} and \S \ref{ss:coset invariants} was a pair of quadratic algebras $K_1$ and $K_2$, 
from which we produced a diagram  (\ref{biquad diagram}).

We may repeat these constructions, but now take the initial input to be $K_0$ and $K_3$.  In this case $K_0\otimes_F K_3\iso K_3\times K_3$, and 
 the  diagram (\ref{biquad diagram}) is replaced by
\[
\xymatrix{
&  { K_3\times  K_3 }  \\
{   K_0 }  \ar@{-}[ur]  &     {   K_3   }  \ar@{-}[u] &  {  K_3  } \ar@{-}[ul]  \\
&  {   F  . }  \ar@{-}[ul]  \ar@{-}[u] \ar@{-}[ur]
}
\]
Here the upper left inclusion is just $K_0=F\times F\subset K_3\times K_3$, the middle inclusion $K_3\to K_3\times K_3$ is the diagonal embedding, and the inclusion $K_3\to K_3\times K_3$ on the right is  the twisted diagonal $y_3\mapsto (y_3,y_3^{\sigma_3})$.

Repeating the construction of (\ref{new invariant}) in this new setting yields a function
\begin{equation}\label{new split invariant}
K_0^\times \backslash \mathrm{Iso}(K_3,K_0)  /K_3^\times   \map{\inv}    \{ \xi \in K_3 : \Tr_{K_3/F}(\xi)=1 \} .
\end{equation}
In fact, the construction  of this function simplifies slightly because  $K_0$ is split.   For comparison with later constructions (see especially the proof of Lemma \ref{lem:split fiber bijection}) we now make this completely explicit.

There are canonical isomorphisms of $F$-algebras
\[
K_0\otimes_F K_3\iso K_3\oplus K_3,\quad  K_3\otimes_F K_3 \iso K_3\oplus K_3.
\]
The first  is the $K_3$-linear extension of $K_0=F\oplus F \subset K_3\oplus K_3$, and the second is  
$x\otimes y\mapsto (xy, x^{\sigma_3}y)$.
Any $\phi \in \Hom_F(K_3,K_0)$ induces a $K_3$-linear map
\[
K_3\oplus K_3 \iso K_3\otimes_F K_3 \map{\phi \otimes \mathrm{id} } K_0 \otimes_F K_3 \iso K_3\oplus K_3
\]
represented by a matrix $\left(\begin{smallmatrix} a & b \\ c & d \end{smallmatrix}\right)\in M_2(K_3)$ satisfying
$b=a^{\sigma_3}$ and $c=d^{\sigma_3}$.

Define a one-dimensional $F$-vector space  \[ \Delta= \Hom_F(\det\nolimits_F(K_3),\det\nolimits_F(K_0)),\]
and define
$
\newdet(\phi) \in \Delta \otimes_F K_3 
$
by
\[
\xymatrix{
{  \det_F(K_3) \otimes_F K_3} \ar@{=}[r]  \ar[d]_{  \newdet(\phi) } & {  \det_{K_3}(K_3\otimes_F K_3) }   \ar@{=}[r]   &  {   \det_{K_3}(K_3\oplus K_3) } \ar[d]^{ad} \\
{  \det_F(K_0) \otimes_F K_3}  \ar@{=}[r] & {  \det_{K_3}(K_0\otimes_F K_3) }   \ar@{=}[r] &  {   \det_{K_3}(K_3\oplus K_3) .} 
}
\] 
Here $ad$ denotes the map $e_1\wedge e_2\mapsto (ae_1)\wedge (d e_2)$.  Exactly as before, we have  a commutative diagram
\[
\xymatrix{
{  \Hom_F(K_3,K_0)  } \ar[drr]_{ \det} \ar[rr]^{ \newdet}  & &   {   \Delta\otimes_F K_3   } \ar[d]^{\Tr_{K_3/F} }\\
&  &  {  \Delta,} 
}
\]
and  (\ref{new split invariant}) is  given  by $\inv(\phi) =  \newdet(\phi)  / \det(\phi)$.

\begin{proposition}\label{prop:split coset regular}
If we define
\[
\mathrm{Iso}_\reg(K_3,K_0)   = \{ \phi\in \mathrm{Iso}(K_3,K_0) : \inv(\phi)\in K_3^\times \},
\]
the function (\ref{new split invariant}) restricts to a bijection
\[
K_0^\times \backslash \mathrm{Iso}_\reg(K_3,K_0)  /K_3^\times       \map{\inv}    \{ \xi \in K_3^\times : \Tr_{K_3/F}(\xi)=1 \} .
\]
\end{proposition}

\begin{proof}
Exactly as in \S \ref{ss:invariants}, denote by $\mathcal{Q}(K_0, K_3)$ the set of isomorphism classes of quaternion embeddings $(B,\alpha_0,\alpha_3)$.  The subset of regular pairs is again denoted by   $\mathcal{Q}_\reg(K_0, K_3) \subset \mathcal{Q}(K_0, K_3).$

 Exactly as in   (\ref{coset to quaternion}), there is a  canonical bijection
\[
K_0^\times \backslash \mathrm{Iso}(K_3,K_0)  / K_3^\times    \iso 
 \GL_2(F) \backslash    \left\{  \begin{array}{c}  \mbox{pairs of embeddings }  \\     \alpha_0: K_0 \to M_2(F)  \\   \alpha_3 :K_3 \to M_2(F)  \end{array} \right\} .
 \]
 As $K_0$ is split, any quaternion embedding $(B,\alpha_0,\alpha_3)$ must have $B\iso M_2(F)$.  Thus we obtain a canonical bijection
 \[
 K_0^\times \backslash \mathrm{Iso}(K_3,K_0)  / K_3^\times    \iso  \mathcal{Q}(K_0,K_3).
 \]
  Applying the construction of (\ref{basic invariant}) with the triple of algebras $(K_1,K_2,K_3)$ replaced  by $(K_0, K_3,K_3)$ yields a function
 \[
\mathcal{Q}(K_0,K_3) \map{\inv}  \{ \xi \in K_3^\times : \Tr_{K_3/F}(\xi)=1 \} ,
 \]
 which, as a tedious but elementary calculation shows,  agrees with (\ref{new split invariant}) under the above identification.
 Thus the claim follows from the bijectivity of (\ref{regular invariant}).
 \end{proof}


\section{The analytic calculation}
\label{s:analytic}


For the remainder of the paper we return to the situation of the introduction, so that $F=\kk(X)$ and all $F$-algebras appearing in (\ref{biquad diagram}) are Galois field extensions of $F$.


\subsection{Automorphic forms}
\label{ss:basic automorphic}


Denote by $\mathcal{A}(G_0)$ the space of automorphic forms \cite[\S 5]{Borel-Jacquet} on $G_0(\A)$, and by 
$
\mathcal{A}_\cusp(G_0) \subset \mathcal{A}(G_0)
$
the subspace of cuspidal automorphic forms.  
The subspace of unramified (that is, $U_0$-invariant) cuspforms is finite-dimensional, and admits a decomposition
\[
\mathcal{A}_\cusp(G_0)^{U_0} = \bigoplus_{\mathrm{unr.\, cusp.\, }\pi } \pi^{U_0}
\]
as a direct sum of lines, where the sum is over the  unramified cuspidal automorphic representations $\pi \subset \mathcal{A}_\cusp(G_0)$.

Denote by $\mathscr{H}$ the (commutative) Hecke algebra of all compactly supported functions $f: U_0 \backslash G_0(\A) /U_0 \to \Q$.
The $\mathscr{H}$-module of compactly supported unramified $\Q$-valued automorphic forms is denoted
\[
\mathscr{A} = C_c^\infty(  G_0(F) \backslash G_0(\A) / U_0 ,\Q ).
\]
We let $\mathscr{A}_\C = \mathscr{A} \otimes \C$ denote  the corresponding complex space, so that 
\begin{equation}\label{harder}
\mathcal{A}_\cusp(G_0)^{U_0} \subset \mathscr{A}_\C \subset \mathcal{A} (G_0)^{U_0}.
\end{equation}
Note that the first inclusion follows from Harder's theorem \cite[Proposition 5.2]{Borel-Jacquet} that every cuspidal automorphic form on $G_0(\A)$ is compactly supported.

According to  \cite[\S 4.1]{YZ}, the Satake transform induces a canonical $\Q$-algebra surjection
$
a_{\mathrm{Eis}} : \mathscr{H} \to \Q  [\Pic_X(\kk)]^{\iota_\Pic},
$ 
for a particular involution $\iota_\Pic$ of $ \Q [\Pic_X(\kk)]$.   The Eisenstein ideal  is defined by
\begin{equation}\label{eisenstein def}
\mathcal{I}^{\mathrm{Eis}} = \ker \big( a_\Eis  :  \mathscr{H} \to \Q  [\Pic_X(\kk)]^{\iota_\Pic}  \big).
\end{equation}

As in \cite[\S 7.3]{YZ}, define $\Q$-algebras
\begin{align*}
\mathscr{H}_\mathrm{aut} & = \mathrm{Image} \big( \mathscr{H}  \to \End_\Q (\mathscr{A}) \times \Q[\Pic_X(\kk)]^{\iota_\Pic} \big) \\
\mathscr{H}_\cusp &= \mathrm{Image}  \left(  \mathscr{H} \to \End_\C(  \mathcal{A}_\cusp(G_0)^{U_0} ) \right).
\end{align*}
It follows from  (\ref{harder}) that the quotient map   $\mathscr{H}\to \mathscr{H}_\cusp$  factors through  $\mathscr{H}_\mathrm{aut}$,
and the resulting map
\begin{equation}\label{automorphic decomp}
\mathscr{H}_\mathrm{aut} \to \mathscr{H}_\cusp \times \Q  [\Pic_X(\kk)]^{\iota_\Pic} 
\end{equation}
is an isomorphism by \cite[Lemma 7.16]{YZ}.

For each unramified cuspidal automorphic representation $\pi \subset \mathcal{A}_\cusp(G_0)$, denote by
\[
\lambda_\pi : \mathscr{H} \to \C
\] 
the character through which the Hecke algebra acts on the line $\pi^{U_0}$. 
 As in \cite[\S 7.5.1]{YZ}, the $\Q$-algebra $\mathscr{H}_\cusp$
is isomorphic to a finite product of number fields, and the product of all characters $\lambda_\pi$ induces an isomorphism
\[
 \mathscr{H}_\cusp \otimes \C  \iso \bigoplus_{    \mathrm{unr.\, cusp.\, }\pi   }  \C.
\] 

\begin{remark}\label{rem:auto galois}
There is an action of $\Aut(\C/\Q)$ on the finite set of unramified cuspidal automorphic representations $\pi \subset \mathcal{A}_\cusp(G_0)$, characterized by the relation
$
\lambda_{\pi^\sigma} = \sigma \circ \lambda_\pi.
$
\end{remark}

\begin{remark}\label{rem:real valued}
The Petersson inner product identifies the contragredient of $\pi \subset \mathcal{A}_\cusp(G_0)$ with the space $\overline{\pi}$ of complex conjugate functions.  As $\pi$ has trivial central character, multiplicity one implies that $\pi = \pi^\vee = \overline{\pi}$.  From this it is easy to see first that each character $\lambda_\pi : \pi \to \C$ is real-valued, and then that  $\mathscr{H}_\cusp$ is isomorphic to a product of totally real number fields.
\end{remark}

All of the above was for the group scheme $G_0\iso \PGL_2$ over $X$, but the same discussion holds word-for-word if  $G_0$ is replaced by  $G_1$, $G_2$, or $G_3$.

\begin{lemma}\label{lem:naive transfer}
For any $i\in \{1,2,3\}$ there is a canonical bijection
\[
G_0(F) \backslash G_0(\A) / U_0 \to G_i(F) \backslash G_i(\A) / U_i.
\]
It induces an isomorphism of $\C$-vector spaces
\[
\mathcal{A}(G_0)^{U_0} \iso \mathcal{A}(G_i)^{U_i}
\]
respecting the subspaces of cusp forms.
\end{lemma}

\begin{proof}
Fix an isomorphism $\rho : K_0 \to K_i$ of $F$-vector spaces.  The induced isomorphism $\rho: \A_0 \to \A_i$ satisfies 
$\rho(\mathbb{O}_0) = h \mathbb{O}_i$ for some $h \in \tilde{G}_i(\A)$. 
The choice of $\rho$ also determines an isomorphism  $\rho : G_{0}(F) \to G_{i}(F)$  by 
$
  \rho(g)= \rho \circ g \circ \rho^{-1},
$
and the desired bijection is
\[
G_0(F) \backslash G_0(\A) / U_0 \map{ g \mapsto  \rho(g) h} G_i(F) \backslash G_i(\A) / U_i.
\]
This is easily seen to be  independent of the choices of $\rho$ and $h$.
\end{proof}


\subsection{The analytic distribution}
\label{ss:orbital notation}


The $X$-scheme
\[
\tilde{J} = \underline{\mathrm{Iso}}_{\co_X}( f_{3*} \co_{Y_3} ,  f_{0*} \co_{Y_0} ) 
\]
is both a left $\tilde{G}_0$-torsor and a right $\tilde{G}_3$-torsor, and similarly $
J=\tilde{J}  /\G_m$ is both a left $G_0$-torsor and a right $G_3$-torsor.
There are canonical identifications
\[
T_0(F) \backslash J(F) / T_3(F) = \tilde{T}_0(F) \backslash \tilde{J}(F) / \tilde{T}_3(F) = K_0^\times \backslash \mathrm{Iso}( K_3 , K_0 ) / K_3^\times,
\]
and so (\ref{new split invariant}) defines a bijective function
\begin{equation}\label{orbit invariant}
T_0(F) \backslash J(F) / T_3(F) \map{\inv}  \{ \xi \in K_3 : \Tr_{K_3/F}(\xi)=1 \}.
\end{equation}
The bijectivity follows from Proposition \ref{prop:split coset regular}, as $K_3$ is now a field.

\begin{lemma}\label{lem:coset transfer}
There is a canonical homeomorphism
\begin{equation}\label{coset switch}
 U_0 \backslash J(\A) / U_3 \iso U_0 \backslash G_0(\A) /U_0.
\end{equation}
\end{lemma}

\begin{proof}
As $\mathbb{O}_0$ and $\mathbb{O}_3$ are both free $\mathbb{O}$-modules of rank two, we may choose an 
 isomorphism of $\A$-modules  $a: \A_0 \to \A_3$ in such a way that $a(\mathbb{O}_0)=\mathbb{O}_3$.   This determines a homeomorphism
\[
\tilde{J}(\A)  = \mathrm{Iso} ( \A _3, \A_0)  \map{\phi \mapsto \phi \circ a} \mathrm{Iso}( \A_0 , \A_0) = \tilde{G}_0(\A).
\]
 It is easy to check that this descends to a homeomorphism (\ref{coset switch}),  which is independent of the choice of $a$. 
\end{proof}

Now fix $f\in \mathscr{H}$.  Slightly switching the point of view, we use the bijection of Lemma \ref{lem:coset transfer} to instead view 
$
f : U_0 \backslash J(\A) / U_3 \to \Q,
$
and  define a function on $G_0(\A) \times G_3(\A)$ by
 \begin{equation}\label{the kernel}
\K_f(g_0,g_3) = \sum_{\gamma \in J(F)} f(g_0^{-1}\gamma g_3).
\end{equation}
  
 \begin{remark}\label{rem:usual kernel}
Using Lemma  \ref{lem:naive transfer}, we may instead view $\K_f$ as a function  on $G_0(\A) \times G_0(\A)$.
 This agrees with the  kernel function defined in  \cite[(2.3)]{YZ}.
 \end{remark}
 
 We adopt the usual notation 
\[
[ T_i ] = T_i (F)\backslash T_i(\A) ,\quad 
[G_i]=G_i(F) \backslash G_i(\A) ,
\]
and recall the normalization of Haar measures of \S \ref{ss:notation}.
Define a distribution on $\mathscr{H}$ by 
\begin{equation}\label{J distribution def}
\J(f,s) = \int^\reg_{[T_0] \times [T_3]} \K_f(t_0, t_3)\,  |t_0|^{2s}  \eta(t_3) \, dt_0 \, dt_3.
\end{equation}
We will explain what this regularized integral means momentarily.  For the other notation, recalling that  $T_0 \subset G_0=\PGL_2$ is the diagonal torus,  let  \[ |\cdot| : T_0(\A) \to \R^\times\] be the homomorphism  $\left| \left(\begin{smallmatrix} a & \\ & d \end{smallmatrix} \right)\right| = |a/d|$.  The character $\eta$  is defined by the following lemma.

\begin{lemma}\label{lem:eta}
The quadratic character 
$
\eta: \A_3^\times \to \{\pm{1}\}
$ 
determined by $Y/Y_3$ factors through the quotient $T_3(\A) = \A_3^\times / \A^\times$, and hence defines a character 
\[
\eta : [T_3] \to \{\pm 1\}.
\]
\end{lemma}

\begin{proof}
Denote by 
$
\chi_i : \A^\times \to \{\pm 1\}
$ 
 the quadratic character determined by $Y_i/X$.  An exercise in class field theory shows that 
\[
\chi_1(\Nm(x)) = \eta(x) = \chi_2(\Nm(x))
\] 
for all $x\in \A_3^\times$, where $\Nm:\A_3^\times \to \A^\times$ is the norm. The claim is immediate from this,
and the fact  that  $\Nm(t)=t^2$ for all  $t\in \A^\times$.
\end{proof}

The integral in (\ref{J distribution def}) need not converge absolutely, and we now explain how it is regularized.
First define 
\[
T_0(\A)_n = \left\{t_0 \in T_0(\A) :   |t_0| = q^{-n} \right\} 
\]
and $[T_0]_n = T_0(F) \backslash T_0(\A)_n$, and set
\begin{align}\label{partial distribution}
\J_n(f,s) & = \int_{[T_0]_n \times [T_3]} \K_f(t_0, t_3) |t_0|^{2s} \eta(t_3)\,  dt_0 \, dt_3 \\
& =  q^{-2ns} \int_{[T_0]_n \times [T_3]} \K_f(t_0, t_3)  \eta(t_3)\,  dt_0 \, dt_3 \nonumber
\end{align}
This integral is absolutely convergent. Indeed, the set $[T_0]_n$ is compact, and the finiteness of 
\[
 T_3(F) \backslash T_3(\A) /(U_3 \cap T_3(\A)) \iso K_3^\times \backslash \A_3^\times /\A^\times \mathbb{O}_3^\times \iso \Pic(Y_3) /f_3^*\Pic(X)
\]
implies that $[T_3]$ is also compact.

\begin{proposition}\label{prop:convergence}
The integral $\J_n(f,s)$ vanishes for $|n|$ sufficiently large.  
\end{proposition}

\begin{proof}
For any $\gamma \in  J(F) $, and for any $g_0\in [G_0]$ and $g_3\in [G_3]$,  we set 
\begin{equation}\label{coset kernel} 
\K_{f,\gamma}(g_0,g_3) = \sum_{\delta \in  T_0(F) \gamma T_3(F) } f(g_0^{-1}\delta g_3)  
\end{equation}
so that 
\begin{equation*}
\K_f(g_0,g_3)   = \sum_{\gamma \in T_0(F)\backslash J(F) / T_3(F)} \K_{f,\gamma}(g_0,g_3) .
\end{equation*}
We also set 
\begin{equation} \label{trunc J}
\J_n(\gamma, f,s)  = \int_{[T_0]_n \times [T_3]} \K_{f,\gamma}(t_0, t_3) |t_0|^{2s} \eta(t_3)\,  dt_0 \, dt_3,
\end{equation}
so that 
\begin{equation*}
\J_n(f,s)   = \sum_{\gamma \in T_0(F)\backslash J(F) / T_3(F)} \J_n(\gamma, f, s).
\end{equation*}

\begin{lemma}\label{finiteness}
For all but finitely many  $\gamma \in T_0(F) \backslash J(F)/  T_3(F)$, the function 
(\ref{coset kernel}) vanishes identically on $[T_0] \times [T_3]$.
 \end{lemma}

\begin{proof}
The function (\ref{orbit invariant}) extends in a natural way to a continuous function on adelic points 
\[
\inv :   T_0(\A) \backslash J(\A) / T_3(\A)  \to \{ \xi \in \A_3 : \Tr_{K_3/F} (\xi)=1 \}.
\]
Indeed, the function (\ref{orbit invariant}) was defined using the constructions of \S \ref{ss:split coset invariant}, which can be applied 
locally at every place of $F$.   (Of course,  locally $K_3$ need not be a field, and so this function need not be a bijection).

Let $C\subset \A_3$ be the image under $\inv: J(\A) \to \A_3$ of the support of $f$.  This is a compact set, and so has 
finite intersection with the (discrete) image of the injection
\[
 T_0(F) \backslash J(F) / T_3(F) \map{\inv} K_3 \subset \A_3.
\]
Thus there are only finitely many $\gamma \in T_0(F) \backslash J(F)/  T_3(F)$ with $\inv(\gamma) \in C$.

It is clear from the definitions that if $\inv(\gamma) \not\in C$ then (\ref{coset kernel}) vanishes identically on $[T_0]\times [T_3]$, proving the claim.
\end{proof}

 To unfold the integral (\ref{trunc J}), we need to understand the stabilizer
\begin{equation}\label{S def}
S_\gamma = \{ (t_0,t_3) : t_0^{-1} \gamma t_3 = \gamma \} \subset T_{0F}\times T_{3F},
\end{equation}
which is an algebraic group over $F$.

\begin{lemma}\label{lem:S trivial}
The morphism $T_{0F}\times T_{3F} \to J_F$ given by $(t_0, t_3) \mapsto t_0^{-1}\gamma t_3$ is a closed immersion. In particular, $S_\gamma$ is trivial.
\end{lemma}

\begin{proof}
To prove that this is a closed immersion, we may base change to $K_3$, over which $T_3$ becomes isomorphic to $T_0$.  Then it is enough to check that the orbit of $\gamma$ is regular semisimple in the sense of \cite{YZ}, or in other words, that the invariant is not $(1,0)$ or $(0,1)$, when viewed as trace-one elements of $K_3 \otimes_F K_3 \simeq K_3 \times K_3$.  But $\inv(\gamma)$ is a unit in $K_3$, so it cannot be $(1,0)$ or $(0,1)$.  
\end{proof}

 \begin{lemma}\label{lem:finiteness 2}
For fixed  $\gamma$ and $f$ as above, the integral (\ref{trunc J}) vanishes  for all but finitely many $n$.
 \end{lemma}
\begin{proof}
Since $S_\gamma$ is trivial, we may unfold the integral (\ref{trunc J}) to obtain
\[
\J_n(\gamma, f, s) = \int_{ T_0(\A)_n \times T_3(\A)}  f(t_0^{-1}\gamma t_3) \,  |t_0|^{2s} \eta(t_3) \, dt_0 \, dt_3.
\]
The map $i \colon T_0(\A) \times T_3(\A) \to J(\A)$ given by $(t_0, t_3) \mapsto t_0^{-1}\gamma t_3$ is a closed embedding, so $f \circ i$ has  compact support.  It follows that $\J_n(\gamma, f,s)$ vanishes for $|n|$ large enough.   
\end{proof}

%
%
%

We now complete the proof of Proposition \ref{prop:convergence}.     By Lemma \ref{finiteness}, 
\[
\J_n(f,s)   = \sum_{\gamma \in \Gamma_f}  \J_n(\gamma,f,s) 
\]
for some finite subset $\Gamma_f \subset T_0(F) \backslash J(F) /T_3(F)$ independent of $n$.
 Lemma \ref{lem:finiteness 2} now implies that $\J_n(f,s)$ vanishes  for $|n|\gg 0$.
\end{proof}

Using Proposition \ref{prop:convergence},  the regularized integral  (\ref{J distribution def}) is defined as
\[
\J(f,s) = \sum_{n \in \Z} \J_n(f,s).
\] 
This is a Laurent polynomial in $q^s$.
Recalling (\ref{trunc J}), we also define
\[
\J(\gamma,f,s) =\sum_{n\in \Z} \J_n(\gamma,f,s),
\]
 so that there are decompositions
\begin{equation}\label{J decomp}
\J(f,s)   =   \sum_{\gamma \in T_0(F) \backslash J(F)/T_3(F)}  \J(\gamma, f,s)  = \sum_{\substack{ \xi\in K_3 \\ \Tr_{K_3/F}(\xi) =1  }} \J( \xi ,f,s).
\end{equation}
Here, in the final expression,  we use (\ref{orbit invariant}) to define
\[
\J(\xi,f,s)= \J(\gamma,f,s)
\] 
for the unique double coset $\gamma \in T_0(F) \backslash J(F)/T_3(F)$ satisfying $\inv(\gamma)=\xi$.


\subsection{Spectral decomposition}
\label{ss:spectral}


Using the normalization of Haar measures of \S \ref{ss:notation}, 
for any $\phi \in \mathcal{A}_\cusp(G_0)^{U_0}$ we define a period integral
\[
\mathscr{P}_0(\phi , s ) = \int_{[T_0]} \phi(t_0) |t_0|^{2s} \, dt_0.
\]
This integral is absolutely convergent for all $s\in \C$.
Using the isomorphism of Lemma \ref{lem:coset transfer} to view $\phi \in \mathcal{A}_\cusp(G_3)^{U_3}$, we define another period integral
\[
\mathscr{P}_3(\phi , \eta ) = \int_{[T_3]} \phi(t_3) \eta(t_3) \, dt_3.
\]
As $[T_3]$ is compact, this integral is also absolutely convergent.

Recall the Eisenstein ideal $\mathcal{I}^\Eis \subset \mathscr{H}$ of (\ref{eisenstein def}).

\begin{proposition}\label{J pi decomp}
Every $f\in\mathcal{I}^\Eis$ satisfies 
\begin{equation}\label{spectral}
\J (f,s) =  \sum_{\mathrm{unr.\, cusp.\,}\pi }  \lambda_\pi(f)   \frac{  \mathscr{P}_0(  \phi ,s) \mathscr{P}_3(\overline{\phi} ,\eta)}{\langle \phi, \phi \rangle },
\end{equation}
where the sum is over all unramified cuspidal automorphic representations $\pi \subset \mathcal{A}_\cusp(G_0)$, and $\phi \in \pi^{U_0}$ is any nonzero vector.
Moreover, $\J(f,s)$ only depends on the image of $f$ under the quotient map  $\mathscr{H} \to \mathscr{H}_\mathrm{aut}$.
\end{proposition}

\begin{proof}
If we use Remark \ref{rem:usual kernel} to view  (\ref{the kernel}) as a function on $G_0(\A) \times G_0(\A)$,  then invoke the decomposition 
\[
\K_f (x,y) = \K_{f,\cusp} (x,y)  +  \K_{f,\mathrm{sp}} (x,y)  
\]
of \cite[Theorem 4.3]{YZ}, and then convert all three terms back into functions on  $G_0(\A) \times G_3(\A)$, the result is a decomposition
\begin{align*}
\K_f (g_0,g_3)   &  =  \sum_{\mathrm{unr.\, cusp.\,}\pi } \lambda_\pi(f)  \cdot  \frac{ \phi(g_0) \overline{\phi(g_3)} }{ \langle \phi ,\phi \rangle}  \\
&\quad  + \sum_{\mathrm{unr.\, quad.\,}\chi } \lambda_\chi(f)   \cdot  \chi(\det(g_0))  \cdot   \chi(\det(g_3))  . 
\end{align*}
The first  sum is over all unramified cuspidal representations $\pi$,  and $\phi \in \pi^{U_0}$ is any nonzero vector.
The second sum is over all unramified quadratic characters 
\[
\Pic(X) \iso  F^\times \backslash \A^\times / \mathbb{O}^\times  \map{\chi} \{ \pm 1\},
\]
and 
\[
\lambda_\chi(f) = \int_{G_0(\A)} f(g) \chi(\det(g))\, dg.
\]

The distribution (\ref{partial distribution}) now decomposes as
\[
\J_n(f,s) =   \sum_{\mathrm{unr.\, cusp.\,}\pi }  \J_n^\pi(f,s) 
+  \sum_{\mathrm{unr.\, quad.\,}\chi } \J_n^\chi(f,s) ,
\]
where we have set
\[
\J^\pi _n(f,s) =  \frac{ \lambda_\pi(f)  }{  \langle \phi ,\phi \rangle }
\left( \int_{[T_0]_n } \phi(t_0)   |t_0|^{2s} \,  dt_0  \right)
\left( \int_{ [T_3]}   \overline{\phi(t_3)}   \eta(t_3)\,   dt_3 \right)
\]
and
\[
\J^\chi _n(f,s)  = \lambda_\chi(f)  
\left(   \int_{[T_0]_n } \chi(\det(t_0))  |t_0|^{2s} \,  dt_0    \right)
\left(  \int_{ [T_3]}   \chi(\det(t_3))   \eta(t_3)\,   dt_3 \right) .
\]

In fact,  $\J_n^\chi(f,s)=0$ for all $n$ and all $\chi$.    The proof is similar to that of \cite[Lemma 4.4]{YZ}.  Briefly,  if the restriction of $\chi$ to 
\[
\Pic^0(X)=  F^\times \backslash  \{ a \in \A^\times : |a|=1 \} / \mathbb{O}^\times
  \]
 is nontrivial then the integral over $[T_0]_n$ vanishes.  If the restriction of $\chi$ to $\Pic^0(X)$ is trivial then the integral over $[T_3]$ vanishes.   This leaves us with
 \[
\J_n(f,s) =    \sum_{\mathrm{unr.\, cusp.\,}\pi }   \J_n^\pi(f,s) ,
\]
and (\ref{spectral}) follows by summing both sides over $n$.

For the final claim, suppose $f$ has trivial image under $\mathscr{H}\to\mathscr{H}_\mathrm{aut}$.  This implies that $f$ annihilates $\mathscr{A}_\C$, 
and lies in $\mathcal{I}^\Eis$.  The first inclusion in  (\ref{harder}) implies that  $\lambda_\pi(f)=0$ for all $\pi$, and so $\J(f,s)=0$ by (\ref{spectral}).
\end{proof}

The Hecke algebra $\mathscr{H}$ has a $\Q$-basis $\{ f_D \}$ indexed by the effective divisors $D\in \Div(X)$, and defined as follows (see also \cite[\S 3.1]{YZ}).  Let $S_D$ be the image of the set 
\[\left\{ X \in M_2(\mathbb{O})\colon \mathrm{div}(\det X) = D\right\}\]
in $\PGL_2(\A) = G_0(\A)$.  Then $f_D \colon U_0\backslash G_0(\A)/U_0 \to \Q$ is the characteristic function of $S_D$.

The remainder of \S \ref{s:analytic} is devoted to interpreting the $r^\mathrm{th}$-derivative
\[
\frac{d^r}{ds^r} \J (f_D ,s) |_{s=0}
= \sum_{\substack{ \xi\in K_3 \\ \Tr_{K_3/F}(\xi) =1  }} \frac{d^r}{ds^r} \J( \xi ,f_D,s)|_{s=0},
\]
in terms of algebraic geometry, for each such $D$.  This interpretation can be found in \S \ref{ss:geometric orbital} below (see especially Proposition \ref{prop:geometric orbital}), after the definitions and preliminary results of \S \ref{ss:analytic moduli} and \ref{ss:orbital local}.


\subsection{Some moduli spaces}
\label{ss:analytic moduli}


Fix an integer $d$.    Our  goal is to describe a commutative diagram of $\kk$-schemes
\begin{equation}\label{fundamental 2}
\xymatrix{
 {  N_{ (d_1,d_2) }  }   \ar[d]_{\beta}  \ar[r] &  { \Sigma_{d_1}(Y_3) \times_\kk \Sigma_{d_2}(Y_3) } \ar[d]^{\otimes}  \\
 {    A_d  }  \ar[d]_{\Tr}   \ar[r]_{f_3^\sharp} &    \Sigma_{2d}(Y_3)  \\
  {    \Sigma_d(X)  } 
}
\end{equation}
for any pair of  non-negative integers with $d_1+d_2=2d$, in such a way that the square is cartesian.
   We will define these  schemes by specifying their functors of points.
Let $S$ be a $\kk$-scheme.

Denote by    $\Sigma_d(X)(S)$  the set of isomorphism classes of pairs $(\Delta, \zeta )$  consisting of
\begin{itemize}
\item
 a line bundle $\Delta$ on $X_S=X\times_\kk S$ of degree $d$,
\item
a nonzero section $\zeta \in  H^0(X_S, \Delta)$.
\end{itemize}
Here (and hereafter) these conditions should be interpreted fiber-by-fiber: for every $s\in S$ we assume $\deg(\Delta_s)=d$ and  $\zeta_s\neq 0$.  With this convention, it follows from \cite[Proposition I.2.5]{Milne} that $\mathrm{div}(\zeta)$ is flat over $S$.  Moreover, this divisor determines the pair  $(\Delta,\zeta)$ up to isomorphism,  and so $\Sigma_d(X)$ is identified with the Hilbert scheme parametrizing effective relative Cartier divisors on $X$ of degree $d$.  It follows from the discussion of  \cite[\S 9.3]{BLR} that there is a canonical isomorphism
\begin{equation}\label{symmetric id}
\Sigma_d(X) \iso \mathrm{Sym}^d(X) \iso S_d\backslash X^d
\end{equation}
(the rightmost scheme is the  GIT quotient), and that $\Sigma_d(X)$ is a smooth projective $\kk$-scheme.
The schemes   $\Sigma_{2d}(Y_3)$ and $\Sigma_{d_i}(Y_3)$  are defined similarly, and the vertical arrow in (\ref{fundamental 2})  labeled $\otimes$ has the obvious meaning.

Let  $A_d(S)$  be set of isomorphism classes of  pairs  $(\Delta, \xi )$ consisting of
\begin{itemize}
\item
a  line bundle $\Delta$ on $X_{S}$ of degree $d$,
\item
a section $\xi \in H^0( Y_{3S} , f_3^*\Delta)$ with nonzero trace 
\[
\mathrm{Tr}_{Y_3/X}(\xi) = \xi+\xi^{\sigma_3} \in H^0(X_S,\Delta).
\] 
\end{itemize}
The arrows in (\ref{fundamental 2}) emanating from $A_d$ are
\[
\Tr (\Delta,\xi) = (\Delta , \Tr_{Y_3/X} (\xi) ) \quad\mbox{ and }\quad  f_3^\sharp (\Delta, \xi)  = (f_3^*\Delta ,\xi).
\]
It is easy to check that $\Tr$ is a quasi-projective morphism, and hence  $A_d$ is a quasi-projective $\kk$-scheme.

Now define  $\tilde{N}_{d_1,d_2}(S)$ to be the groupoid of quadruples $(\mathcal{L}_1, \mathcal{L}_2,  \mathcal{L}_3, \phi)$ consisting of
\begin{itemize}
\item line bundles $\mathcal{L}_1, \mathcal{L}_2\in \mathrm{Pic}(X_S)$ and $\mathcal{L}_3\in \mathrm{Pic}(Y_{3S})$ satisfying 
\[
2\deg(\mathcal{L}_1) - d_1=  \deg(\mathcal{L}_3)  =2\deg(\mathcal{L}_2) - d_2 ,
\]
\item a morphism  $\phi  :   f_{3*}\mathcal{L}_3 \to \mathcal{L}_1 \oplus \mathcal{L}_2 $ of rank 2 vector bundles on $X_S$
with nonzero determinant.
\end{itemize}
The functor $\tilde{N}_{d_1,d_2}$ from $\kk$-schemes to groupoids is represented by an Artin stack.
The Picard group $\Pic(X_S)$ acts on $\tilde{N}_{ d_1,d_2}(S)$ by twisting
\[
( \mathcal{L}_1,\mathcal{L}_2 , \mathcal{L}_3, \phi)  \otimes \mathcal{L}
 = 
 (   \mathcal{L}_1  \otimes \mathcal{L}  ,     \mathcal{L}_2   \otimes \mathcal{L}   ,   \mathcal{L}_3  \otimes  f_3^* \mathcal{L} ,  \phi \otimes  \mathrm{id} ),
 \]   
defining an action of the Picard stack $\Pic_X$ on $\tilde{N}_{d_1,d_2}$.  Recall $\Pic_X$ is a smooth Artin stack over $\kk$ of dimension $g-1$.  The representability of the quotient stack
\[
N_{ (d_1,d_2) } = \tilde{N}_{(d_1,d_2)} /\Pic_X 
\]
(see \cite[\S 4]{Ngo} for a discussion of the meaning of such a quotient)
by a scheme is part of the following result, which also defines the arrows in (\ref{fundamental 2}) emanating from $N_{(d_1,d_2)}$.

\begin{proposition}\label{prop:fundamental 2 cartesian}
There is a canonical isomorphism 
\[
N_{(d_1,d_2)} \iso A_d \times_{  \Sigma_{2d}(Y_3) } \big(  \Sigma_{d_1}(Y_3) \times_\kk \Sigma_{d_2}(Y_3)  \big).
\]
\end{proposition}

\begin{proof}
Suppose $( \mathcal{L}_1, \mathcal{L}_2 ,  \mathcal{L}_3 , \phi) \in \tilde{N}_{(d_1,d_2) }(S)$.
The pullback of $\phi$ via $f_3:Y_{3S} \to X_S$ is a morphism
\[
 \mathcal{L}_3 \oplus  \mathcal{L}_3^{\sigma_3}  \to  f_3^*\mathcal{L}_1 \oplus f_3^*\mathcal{L}_2   ,
\]
 encoded by four maps
\begin{eqnarray}\label{more matrix entries}
 \mathcal{L}_3    &\map{\quad \bm{a} \quad } &f_3^*\mathcal{L}_1    \\
  \mathcal{L}_3^{\sigma_3}     &\map{ \bm{b}=\bm{a}^{\sigma_3}  }  & f_3^*\mathcal{L}_1  \nonumber  \\
  \mathcal{L}_3      &\map{\quad \bm{c} \quad } & f_3^*\mathcal{L}_2 \nonumber  \\
    \mathcal{L}_3^{\sigma_3}    &\map{ \bm{d} =\bm{c}^{\sigma_3}} & f_3^*\mathcal{L}_2  \nonumber  .
\end{eqnarray}

By viewing $\bm{a}$ and $\bm{d}$ as global sections of the line bundles 
\begin{equation}\label{K bundles}
\mathcal{K}_1=\underline{\Hom} (    \mathcal{L}_3 , f_3^*\mathcal{L}_1 ) ,\quad  \mathcal{K}_2=\underline{\Hom} (  \mathcal{L}_3^{\sigma_3} ,  f_3^*\mathcal{L}_2  ) 
\end{equation}
of degree $d_1$ and $d_2$, respectively, we obtain $S$-points of $\Sigma_{d_1}(Y_3)$ and $\Sigma_{d_2}(Y_3)$.
This defines a morphism
\[
\tilde{N}_{(d_1,d_2) } \to \Sigma_{ d_1}(Y_3) \times_\kk \Sigma_{d_2}(Y_3) ,
\]
which is easily seen to descend to the quotient $N_{(d_1,d_2)}$.

Define a degree $d$ line bundle
\begin{equation}\label{first det bundle}
\Delta  = \underline{\Hom}(    \det( f_{3*} \mathcal{L}_3)   ,  \det( \mathcal{L}_1\oplus \mathcal{L}_2)    )
\end{equation}
on $X_S$.  Its pullback to $Y_{3S}$ is  
\[
f_3^*\Delta   \iso \underline{\Hom} \big(    \mathrm{det}(\mathcal{L}_3 \oplus \mathcal{L}_3^{\sigma_3})   ,  \mathrm{det}( f_3^*\mathcal{L}_1\oplus f_3^*\mathcal{L}_2 ) \big) 
\]
which has a global section $\bm{a}\bm{d}$ defined by $ e_1\wedge e_2 \mapsto \bm{a}(e_1)\wedge \bm{d}(e_2)$ for  local sections $e_1$ and $e_2$ of $\mathcal{L}_3$ and $\mathcal{L}_3^{\sigma_3}$, respectively.  The equality of global sections
\[
\det(\phi) = \mathrm{Tr}_{Y_3/X}(\bm{a}\bm{d}) \in H^0(X_S, \Delta),
\]
 shows that $\mathrm{Tr}_{Y_3/X}(\bm{a}\bm{d})$ is nonzero.
Thus the pair $(\Delta, \bm{a} \bm{d})$ defines an $S$-point of  $A_d$, and we have  constructed a morphism
\[
 \tilde{N}_{(d_1,d_2) } \to A_d.
\]
Again, this is easily seen to descend to the quotient $N_{(d_1,d_2)}$.

Combining the above constructions defines a morphism
\begin{equation}\label{cartesian moduli map 2}
 N_{(d_1,d_2) } \to A_d \times_{ \Sigma_{2d}(Y_3) }  \big( \Sigma_{ d_1}(Y_3) \times_\kk \Sigma_{d_2}(Y_3) \big).
\end{equation}
To show it is an isomorphism we construct the inverse.  An $S$-point of the fiber product consists of a line bundle $\Delta$ on $X_S$  of a degree $d$,
line bundles $\mathcal{K}_1$ and $\mathcal{K}_2$ on $Y_{3S}$ of degrees $d_1$ and $d_2$, global sections 
\[
\xi \in H^0(Y_{3S}, f_3^*\Delta),\quad 
\bm{a} \in H^0(Y_{3S} , \mathcal{K}_1)  , \quad \bm{d}\in H^0(Y_{3S} , \mathcal{K}_2) ,
\]
such that $\mathrm{Tr}_{Y_3/X} (\xi) \in H^0(X_S,\Delta)$ is nonzero,
and an isomorphism $f_3^* \Delta \iso \mathcal{K}_1\otimes \mathcal{K}_2$  identifying $\xi = \bm{a}\otimes\bm{d}$.

Starting from this data we define $(\mathcal{L}_1,\mathcal{L}_2,\mathcal{L}_3,\phi)$ as follows.  Let $\mathcal{L}_1$ be any line bundle on $X_S$.
If we define  $\mathcal{L}_3 = \mathcal{K}_1^{-1} \otimes f_3^*\mathcal{L}_1$, there are canonical isomorphisms
\begin{align*}
\big( \mathcal{L}_3^{\sigma_3}  \otimes \mathcal{K}_2 \big)^{\sigma_3}
& \iso \mathcal{L}_3 \otimes ( \mathcal{K}_1^{-1}  \otimes f_3^*\Delta )^{\sigma_3} \\
& \iso  \mathcal{L}_3 \otimes ( \mathcal{L}_3  \otimes f_3^*\mathcal{L}_1^{-1} \otimes f_3^*\Delta )^{\sigma_3} \\
& \iso  \mathcal{L}_3^{\sigma_3} \otimes ( \mathcal{L}_3  \otimes f_3^*\mathcal{L}_1^{-1} \otimes f_3^*\Delta ) \\
& \iso  \mathcal{L}_3^{\sigma_3} \otimes ( \mathcal{K}_1^{-1}  \otimes  f_3^*\Delta ) \\
& \iso \mathcal{L}_3^{\sigma_3} \otimes \mathcal{K}_2.
\end{align*}
Viewing this as descent data relative to $Y_3/X$, we obtain a line bundle $\mathcal{L}_2$ on $X_S$ endowed with an isomorphism 
$ f_3^* \mathcal{L}_2 \iso  \mathcal{K}_2  \otimes \mathcal{L}_3^{\sigma_3} $.  We now view $\bm{a}$ and $\bm{d}$ as global sections of the line bundles
\[
 \underline{\Hom}( \mathcal{L}_3  ,  f_3^*\mathcal{L}_1     )  \iso \mathcal{K}_1  ,\quad 
 \underline{\Hom}(  \mathcal{L}_3^{\sigma_3}  ,  f_3^*\mathcal{L}_2   )  \iso   \mathcal{K}_2  ,
\]
and use (\ref{more matrix entries}) to define two more global sections $\bm{b}$ and $\bm{c}$.  These four sections define a morphism
\[
\mathcal{L}_3 \oplus  \mathcal{L}_3^{\sigma_3}    \to  f_3^*\mathcal{L}_1 \oplus f_3^*\mathcal{L}_2   
\]
of line bundles on $Y_{3S}$, 
which descends to a morphism $\phi :  f_{3*}\mathcal{L}_3 \to \mathcal{L}_1\oplus \mathcal{L}_2 $ of vector bundles on $X_S$.
The composition
\begin{align*}
f_3^*\Delta & \iso \mathcal{K}_1\otimes\mathcal{K}_2  \\
& \iso \underline{\Hom}( \mathcal{L}_3 \otimes  \mathcal{L}_3^{\sigma_3}  ,  f_3^*\mathcal{L}_1   \otimes  f_3^*\mathcal{L}_2   ) \\
& \iso f_3^*\underline{\Hom}(    \det( f_{3*} \mathcal{L}_3)   ,  \det( \mathcal{L}_1\oplus \mathcal{L}_2)    ).
\end{align*}
is compatible with the descent data on both sides, and descends to an isomorphism 
\[
\Delta  \iso  \underline{\Hom}(    \det( f_{3*} \mathcal{L}_3)   ,  \det( \mathcal{L}_1\oplus \mathcal{L}_2)    )
\]
sending $ \mathrm{Tr}_{Y_3/X} (\xi) \mapsto \det(\phi)$.  Thus $\det(\phi)$ is nonzero.

This construction defines the desired inverse to (\ref{cartesian moduli map 2}).
\end{proof}

\begin{proposition}\label{prop:Nsmooth}
Let $g$ and $g_3$ be the genera of $X$ and $Y_3$, respectively, so that $g_3=2g-1$.
\begin{enumerate}
\item
The vertical morphisms in (\ref{fundamental 2}) labeled $\beta$ and $\otimes$ are finite.   
\item
 If $d\ge 2 g_3-1 $ then $N_{(d_1,d_2)}$ is  smooth over $\kk$ of  dimension $2d -g +1$.
\end{enumerate}
\end{proposition}

\begin{proof}

For the first  claim,  the morphism in (\ref{fundamental 2}) labeled $\otimes$ can be identified with the canonical morphism
\[
\mathrm{Sym}^{d_1}(Y_3) \times \mathrm{Sym}^{d_2}(Y_3) \to \mathrm{Sym}^{2d}(Y_3) ,
\]
using the obvious analogues of (\ref{symmetric id}).
This is obviously finite.   Proposition \ref{prop:fundamental 2 cartesian}, which asserts that the square in (\ref{fundamental 2}) is cartesian,  therefore implies the finiteness of $\beta$.

The proof of the second claim is similar to \cite[Proposition 3.1(2)]{YZ}.
Examining the proof of Proposition \ref{prop:fundamental 2 cartesian} yields  a cartesian diagram
\begin{equation}\label{deep cartesian}
\xymatrix{
{  N_{(d_1,d_2) }  } \ar[r]\ar[d]  &  {  \Sigma_{d_1}(Y_3)   }  \ar[d] \\
{  \Pic^d_X \times _\kk \Sigma_{d_2}(Y_3)   } \ar[r]  & { \Pic^{d_1}_{Y_3} }  .
}
\end{equation}
Here $\Pic_X^d \subset \Pic_X$ is the substack of line bundles of degree $d$, and  
$ \Pic^{d_1}_{Y_3}$ is defined similarly.
To define the various morphisms, start with a quadruple  $(\mathcal{L}_1,\mathcal{L}_2,\mathcal{L}_3,\phi)\in N_{(d_1,d_2)}(S)$.  
Recall the  line bundles 
\[
\Delta \in \Pic^d(X_S) , \quad \mathcal{K}_1\in \Pic^{d_1}(Y_{3S}) ,  \quad \mathcal{K}_2\in \Pic^{d_2}(Y_{3S}) 
\]
 of (\ref{first det bundle}) and (\ref{K bundles}), which are related by 
 \[
 f_3^* \Delta \iso \underline{\Hom}(\mathcal{L}_3\otimes \mathcal{L}_3^{\sigma_3} ,  f_3^*\mathcal{L}_1 \otimes f_3^*\mathcal{L}_2)
 \iso \mathcal{K}_1\otimes \mathcal{K}_2,
 \]
 and the global  sections $\bm{a}\in H^0(Y_{3S} , \mathcal{K}_1)$ and  $\bm{b}\in H^0(Y_{3S} , \mathcal{K}_2)$.
 The top horizontal arrow sends $(\mathcal{L}_1,\mathcal{L}_2,\mathcal{L}_3,\phi)\mapsto (\mathcal{K}_1,\bm{a})$.  The vertical arrow on the left sends $(\mathcal{L}_1,\mathcal{L}_2,\mathcal{L}_3,\phi)\mapsto (\Delta, \mathcal{K}_2,\bm{d})$, and the bottom horizontal arrow sends 
 $(\Delta, \mathcal{K}_2,\bm{d}) \mapsto f_3^*\Delta \otimes \mathcal{K}_2^{-1}$.  The right vertical arrow sends $(\mathcal{L}_1,\bm{a}) \mapsto \mathcal{L}_1$.

Now assume that   $d \ge 2 g_3-1$.  This implies that either  $d_1 \ge 2 g_3-1$ or $d_2 \ge 2g_3-1$, and  without loss of generality we assume the former inequality.  This implies that the vertical arrow on the right in (\ref{deep cartesian}) is smooth of relative dimension $d_1-g_3+1$, and hence the same is true of the vertical arrow on the left.
As the target of the left arrow is smooth over $\kk$ of dimension $g-1+d_2$, we deduce that $N_{(d_1,d_2)}$ is  smooth of dimension $2d - g+ 1$.
\end{proof}


\subsection{A local system}
\label{ss:orbital local}


Let $d$ be any non-negative integer.  By identifying $\{ \pm 1 \}^d \iso \Aut(Y/Y_3)^d$, the group $\Gamma_d=\{ \pm 1\}^d \semi S_d$ acts on $Y^d$.   As in  (\ref{symmetric id}), there are canonical isomorphisms
\[
\Sigma_d(Y_3) \iso \mathrm{Sym}^d(Y_3)  \iso \Gamma_d \backslash Y^d.
\]

Denote by  $\eta_d: \{\pm 1\}^d \to \{ \pm 1\}$ the product character,  extend it to a character $\eta_d:\Gamma_d\to \{ \pm 1\}$ trivial on $S_d$, and form  the GIT quotient
\[
\Sigma_d(Y/Y_3) =  \mathrm{Ker} (\eta_d) \backslash Y^d.
\]

\begin{proposition}
The canonical morphism 
\begin{equation}\label{sym cover}
\Sigma_d(Y/Y_3) \to \Sigma_d(Y_3)
\end{equation}
 is a finite \'etale double cover.
\end{proposition}

\begin{proof}
The finiteness claim is clear.  \'Etaleness can be verified on the level of completed \'etale local rings.  Accordingly, let $A$ be the completed \'etale local ring  at a point of $\Sigma_d(Y_3) ( \kk^{\mathrm{alg}})$, and let $B$ be the product of the completed \'etale local rings at all points
of $\Sigma_d(Y/Y_3)(\kk^{\mathrm{alg}})$ above it.  We now have a cartesian diagram
\[
\xymatrix{
{ \Spec(B) } \ar[r]  \ar[d]  & {  \Sigma_d(Y/Y_3) } \ar[d]  \\
{ \Spec(A) } \ar[r]   &  {  \Sigma_d(Y_3),} 
}
\]
and it suffices to prove that $B\iso A\oplus A$ as  $A$-algebras.

Set $D= \kk^{\mathrm{alg}}[[x]]$, and consider the $D$-algebra $E=D\oplus D$.  
Let $D_d$ and $E_d$ be the $d$-fold completed tensor products  (over $\kk^{\mathrm{alg}}$) of $D$ and $E$ with themselves.  
These are complete local $\kk$-algebras 
carrying actions of $S_d$ and $\Gamma_d$, respectively, and $D_d^{S_d} = E_d^{\Gamma_d}$.
As $Y\to Y_3$ is an \'etale double cover of smooth curves, we may identify $A\to B$ with 
\[
D_d^{S_d}  \to E_d^{\mathrm{Ker}(\eta_d)}.
\]

Label the orthogonal idempotents in $E$ as $e_+, e_-\in E$.  Each tuple of signs $x=(x_1,\ldots, x_d) \in \{ \pm 1\}^d$ determines an idempotent
\[
e_x = e_{x_1} \otimes \cdots \otimes e_{x_d} \in E_d.
\]
Each of the orthogonal idempotents 
\[
f_+  = \sum_{\substack{    x\in \{\pm 1\}^d \\ \eta_d(x) = 1   } } e_x ,
\qquad 
f_-  = \sum_{\substack{    x\in \{\pm 1\}^d \\ \eta_d(x) = -1   } } e_x 
\]
in $E_d$  is fixed by the action of the subgroup $\mathrm{Ker}(\eta_d) \subset \Gamma_d$, 
and the composition  $D_d \to E_d \to f_\pm E_d$ restricts to an isomorphism 
\[
D_d^{S_d} \iso ( f_\pm E_d) ^{\mathrm{ker}(\eta_d)}.
\]
Thus we obtain the desired decomposition
\[
E_d^{\mathrm{ker}(\eta_d)} = ( f_+ E_d) ^{\mathrm{ker}(\eta_d)}  \oplus  (f_- E_d)^{\mathrm{ker}(\eta_d)}  \iso D_d^{S_d} \oplus D_d^{S_d}.
\]
\end{proof}

Choose any prime $\ell$ different from the characteristic of $\kk$.  
The \'etale double cover (\ref{sym cover}) determines a quadratic character of the \'etale fundamental group of $\Sigma_d(Y_3)$,
which then determines a rank one  \'etale local system $L_d$ of $\Q_\ell$-vector spaces on $\Sigma_d(Y_3)$.

Now fix $d_1,d_2\ge 0$,  and let 
\[
j: N_{(d_1,d_2)} \to \Sigma_{d_1}(Y_3) \times_\kk \Sigma_{d_2}(Y_3)
\]
be as in (\ref{fundamental 2}).  Define a  rank one local system
\begin{equation}\label{local system}
 L_{(d_1,d_2)}  =  j^* (L_{d_1} \boxtimes \Q_\ell ) 
 \end{equation}
 of $\Q_\ell$-vector spaces on $N_{(d_1,d_2)}$.

\begin{proposition}\label{prop:etale fiber action}
Suppose  $z=(\mathcal{L}_1,\mathcal{L}_2,\mathcal{L}_3,\phi)\in N_{(d_1,d_2)} (\kk)$
is a $\kk$-point, and $\bar{z} \to z$ is a geometric point above it.
Recalling the character 
\[
\Pic(Y_3) \iso K_3^\times \backslash \A_3^\times / \mathbb{O}_3^\times \map{\eta} \{ \pm 1\}
\]
of Lemma \ref{lem:eta}, the Frobenius $\Frob_z\in \Aut(\bar{z}/z)$ acts on  $L_{(d_1,d_2) , \bar{z}}$ via the scalar  $\eta(\mathcal{L}_3)$.
\end{proposition}

\begin{proof}
If  $y=(\mathcal{K} , \bm{a}) \in \Sigma_n(Y_3)(\kk)$ is a $\kk$-point, and $\bar{y} \to y$ is a geometric point above it,
it is easy to see from the definitions that $\Frob_y$ acts on the fiber $L_{n,\bar{y}}$ as $\eta(\mathcal{K})$.   Recalling the construction of the map $j$ from the proof of Proposition \ref{prop:fundamental 2 cartesian}, especially (\ref{K bundles}), it follows that $\Frob_z$ acts on $L_{(d_1,d_2) , \bar{z}}$ by the scalar 
$
\eta(\mathcal{K}_1) =   \eta(\mathcal{L}_3) \eta(f_3^*\mathcal{L}_1).
$
The proof of Lemma \ref{lem:eta} shows that $\eta(f_3^*\mathcal{L}_1) = \chi_1(\mathcal{L}_1^{\otimes 2}) =1$, completing the proof.
\end{proof}


\subsection{Geometric interpretation of orbital integrals}
\label{ss:geometric orbital}


Let $D$ be an effective  divisor on $X$ of degree $d$.   The constant function $1$ defines a global section of $\co_X(D)$, defining a point
$(\co_X(D), 1) \in \Sigma_d(X)(\kk)$.  We define $A_D$ as the fiber product
\[
\xymatrix{
{ A_D } \ar[rr] \ar[d] & &  { A_d } \ar[d] \\
{ \Spec(\kk) } \ar[rr]^{  ( \co_X(D) ,1 )  } & & { \Sigma_d(X) .}
}
\]

\begin{proposition}\label{prop:invariant domain}
There is a canonical bijection
\[
A_D(\kk) \iso  \left\{ \xi \in K_3 : \begin{array}{c} \Tr_{K_3/F}(\xi) =1 \\ \mathrm{div}(\xi) + f_3^*D \ge 0   \end{array}  \right\} .
\]
\end{proposition}

\begin{proof}
A $\kk$-point of $A_D$ consists of a line bundle $\Delta \in \Pic^d(X)$ and a global section $\xi \in H^0(Y_3,f_3^*\Delta)$, together with  an 
isomorphism $\Delta\iso \co_X(D)$ identifying $\Tr_{Y_3/X}(\xi)=1$.  In other words, a trace $1$ element of
\[
 H^0(Y_3, f_3^*\co_X(D) ) = H^0(Y_3, \co_{Y_3}(f_3^*D) ) = 
 \left\{ \xi \in K_3 :  \mathrm{div}(\xi) + f_3^*D \ge 0   \right\} .
\]
\end{proof}

Recall the canonical $\Q$-basis $\{f_D\} \subset \mathscr{H}$ indexed by effective divisors $D\in \mathrm{Div}(X)$. 
We are now ready to give a geometric interpretation of the orbital integral $\J(\xi,f_D,s)$ appearing in (\ref{J decomp}).

 Using the homeomorphism of Lemma \ref{lem:coset transfer}, we change the point of view and regard $f_D$  as a compactly supported function 
\begin{equation}\label{D hecke}
f_D : U_0 \backslash J(\A) / U_3 \to \Q.
\end{equation}
To make this more explicit, define a free  rank one $\mathbb{O}$-module
\[
\Delta =  \Hom\big( \mathrm{det}(\mathbb{O}_3) , \mathrm{det}( \mathbb{O}_0 ) \big),
\]
and define $\tilde{\Omega}_D$ as the set of all  
$
\phi \in \tilde{J}(\A)= \mathrm{Iso} ( \A_3  , \A_0) 
$ 
such that  
\begin{itemize}
\item
$\phi( \mathbb{O}_3)  \subset \mathbb{O}_0 $, 
\item
 $\det(\phi)  \in \Delta$ generates  $ a_D \Delta$ as an $\mathbb{O}$-module.
\end{itemize}
Here $a_D \in \A^\times/ \mathbb{O}^\times \iso \mathrm{Dix}(X)$ represents the divisor $D$.  Now let $\Omega_D$ be the image of $\tilde{\Omega}_D$ under the quotient map  $\tilde{J}(\A) \to J(\A)$.  The  function (\ref{D hecke})  is the characteristic function of $\Omega_D$.

Recall the local system $L_{(d_1,d_2)}$ of (\ref{local system}).
Its pushforward   via the finite morphism  $\beta$  of (\ref{fundamental 2}) is a  constructible sheaf on $A_d$.
  The following result relates  this sheaf to the distribution (\ref{J decomp}).

\begin{proposition}\label{prop:geometric orbital}
Fix  $\xi\in K_3$ with $\Tr_{K_3/F}(\xi)=1$,  and view $A_D(\kk) \subset K_3$ using Proposition $\ref{prop:invariant domain}$.
\begin{enumerate}
\item If $\xi \not\in A_D(\kk)$  then $\J( \xi , f_D,s) = 0$.
\item If  $\xi \in A_D(\kk)$  then 
 \[
 \J( \xi , f_D,s) =
\sum_{  \substack{  d_1,d_2 \ge 0  \\ d_1+d_2=2d }}
 q^{  (d_1-d_2) s }  \cdot
 \mathrm{Trace} \big(\Frob_\xi  ; \,   \big( \beta_* L_{ (d_1,d_2) } \big)_{\bar \xi} \big),
\]
where $\bar{\xi}$ is a geometric point above $\xi:\Spec(\kk) \to A_D$.
\end{enumerate}       
\end{proposition}

\begin{proof}
Recalling (\ref{orbit invariant}), let $\gamma \in T_0(F) \backslash J(F) / T_3(F)$ be the unique element with $\xi=\inv(\gamma)$, and fix a lift
 \[ 
 \gamma\in \tilde{J}(F)=\mathrm{Iso}(K_3,K_0).
 \]  
 We view  $\gamma$ as a rational map
\[ 
f_{3*}\co_{Y_3}  \map{\gamma}   f_{0*}\co_{Y_0} = \co_X\oplus \co_X
\]
of rank two vector bundles.
Each triple
\begin{equation}\label{triple divisors}
(E_1,E_2,E_3) \in \Div(X) \backslash \big( \Div(X) \times \Div(X) \times \Div(Y_3) \big)
\end{equation}
determines  a rational map $\phi_\gamma$ by the commutativity of 
\[
\xymatrix{
{  f_{3*}\co_{Y_3}(-E_3)}   \ar@{-->}[d] \ar@{-->}[rr]^{\phi_{\gamma}} & &  {  \co_X(-E_1) \oplus \co_X(-E_2) } \ar@{-->}[d] \\
{  f_{3*}\co_{Y_3} } \ar@{-->}[rr]_{ \gamma} &   &  { \co_X \oplus \co_X }. 
}
\]
Here the vertical arrows are the canonical rational maps.

Denote by $\mathfrak{N}_{D,\gamma}$ the set of all triples (\ref{triple divisors}) such that $\phi_\gamma$ is a morphism of vector bundles (as opposed to just  a rational map) with   $\mathrm{div} (  \det(\phi_\gamma) ) = D$.  
If the canonical bijection
\begin{eqnarray*}\lefteqn{
\A^\times \backslash (\tilde{T}_0 \times \tilde{T}_3)(\A) / (\tilde{T}_0 \times \tilde{T}_3)(\mathbb{O})      } \\
&  \iso &  \Div(X) \backslash \big(  \Div(Y_0) \times \Div(Y_3) \big) \\
& \iso & \Div(X) \backslash \big(  \Div(X) \times \Div(X) \times \Div(Y_3) \big)
\end{eqnarray*}
sends $(t_0,t_3) \mapsto (E_1,E_2,E_3)$,  then
\begin{equation}\label{characteristic hecke}
\tilde{f}_D(t_0^{-1} \gamma t_3) = 
\begin{cases}
1 &  \mbox{if } (E_1,E_2,E_3) \in \mathfrak{N}_{D,\gamma}    \\
0 & \mbox{otherwise.}
\end{cases}
\end{equation}
Here, as in  (\ref{D hecke}), we define $\tilde{f}_D$ as the characteristic function of the subset 
\[
\tilde{\Omega}_D = \prod_{x\in |X|} \tilde{\Omega}_{D,x} \subset \tilde{J}(\A).
\]

Any triple $(E_1,E_2,E_3) \in \mathfrak{N}_{D,\gamma}$ defines a quadruple
\begin{equation}\label{quadruple}
(\co_X( -E_1) , \co_X(-E_2) , \co_{Y_3}(-E_3) , \phi_\gamma ) \in N_{(d_1,d_2)}(\kk)
\end{equation}
where $d_i = \deg(E_3)-2\deg(E_i)$.  The image of this quadruple under
\[
N_{(d_1,d_2)}(\kk) \to A_d(\kk) \map{\Tr} \Sigma_d(X)(\kk)
\]  
is the pair $(\Delta,\zeta)$ defined by  the line bundle
\[
\Delta = \underline{\Hom} \big( \mathrm{det} ( f_{3*} \co_{Y_3}(-E_3) ) , \det (\co_X( -E_1) \oplus \co_X(-E_2) ) \big)
\]
on $X$ and its global section  $\zeta = \det(\phi_\gamma)$ with divisor $D$.
In other words,  the image of (\ref{quadruple}) under  $N_{(d_1,d_2)}(\kk) \to A_d(\kk)$ lies in the fiber over 
\[
(\Delta,\zeta) \iso (\co_X(D),1) \in \Sigma_d(X)(\kk),
\] 
and so defines an element of $A_D(\kk) \subset K_3$.  Tracing through the definitions, this element is precisely $\xi$, and hence $\xi \in A_D(\kk)$.

 In particular, if  $\xi \not\in A_D(\kk)$ then $\mathfrak{N}_{D,\gamma} = \emptyset$.
It then follows from (\ref{characteristic hecke}) that
\[
\tilde{f}_D(t_0^{-1} \gamma t_3)=0  \mbox{ for all } (t_0,t_3) \in \tilde{T}_0(\A) \times \tilde{T}_3(\A),
\] 
which in turn implies 
\[
f_D(t_0^{-1} \gamma t_3)=0  \mbox{ for all } (t_0,t_3) \in T_0(\A) \times T_3(\A).
\] 
The relation $\J(\gamma, f_D,s)=\J(\xi, f_D,s) = 0$ now follows directly from the definitions of \S \ref{ss:orbital notation}, giving the first claim.
   
 
From now on we assume that $\xi \in A_D(\kk)$, and form the fiber product
\begin{equation}\label{xi fiber}
\xymatrix{
{  N_{(d_1,d_2) ,\xi } } \ar[r]\ar[d] &  {  N_{(d_1,d_2)}  } \ar[d]^\beta \\
{ \Spec(\kk) }  \ar[r]_{\xi} & { A_d .}
}
\end{equation}
It follows from what was said above that
\[
(E_1,E_2,E_3)  \mapsto (\co_X( -E_1) , \co_X(-E_2) , \co_{Y_3}(-E_3) , \phi_\gamma )  
\]
defines a function
\begin{equation}\label{fiber param}
 \mathfrak{N}_{D,\gamma} \to \bigsqcup_{ \substack{  d_1,d_2 \ge 0 \\   d_1+d_2=2d} }  N_{(d_1,d_2) ,\xi } (\kk) .
\end{equation}

We interrupt the proof of Proposition \ref{prop:geometric orbital} for a lemma.

\begin{lemma}\label{lem:split fiber bijection}
The function (\ref{fiber param})  is a bijection.
\end{lemma}

\begin{proof}
To construct the inverse, start with a quadruple
\[
(\mathcal{L}_1,\mathcal{L}_2,\mathcal{L}_3,\phi) \in  \bigsqcup_{ \substack{  d_1,d_2 \ge 0 \\   d_1+d_2=2d} }  N_{(d_1,d_2) ,\xi } (\kk) .
\]

For $i\in \{1,2,3\}$, let  $V_i$ be the space of rational sections of $\mathcal{L}_i$, and set $V_0=V_1\oplus V_2$.  We view $V_0$ and $V_3$ as rank one modules over $K_0$ and $K_3$, respectively.  The morphism $\phi: f_{3*} \mathcal{L}_3 \to \mathcal{L}_1\oplus \mathcal{L}_2$ of vector bundles induces an $F$-linear isomorphism $\phi : V_3\to V_0$. 

Start by choosing any $K_0$-linear isomorphism $g_0 : V_0 \iso K_0$, and any $K_3$-linear isomorphism $g_3 : V_3 \iso K_0$.  These choices determine a bijection
\[
\mathrm{Iso}(V_3,V_0) \iso \mathrm{Iso}( K_3,K_0) ,
\]
and it is easy to see that the induced bijection
\begin{equation}\label{basis bijection}
K_0^\times \backslash \mathrm{Iso}(V_3,V_0) / K_3^\times \iso  K_0^\times \backslash \mathrm{Iso}( K_3,K_0)  / K_3^\times 
\end{equation}
does not depend on the initial choices of $g_0$ and $g_3$.  Let $\gamma'$ denote the image of $\phi$ under this bijection.

We claim that  $\gamma'=\gamma$.  This follows by unwinding the construction of the morphism $N_{(d_1,d_2)} \to A_d$, and comparing with the 
constructions of  \S \ref{ss:split coset invariant}.  As our initial triple $(\mathcal{L}_1,\mathcal{L}_2,\mathcal{L}_3,\phi)$ lies in the fiber (\ref{xi fiber}), the invariant
\[
K_0^\times \backslash \mathrm{Iso}( K_3,K_0)  / K_3^\times  \map{\inv}  K_3 
\]
of (\ref{new split invariant}) satisfies $\xi=\inv(\gamma')$.  On the other hand, we defined $\xi = \inv(\gamma)$.  
As we are assuming that $K_3$ is a field,  Proposition \ref{prop:split coset regular} implies that $\inv$ is a bijection, and so 
 $\gamma'=\gamma$.

The fact that (\ref{basis bijection}) identifies $\phi$ with $\gamma$ implies that we may choose the isomorphisms $g_0$ and $g_3$ so that the diagram 
\[
\xymatrix{
{ V_3  } \ar[r]^{\phi}  \ar[d]_{g_3}  &  {  V_0   }   \ar[d]^{g_0} \\
{  K_3  }  \ar[r]_{\gamma}   &  {  K_0   }
}
\]
commutes.  Moreover, this commutativity determines $(g_0,g_3)$ uniquely up to simultaneous rescaling by $F^\times$.  Indeed, any other pair making the diagram commute would have the form $(t_0 g_0 , t_3 g_3)$ for some $(t_0,t_3) \in K_0^\times \times K_3^\times$ satisfying 
$t_0^{-1} \gamma t_3=\gamma$.  The image of this pair under
\[
K_0^\times \times K_3^\times =\tilde{T}_0(F) \times \tilde{T}_3(F) \to T_0(F) \times T_3(F)
\]
lies in the subgroup $S_\gamma$ of (\ref{S def}), which is trivial by Lemma \ref{lem:S trivial}.  Thus $t_0,t_3\in F^\times$, and the condition 
$t_0^{-1} \gamma t_3=\gamma$ implies that $t_0=t_3$.

The $K_0$-linearity of $g_0$ allows us to rewrite the above diagram as
\[
\xymatrix{
{ V_3  } \ar[r]^{\phi}  \ar[d]_{g_3}  &  {  V_1 \oplus V_2   }   \ar[d]^{g_1 \oplus g_2} \\
{  K_3  }  \ar[r]_{\gamma}   &  {  F\oplus F   }
}
\]
for isomorphisms $g_1$ and $g_2$.  Let $(s_1,s_2,s_3)\in V_1\oplus V_2\oplus V_3$ be defined by $g_i(s_i)=1$.
These three vectors are (tautologically) rational sections of $\mathcal{L}_1$, $\mathcal{L}_2$, and $\mathcal{L}_3$, respectively, and we at last define
\[
(E_1,E_2,E_3) = (  - \mathrm{div}(s_1) ,  - \mathrm{div}(s_2),  - \mathrm{div}(s_3)) .
\]

We leave it as an exercise to the reader to verify that this construction is inverse to (\ref{fiber param}), completing the proof of the lemma.
\end{proof}

Now we complete the proof of Proposition \ref{prop:geometric orbital}.
First use  Lemma \ref{lem:S trivial} to rewrite
\begin{align*}
\J(\xi, f_D,s) 
& =  \int_{   ( T_0 \times T_3)(\A)} 
f_D(t_0^{-1} \gamma t_3 )\,  | t_0 |^{2s}  \eta(t_3) \, dt_0 \, dt_3  \\
& =  \int_{  \A^\times  \backslash (\tilde{T}_0 \times \tilde {T}_3)(\A)} 
\tilde{f}_D(t_0^{-1} \gamma t_3 )\,  | t_0 |^{2s} \eta(t_3) \, dt_0 \, dt_3 .
\end{align*}
Using (\ref{characteristic hecke}) this may be rewritten as
\[
\J(\xi, f_D,s) = \sum_{  (E_1,E_2,E_3) \in \mathfrak{N}_{D,\gamma}  } q^{-\deg(E_1-E_2) 2 s } \eta ( E_3 ),
\]
and Lemma \ref{lem:split fiber bijection} allows us to rewrite this as
\begin{align*}
\J(\xi, f_D,s) 
 &=   
\sum_{ \substack{  d_1,d_2 \ge 0 \\   d_1+d_2=2d \\  ( \mathcal{L}_1,\mathcal{L}_2, \mathcal{L}_3, \phi) \in N_{(d_1,d_2) ,\xi } (\kk)  } }   
q^{  ( 2 \deg(\mathcal{L}_1)  -  2 \deg(\mathcal{L}_2) )  s}\eta(\mathcal{L}_3 )\\
&=
 \sum_{ \substack{  d_1,d_2 \ge 0 \\   d_1+d_2=2d \\  ( \mathcal{L}_1,\mathcal{L}_2, \mathcal{L}_3, \phi) \in N_{(d_1,d_2) ,\xi } (\kk)  } }   
 q^{  (d_1-d_2) s }   \eta(\mathcal{L}_3 ).
\end{align*}
Combining the Grothendieck-Lefschetz trace formula with Proposition \ref{prop:etale fiber action} shows that
\[
 \sum_{( \mathcal{L}_1,\mathcal{L}_2, \mathcal{L}_3, \phi) \in N_{(d_1,d_2) ,\xi } (\kk)  } 
 \eta(\mathcal{L}_3 )
 =
 \mathrm{Trace} \Big(\Frob_\xi  ; \,   \big(\mathbf{R} \beta_* L_{ (d_1,d_2) } \big)_{\bar \xi}
 \Big).
\]
As the morphism $\beta$ is finite, the complex $\mathbf{R} \beta_* L_{ (d_1,d_2) }$ is supported in degree $0$,
completing the proof.
\end{proof}


\section{Intersection theory and moduli spaces of shtukas}
\label{s:intersection theory}


Fix an  integer $r\ge 0$, and an $r$-tuple $\mu=(\mu_1,\ldots, \mu_r) \in \{ \pm 1\}^r$ satisfying  the parity condition
$\sum_{i=1}^r \mu_i =0.$  In particular,  $r$ is even.


\subsection{Shtukas and Heegner-Drinfeld cycles}
\label{ss:cycles}


We rapidly recall some notation from \cite{YZ}.
Recall that $G_0\iso \PGL_{2/X}$.

Let  $\Bun_{G_0}$ be the Artin stack parametrizing $G_0$-torsors on $X$, and let 
 $\Hk_{G_0}^\mu$  be the Hecke stack parameterizing $G_0$-torsors on $X$ with $r$ modifications of type $\mu$.  
 It comes equipped with morphisms
 \[
 p_0,\ldots, p_r : \Hk_{G_0}^\mu \to \Bun_{G_0}
 \]
 and $p_X : \Hk_{G_0}^\mu \to  X^r$.  For the definitions, see \cite[\S 5.2]{YZ}.

Define the moduli stack of $G_0$-Shtukas of type $\mu$  by the cartesian diagram
\[
\xymatrix{\Sht^{\mu}_{G_0}\ar[d]\ar[rr] && \Hk^{\mu}_{G_0}\ar[d]^{(p_{0},p_{r})}\\
\Bun_{G_0}\ar[rr]^{(\mathrm{id},\Fr)} && \Bun_{G_0}\times \Bun_{G_0} .}
\]  
It is a Deligne-Mumford stack, locally of finite type over $\kk$, and the  morphism 
\[
\pi_{G_0} : \Sht_{G_0}^\mu \to X^r
\]
induced by $p_X$  is  separated and smooth of relative dimension $r$.

Recall that our two \'etale double covers $f_1:Y_1\to X$ and $f_2:Y_2\to X$ determine rank one tori $T_1$ and $T_2$ over $X$.
Fix $i \in \{1,2\}$ and let $\mathrm{Bun}_{T_i}$ be the moduli stack of $T_i$-torsors on $X$. 
Denote by   $\Hk_{T_i}^\mu$ the Hecke stack parameterizing $T_i$-torsors with $r$ modifications of type $\mu$.
It comes with morphisms
\[
p_1,\ldots , p_r:  \Hk_{T_i}^\mu \to \mathrm{Bun}_{T_i},
\]
and   $p_{Y_i} \colon \Hk_{T_i}^\mu \to Y_i^r$.
 See \cite[\S5.4]{YZ} for the definitions.

The stack of $T_i$-Shtukas of type $\mu$ is defined by the cartesian diagram
\begin{equation}\label{Sht_T definition}
\xymatrix{\Sht^{\mu}_{T_i}\ar[d]\ar[rr] && \Hk^{\mu}_{T_i}\ar[d]^{(p_{0},p_{r})}\\
\mathrm{Bun}_{T_i}\ar[rr]^{(\mathrm{id},\Fr)} && \mathrm{Bun}_{T_i}\times \mathrm{Bun}_{T_i} . }
\end{equation}
It is a smooth and proper Deligne-Mumford stack over $\kk$, and the morphism
\[
\pi_{T_i} : \Sht_{T_i}^\mu \to Y_i^r
\]
induced by $p_{Y_i}$ is  finite \'etale.  In particular, $\Sht_{T_i}^\mu$ is smooth and proper over $\kk$ of dimension $r$.

 The closed immersions of $T_1$ and $T_2$ into $G_0$ induce  finite  morphisms
\[
\xymatrix{
{   \Sht_{T_1}^\mu}  \ar[dr]_{\theta_1^\mu} &   &   {  \Sht_{T_2}^\mu  } \ar[dl]^{\theta_2^\mu}  \\
 & {  \Sht_{G_0}^\mu  } 
}
\]
which then induce push-forwards
\[
\xymatrix{
{  \Ch_r(  \Sht_{T_1}^\mu})  \ar[dr]_{\theta_{1*}^\mu} &   &   {  \Ch_r( \Sht_{T_2}^\mu ) } \ar[dl]^{\theta_{2*}^\mu}  \\
 & { \Ch_{c,r}( \Sht_{G_0}^\mu )  } 
}
\]
on Chow groups with $\Q$-coefficients \cite[Theorem 2.1.12]{Kresch-cycle}.  We obtain  cycle classes
\[
[\Sht_{T_1}^\mu] ,  [\Sht_{T_2}^\mu]  \in  \Ch_{c,r} (  \Sht_{G_0}^\mu )
\]
by pushing forward the fundamental classes.

As $\Sht_{G_0}^r$ has dimension $2r$, there is an intersection pairing
\[
\langle\cdot,\cdot\rangle :  \Ch_{c,r} (  \Sht_{G_0}^\mu ) \times  \Ch_{c,r} (  \Sht_{G_0}^\mu ) \to \Q
\]
as in  \cite[\S A.1]{YZ}.  Recall the Hecke algebra $\mathscr{H}$ of  \S \ref{ss:basic automorphic}.  For any $f\in\mathscr{H}$ define 
\begin{equation}\label{I distribution def}
\I_r(f) = \langle  [\Sht_{T_1}^\mu] ,  f * [\Sht_{T_2}^\mu]  \rangle \in \Q,
\end{equation}
where $*$ is the action of $\mathscr{H}$ on $\Ch_{c,r} (  \Sht_{G_0}^r )$ defined in   \cite[\S 5.3]{YZ}.

\begin{remark}
The isomorphism class of $\Hk_{G_0}^\mu$ is independent of $\mu$, and so we sometimes call this stack $\Hk_{G_0}^r$.  Similarly,  we sometimes write $\Sht_{G_0}^r$ instead of $\Sht_{G_0}^\mu$.   
\end{remark}


\subsection{Some moduli spaces}


Fix an integer $d$.   The purpose of this subsection is to construct a commutative diagram of $\kk$-schemes
\begin{equation}\label{fundamental 1}
\xymatrix{
 {  M_d  }   \ar[d]_\alpha \ar[r] &  { \Sigma_{2d}(Y) } \ar[d]^{\Nm}  \\
 {    A_d  }  \ar[d]_{\Tr}   \ar[r]_{f_3^\sharp} &    \Sigma_{2d}(Y_3)  \\
  {    \Sigma_d(X)  } 
}
\end{equation}
in such a way that the square is cartesian.    Let $S$ be any $\kk$-scheme.

Recall  from (\ref{fundamental 2}) that  $\Sigma_d(X)(S)$ is the set of  pairs $(\Delta, \zeta )$ consisting of  
\begin{itemize}
\item
a  line bundle $\Delta$ on $X_{S}$ of degree $d$,
\item
a nonzero section $\zeta \in H^0(X_{S} , \Delta)$.
\end{itemize}
(As in \S \ref{ss:analytic moduli}, these conditions are understood to hold fiber-by-fiber.)
The schemes   $\Sigma_{2d}(Y_3)$ and $\Sigma_{2d}(Y)$  are  defined similarly.
The morphism labeled $\Nm$ is  $(\Delta , \zeta) \mapsto (\Nm_{Y/Y_3} (\Delta) , \Nm_{Y/Y_3} ( \zeta ) ).$

Recall also that $A_d(S)$   is the set of all pairs  $(\Delta, \xi )$ consisting of  
\begin{itemize}
\item
a  line bundle $\Delta$ on $X_{S}$ of degree $d$,
\item
a $\xi \in H^0(Y_{3S} ,  f_3^*\Delta)$ with nonzero trace $\Tr_{Y_3/X}(\xi) \in H^0(X_S,\Delta)$.
\end{itemize}

Denote by  $\tilde{M}_d(S)$ the groupoid of triples $(\mathcal{L}_1,\mathcal{L}_2,\phi)$ consisting of
\begin{itemize}
\item
a line bundle  $\mathcal{L}_1 \in \Pic (Y_{1S})$,
\item
a line bundle $\mathcal{L}_2 \in\Pic( Y_{2S})$, 
\item
a morphism 
$
\phi :   f_{2*}\mathcal{L}_2 \to  f_{1*}\mathcal{L}_1 
$
 of rank two vector bundles on $X_S$.  
\end{itemize}
We require further that the line bundle
\[
\Delta = \underline{\Hom} ( \det(  f_{2*}\mathcal{L}_2 ) , \det(  f_{1*}\mathcal{L}_1 )  )
\]
on $X_S$ has degree $d$, and that $\det(\phi) \in H^0(X_S,\Delta)$ is nonzero.   The functor $\tilde{M}_d$ from $\kk$-schemes to groupoids is represented by an Artin stack over $\kk$.

The Picard group  $\Pic(X_S)$ acts on $\tilde{M}_d(S)$ by twisting
\[
(\mathcal{L}_1,\mathcal{L}_2,\phi) \otimes \mathcal{L} = (\mathcal{L}_1 \otimes f_1^*\mathcal{L} ,\mathcal{L}_2 \otimes f_2^*\mathcal{L} ,\phi \otimes \mathrm{id} ),
\]
 inducing an action of the Picard stack $\Pic_X$ on $\tilde{M}_d$.  The representability of the quotient stack
\[
M_d= \tilde{M}_d / \Pic_X
\]
 by a scheme is part of the following proposition, which also defines the two  arrows  in (\ref{fundamental 1}) emanating from $M_d$.

\begin{proposition}\label{prop:cartesian moduli}
There is a canonical isomorphism
\[
M_d \iso A_d \times_{ \Sigma_{2d}(Y_3) } \Sigma_{2d}(Y).
\]
\end{proposition}

\begin{proof}
Start with a $\kk$-scheme $S$ and a triple
$
(\mathcal{L}_1,\mathcal{L}_2,\phi) \in \tilde{M}_d(S).
$
Define line bundles  
\begin{equation}\label{L bundles}
\widetilde{\mathcal{L}}_1 = \mathcal{L}_1|_{Y_S} ,\quad \widetilde{\mathcal{L}}_2=\mathcal{L}_2|_{Y_S}
\end{equation}
on $Y_S$.   The pullback of $\phi$ via $Y_{S} \to X_{S} $ is a morphism
\[
\widetilde{\mathcal{L}}_2 \oplus \widetilde{\mathcal{L}}_2^{\tau_3} \iso  ( f_{2_*}\mathcal{L}_2)|_{Y_S} 
 \map{\phi |_{Y_S}}
(  f_{1_*}\mathcal{L}_1 ) |_{Y_S} \iso \widetilde{\mathcal{L}}_1 \oplus \widetilde{\mathcal{L}}_1^{\tau_3}
\]
of rank two vector bundles on $Y_S$,  encoded by four morphisms 
\begin{eqnarray}\label{matrix entries}
\widetilde{\mathcal{L}}_2 & \map{ \quad \bm{a}\quad } & \widetilde{\mathcal{L}}_1 \\
\widetilde{\mathcal{L}}_2^{\tau_3} \iso \widetilde{\mathcal{L}}_2^{\tau_1} & \map{\bm{b}=\bm{a}^{\tau_1}}  &  \widetilde{\mathcal{L}}_1^{\tau_1}\iso\widetilde{\mathcal{L}}_1 \nonumber  \\
\widetilde{\mathcal{L}}_2 \iso \widetilde{\mathcal{L}}_2^{\tau_2} & \map{\bm{c}=\bm{a}^{\tau_2}}  &  \widetilde{\mathcal{L}}_1^{\tau_2}\iso \widetilde{\mathcal{L}}_1^{\tau_3} \nonumber \\
 \widetilde{\mathcal{L}}_2^{\tau_3} & \map{\bm{d}=\bm{a}^{\tau_3}}  &  \widetilde{\mathcal{L}}_1^{\tau_3}.\nonumber
\end{eqnarray}
%
The assumption $\det(\phi)\neq 0$ implies that $\bm{a}\neq 0$, and setting  $\mathcal{K} = \underline{\Hom}(\widetilde{\mathcal{L}}_2,\widetilde{\mathcal{L}}_1)$ defines a point  $(\mathcal{K},\bm{a}) \in \Sigma_{2d}(Y)(S)$.  We have now constructed a morphism 
\[
\tilde{M}_d \to \Sigma_{2d}(Y),
\]
which is easily seen to descend to the quotient $M_d$.

Consider the map 
\[
\det( f_{2_*}\mathcal{L}_2)|_{Y_S} = 
\det( \widetilde{\mathcal{L}}_2 \oplus \widetilde{\mathcal{L}}_2^{\tau_3} )       \map{ \newdet(\phi) }
  \det( \widetilde{\mathcal{L}}_1 \oplus \widetilde{\mathcal{L}}_1^{\tau_3} ) =  \det(  f_{1_*}\mathcal{L}_1 )|_{Y_S} , 
\]
where the arrow labeled $\newdet(\phi)$ sends  $s \wedge  t \mapsto \bm{a}(s)\wedge \bm{d}(t)$ for local sections $s$ and $t$ of $\widetilde{\mathcal{L}}_2$ and $\widetilde{\mathcal{L}}_2^{\tau_3}$, respectively.  When viewed as a section of $\Delta|_{Y_S}$, this map is $\tau_3$-equivariant.  Hence it admits a canonical descent to 
\[
\newdet(\phi) \in H^0(Y_{3S} ,  f_3^*\Delta)
\]
whose trace $\Tr_{Y_3/X}(\newdet(\phi)) = \det(\phi)$ is nonzero.
Thus  $(\Delta,\newdet(\phi)) \in A_d(S)$, and we have have constructed a morphism 
\[
\tilde{M}_d\to A_d.
\]
Again, this is easily seen to descend to the quotient $M_d$.

The canonical isomorphism
\[
\mathcal{K} \otimes \mathcal{K}^{\tau_3}  \iso  \underline{\Hom}(\widetilde{\mathcal{L}}_2\otimes\widetilde{\mathcal{L}}_2^{\tau_3} , \widetilde{\mathcal{L}}_1\otimes\widetilde{\mathcal{L}}_1^{\tau_3} )  \iso \Delta|_{Y_S}  
\]
on $Y_S$ descends to  an isomorphism $\Nm_{Y/Y_3} ( \mathcal{K} ) \iso f_3^* \Delta$, and this isomorphism sends
$\Nm_{Y/Y_3}(\bm{a}) \mapsto  \newdet(\phi)$.  In other words,   
\[
  ( f_3^*\Delta , \newdet(\phi) )\iso \Nm( \mathcal{K} , \bm{a}) 
\]
define the same element of $\Sigma_{2d}(Y_3) (S)$, and so the two morphisms constructed above  define a map
\begin{equation}\label{cartesian moduli map}
M_d \to  A_d \times_{ \Sigma_{2d}(Y_3) } \Sigma_{2d}(Y).
\end{equation}

We will show that (\ref{cartesian moduli map}) is an isomorphism by constructing the inverse.    An $S$-point of  $A_d \times_{ \Sigma_{2d}(Y_3) } \Sigma_{2d}(Y)$  consists of two pairs
\[
(\Delta , \zeta)\in A_d(S) ,\quad (\mathcal{K},\bm{a}) \in \Sigma_{2d}(Y)(S)
\]
along with an isomorphism  $\Nm_{Y/Y_3} ( \mathcal{K} ) \iso f_3^* \Delta$ satisfying
$\Nm_{Y/Y_3}(\bm{a}) \mapsto  \zeta$.

The isomorphism $\Nm_{Y/Y_3}(\mathcal{K}) \iso  f_3^* \Delta$
endows the line bundle $\mathcal{K} \otimes \mathcal{K}^{\tau_3}$  with descent data relative to $Y_S/X_S$.  
In other words, we are given isomorphisms between this line bundle and all of its $\Aut(Y/X)$-conjugates, and hence an isomorphism
$
\mathcal{K} \otimes \mathcal{K}^{\tau_3}  \iso \mathcal{K}^{\tau_1} \otimes \mathcal{K}^{\tau_2}.
$
This induces the first isomorphism in 
\[
\underline{\Hom}( \mathcal{K} , \mathcal{K}^{\tau_1} ) \iso \underline{\Hom}( \mathcal{K}^{\tau_2} , \mathcal{K}^{\tau_3} )
\iso \underline{\Hom}( \mathcal{K} , \mathcal{K}^{\tau_1} )^{\tau_2},
\]
and, by viewing the composition as descent data relative to $Y_S/Y_{2S}$, we obtain a degree $0$ line bundle $\mathcal{M}$ on $Y_{2S}$ endowed with an isomorphism 
\[
 \mathcal{M}|_{ Y_S} \iso  \underline{\Hom}(\mathcal{K},\mathcal{K}^{\tau_1}).
\]

The canonical trivialization 
\begin{align*}
( \mathcal{M} \otimes \mathcal{M}^{\sigma_2} )  |_{Y_S}    
&\iso 
 \underline{\Hom}(\mathcal{K},\mathcal{K}^{\tau_1}) 
\otimes  \underline{\Hom}(\mathcal{K},\mathcal{K}^{\tau_1})^{\tau_3} \\
& \iso 
 \underline{\Hom}(\mathcal{K} \otimes \mathcal{K}^{\tau_3},\mathcal{K}^{\tau_1} \otimes \mathcal{K}^{\tau_2}) \\
& \iso \co_Y
\end{align*}
is compatible with the natural descent data relative to $Y_S/X_S$ on the source and target, and so the line bundle $\Nm_{Y_2/X}( \mathcal{M})$
is trivial.  As  in the proof of \cite[Proposition 6.1(1)]{YZ}, this implies the existence of an $S$-point $\mathcal{L}_2$ of the quotient $\Pic_{Y_2}/\Pic_X$  satisfying
\[
 \mathcal{L}_2   \iso \mathcal{M} \otimes \mathcal{L}_2^{\sigma_2}.
\]
Viewing the isomorphism 
\begin{align*}
( \mathcal{K} \otimes \mathcal{L}_2|_{Y_S} )^{\tau_1}   
&  \iso  \mathcal{K}^{\tau_1} \otimes \mathcal{L}_2^{\sigma_2}|_{Y_S}  \\
&  \iso  \mathcal{K}^{\tau_1} \otimes    ( \mathcal{M}^{-1} \otimes \mathcal{L}_2)  |_{Y_S}  \\
& \iso \mathcal{K}^{\tau_1} \otimes  \underline{\Hom}(\mathcal{K} ^{\tau_1} ,\mathcal{K} ) \otimes \mathcal{L}_2|_{Y_S}\\
& \iso \mathcal{K}  \otimes \mathcal{L}_2|_{Y_S}
\end{align*}
as descent data relative to $Y_S/Y_{1S}$, we obtain a line bundle $\mathcal{L}_1$ on $Y_{1S}$ endowed with an isomorphism
$\mathcal{L}_1|_{Y_S} \iso \mathcal{K}  \otimes \mathcal{L}_2|_{Y_S}$.
If we set  $\widetilde{\mathcal{L}}_1=\mathcal{L}_1|_{Y_S}$ and $\widetilde{\mathcal{L}}_2=\mathcal{L}_2|_{Y_S}$  as in  (\ref{L bundles}),  this isomorphism can be rewritten as 
\begin{equation}\label{switcharoo bundle}
 \mathcal{K} \iso \underline{\Hom}( \widetilde{\mathcal{L}}_2 , \widetilde{\mathcal{L}}_1 )  .
\end{equation}

Now view $\bm{a}$ as a global section of (\ref{switcharoo bundle}), 
and  define global sections  $\bm{b}$, $\bm{c}$, and $\bm{d}$  using (\ref{matrix entries}).  These four global sections 
define a global section $\phi$ of 
\[
\underline{\Hom}( \widetilde{\mathcal{L}}_2 \oplus \widetilde{\mathcal{L}}_2^{\tau_3}  , \widetilde{\mathcal{L}}_1 \oplus \widetilde{\mathcal{L}}_1^{\tau_3} )
\iso  \underline{\Hom}(  f_{2*} \mathcal{L}_2 , f_{1*}\mathcal{L}_1)|_{Y_S},
\]
which,  by construction, is invariant under $\Aut(Y/X)$.  Thus $\phi$ descends to 
\[
\phi\in \Hom(  f_{2*} \mathcal{L}_2 , f_{1*}\mathcal{L}_1).
\]
The triple $(\mathcal{L}_1,\mathcal{L}_2,\phi)$ defines an object of $M_d(S)$, completing the construction of the inverse of (\ref{cartesian moduli map}).
\end{proof}

\begin{proposition}\label{prop:Msmooth}
Let $g$ and $g_3$ be the genera of $X$ and $Y_3$.
\begin{enumerate}
\item
The morphisms $\alpha$ and $\Nm$ in (\ref{fundamental 1}) are finite.
\item
If $d \geq 2g_3 - 1$ then $M_d$ is smooth over $\kk$ of dimension $2d - g + 1$.
\end{enumerate}
\end{proposition}

\begin{proof}
The map $\Nm$ in (\ref{fundamental 1}) is clearly finite, and therefore  $\alpha$ is as well.  This proves the first claim.
For the proof of the second claim, recall the cartesian diagram
\[
\xymatrix{
{  M_d} \ar[r]^\pi \ar[d]_\rho  &  {  \Sigma_{2d}(Y)   }  \ar[d]^{\Nm} \\
{  \Pic^d_X } \ar[r]_{f_3^*}  & { \Pic^{2d}_{Y_3} }  
}
\]
of Proposition \ref{prop:cartesian moduli}.
In the notation used there, the map $\pi$ sends $(\mathcal{L}_1,\mathcal{L}_2,\phi)$ to $(\mathcal{K},\bm{a})$, and $\rho$ sends the same data to $\Delta$.

The vertical arrow on the right factors as
\[
\Sigma_{2d}(Y) \map{\mathrm{AJ}}  \Pic^{2d}_Y \map{\Nm}  \Pic^{2d}_{Y_3}.
\]
Letting $g_Y$ denote the genus of $Y$,  the Abel-Jacobi morphism $\mathrm{AJ}$ is smooth by our hypothesis 
\[
2d \geq 2g_Y - 1 = 2(2g_3 - 1) - 1.
\]
The norm map $\Pic^{2d}_Y \to \Pic^{2d}_{Y_3}$ is smooth by the following lemma.

\begin{lemma}
The norm $\Nm\colon \Pic_Y \to \Pic_{Y_3}$ is a smooth morphism.  
\end{lemma}

\begin{proof}
We use the infinitesimal lifting criterion for smoothness.  Suppose $A \to B$ is a surjection of local Artinian $k$-algebras with kernel $I$ satisfying $I^2 = 0$.  Suppose also that we have morphisms $\alpha$ and $\beta$ making the square
\[
\xymatrix{
{  \Spec (B)} \ar[r]^\alpha \ar[d]  &  {  \Pic_Y   }  \ar[d]^\Nm \\
{  \Spec (A) } \ar@{-->}[ur]^{\gamma} \ar[r]_{\beta}  & { \Pic_{Y_3} }
}
\]
commute.  We must prove the existence of a morphism $\gamma$ making the two triangles commute.

Let $g : Y\to Y_3$ be the \'etale double cover of (\ref{biquad curves}).  The  trace morphism $\Tr: g_*\co_Y \to \co_{Y_3}$ induces a short exact sequence 
\[
\xymatrix{
{  0   }  \ar[r]    &  {  \ker ( \mathrm{id} \otimes \Tr) } \ar[r]  &  {   I \otimes_B g_*\co_{Y_B}  } \ar[r]^{ \mathrm{id}\otimes   \Tr}  &    { I \otimes_B \co_{Y_{3B}}  }  \ar[r]   & { 0 }.
}
\]  
of coherent sheaves on $Y_{3B}$.  
As $Y_{3B}$ is a curve, the induced map
\[
 H^1(Y_B, I \otimes_B \co_{Y_B})  = H^1(Y_{3B}, I \otimes_B g_* \co_{Y_B})    \map{\mathrm{id}\otimes\Tr}    H^1(Y_{3B}, I \otimes_B \co_{Y_{3B}}) 
\]
on cohomology is surjective.

As the closed immersion $Y_B \hookrightarrow Y_A$ is an isomorphism on the underlying topological spaces, the category of sheaves on these two spaces are canonically identified.  Thus we have an exact sequence of sheaves 
\[
\xymatrix{
{  0 } \ar[r]  &  {  I \otimes_B \mathcal{O}_{Y_B} } \ar[r]^{\quad j}  &  {    \co_{Y_A}^\times   }   \ar[r]   &  {    \co_{Y_B}^\times  } \ar[r]  &  {  1 } 
}
\]
on $Y_B$,  where $j(r \otimes f) = 1 + rf$, and a similar exact sequence on $Y_{3A}$.  
Taking cohomology yields a commutative diagram
\[
\xymatrix{
{  H^1(Y_B, I \otimes_B \co_{Y_B})} \ar[r]  \ar[d]^{ \mathrm{id} \otimes \Tr} & {  \Pic(Y_A) } \ar[d]^{\Nm}  \ar[r]   &  {   \Pic(Y_B) } \ar[d]^{ \Nm } \ar[r] & { 0} \\
{  H^1(Y_{3B}, I \otimes_B \co_{Y_{3B}})}  \ar[r] \ar[d] & { \Pic(Y_{3A}) }   \ar[r] &  {   \Pic(Y_{3B}) } \ar[r] & {0} \\
{ 0} 
}
\] 
with exact rows and columns.
Note that the surjectivity of $\Pic(Y_A)\to \Pic(Y_B)$ follows from the smoothness of the Picard stack $\Pic_Y$ over $\kk$, and similarly with $Y$ replaced by $Y_3$.

The maps $\alpha$ and $\beta$ determine elements of $\Pic(Y_B)$ and $\Pic(Y_{3A})$, having the same image in $\Pic(Y_{3B})$.  
A  diagram chase shows that they  come from a common element of $\Pic(Y_A)$, which is the desired $\gamma$.
\end{proof}

We have now shown that  $\Nm: \Sigma_{2d}(Y) \to \Pic_{Y_3}^{2d}$ is a smooth morphism between stacks of dimension $2d$ and $g_3-1$, and hence $\rho:M_d \to \Pic_X^d$ is smooth of relative dimension 
\[
2d- g_3 + 1 = 2d-2g +2.
\]
As $\Pic_X^d$ is smooth over $\kk$ of dimension $g-1$, $M_d$ is smooth and of dimension $2d-g+1$ over $\kk$.
\end{proof}


\subsection{Interpretation of the intersection number}
\label{ss:geometric intersection}


We define a  correspondence 
\begin{equation}\label{hecke one paw}
\xymatrix{
& { \Hk_{M_d} }  \ar[dr]^{\gamma_1} \ar[dl]_{\gamma_0} \\
{  M_d  }  \ar[dr]_{ \alpha }  & &  {   M_d    }  \ar[dl]^{  \alpha } \\
& { A_d } 
}
\end{equation}
as follows.  We first define a stack $\widetilde{\Hk}_{M_d}$ whose $S$-points classify 
\begin{itemize}
\item 
two points $(\mathcal{L}^{(0)}_1,\mathcal{L}_2^{(0)},\phi^{(0)})$ and  $(\mathcal{L}_1^{(1)},\mathcal{L}_2^{(1)}, \phi^{(1)})$ of $\tilde M_{d}(S)$,
\item
one  $S$-point 
$y = (y_1,y_2)$ of  $Y_S=Y_{1S}\times_{X_S} Y_{2S}$,
\item
 injective $\mathcal{O}_X$-linear morphisms  $s_1 : \mathcal{L}^{(0)}_1  \to  \mathcal{L}^{(1)}_1$ and   $s_2 : \mathcal{L}^{(0)}_2  \to  \mathcal{L}^{(1)}_2$.
\end{itemize}
We require that the cokernels of $s_1$ and $s_2$ are invertible sheaves on the graph of $y_1$ and $y_2$, respectively, and that the diagram
\[
\xymatrix{
{   f_{2*} \mathcal{L}^{(0)}_2 } \ar[r]^{s_2}     \ar[d]_{\phi^{(0)}}  &  {   f_{2*} \mathcal{L}^{(1)}_2 }   \ar[d]^{\phi^{(1)}}     \\
  {   f_{1*} \mathcal{L}^{(0)}_1  }  \ar[r]_{ s_1}    &   {   f_{1*} \mathcal{L}^{(1)}_1  } 
}
\]
of $\co_{X_S}$-modules commutes.  Then $\Pic_X$ acts on $\widetilde\Hk_{M_d}$ by simultaneously twisting the $\mathcal{L}_i^{(j)}$, and we define $\Hk_{M_d} = \widetilde\Hk_{M_d}/\Pic_X$.

Using the top horizontal arrow of (\ref{fundamental 1}),  we may realize the above correspondence on $M_d$  as the pullback of a correspondence on $\Sigma_{2d}(Y)$.  More precisely, there is a correspondence 
\begin{equation}\label{simple hecke}
\xymatrix{
& { H_{2d} (Y) }  \ar[dr] \ar[dl] \\
{  \Sigma_{2d}(Y)  }  \ar[dr]_{ \Nm}  & &  {   \Sigma_{2d}(Y)  }  \ar[dl]^{  \Nm} \\
& { \Sigma_{2d}(Y_3) ,} 
}
\end{equation}
where, for any $\kk$-scheme $S$,  the groupoid $H_{2d}(Y)(S)$ classifies 
\begin{itemize}
\item
a pair of $S$-points $(\mathcal{K}^{(0)} , \bm{a}^{(0)} )$ and $(\mathcal{K}^{(1)} , \bm{a}^{(1)} )$ of $ \Sigma_{2d}(Y)$,
\item
 an $S$-point $y\in Y(S)$, 
 \item
  an isomorphism
\[
s : \mathcal{K}^{(0)} ( y^{\tau_1} - y^{\tau_2} ) \iso \mathcal{K}^{(1)} 
\]
of line bundles on $Y_S$ such that  $s(\bm{a}^{(0)}) = \bm{a}^{(1)}$, where we view the global section $\bm{a}^{(0)}$ of $\mathcal{K}^{(0)}$ as a rational section of  $\mathcal{K}^{(0)} ( y^{\tau_1} - y^{\tau_2} )$.  
\end{itemize}

Note that all of the data is determined by the pair $(\mathcal{K}^{(0)} ,\bm{a}^{(0)})$ and the point $y\in Y(S)$, for from this we may recover the line bundle $\mathcal{K}^{(1)} = \mathcal{K}^{(0)} ( y^{\tau_1} - y^{\tau_2} )$ and its rational section $\bm{a}^{(1)}=\bm{a}^{(0)}$.  The condition that  $\bm{a}^{(1)}$ be a section of $\mathcal{K}^{(1)}$, as opposed to merely a rational section, is equivalent to 
\[
\mathrm{div}(\bm{a}^{(0)} ) + y^{\tau_1} - y^{\tau_2} \ge 0.
\]
This is in turn equivalent to the condition that the effective Cartier divisor $y^{\tau_2}$ appears in the support of $\mathrm{div}(\bm{a}^{(0)})$.
In other words,  we may realize
\[
H_{2d}(Y) \hookrightarrow \Sigma_{2d}(Y) \times_k Y
\]
as the closed subscheme of triples $(\mathcal{K}^{(0)} , \bm{a}^{(0) } , y )$ for which $y^{\tau_2}$ appears in the support of $\mathrm{div}(\bm{a}^{(0)})$.

\begin{proposition}\label{prop:rep one paw}
The diagram (\ref{hecke one paw}) is canonically identified with the diagram
\[
\xymatrix{
& { A_d \times_{  \Sigma_{2d}(Y_3) }   H_{2d} (Y) }  \ar[dr] \ar[dl] \\
{  A_d \times_{  \Sigma_{2d}(Y_3) }  \Sigma_{2d}(Y)  }  \ar[dr]  & &  { A_d \times_{  \Sigma_{2d}(Y_3) }  \Sigma_{2d}(Y)    }  \ar[dl]\\
& {  A_d } 
}
\]
obtained from (\ref{simple hecke}) by base change along the  arrow $A_d \to  \Sigma_{2d}(Y_3)$  in  (\ref{fundamental 1}).
\end{proposition}
\begin{proof}
The proof amounts to carefully tracing through the constructions in the proof of Proposition \ref{prop:cartesian moduli}. 

It is enough to show that $\Hk_{M_d} \cong A_d \times_{  \Sigma_{2d}(Y_3) }   H_{2d} (Y)$.  We define a map 
\[
\Hk_{M_d} \to A_d \times_{  \Sigma_{2d}(Y_3) }  H_{2d}(Y)
\]
  as follows.  Given a quintuple
\[((\mathcal{L}^{(0)}_1,\mathcal{L}_2^{(0)},\phi^{(0)}),( \mathcal{L}_1^{(1)},\mathcal{L}_2^{(1)}, \phi^{(1)}), (y_1,y_2), s_1, s_2) \in \widetilde\Hk_{M_d}(S),\] we obtain from the first two pieces of data, points $(\mathcal{K}^{(0)} , \bm{a}^{(0)} )$ and $(\mathcal{K}^{(1)} , \bm{a}^{(1)} )$ of $ \Sigma_{2d}(Y)(S)$.  
For $i = 1,2$, we have $\mathcal{K}^{(i)} \cong \Hom(\mathcal{L}_2^{(i)}|_{Y_S}, \mathcal{L}_1^{(i)}|_{Y_S})$, and so the isomorphisms $ \mathcal{L}_1^{(0)}(y_1) \cong \mathcal{L}_1^{(1)}$ and $ \mathcal{L}_2^{(0)}(y_2) \cong \mathcal{L}_2^{(1)}$ induce an isomorphism $ \mathcal{K}^{(0)}(y^{\tau_1} - y^{\tau_2})\cong \mathcal{K}^{(1)}$ sending $\bm{a}^{(0)}$ to $\bm{a}^{(1)}$.  
This gives a map 
\[
\widetilde\Hk_{M_d} \to A_d \times_{  \Sigma_{2d}(Y_3) }   H_{2d} (Y),
\]
 which factors through $\Hk_{M_d}$.  

To construct a map in the other direction,  suppose given an $S$-point 
\[(\Delta, \xi, \mathcal{K}^{(0)}, \bm{a}^{(0)},\mathcal{K}^{(1)}, \bm{a}^{(1)},y,s) \in A_d \times_{  \Sigma_{2d}(Y_3) } H_{2d} (Y).\]
The proof of Proposition \ref{prop:cartesian moduli} constructs a point $(\mathcal{L}^{(i)}_1, \mathcal{L}_2^{(i)}, \phi^{(i)}) \in M_d(S)$ corresponding to $(\Delta, \xi, \mathcal{K}^{(i)}, \bm{a}^{(i)})$, and we must show that that there are isomorphisms $\mathcal{L}_1^{(1)} \cong \mathcal{L}_1^{(0)}(y_1)$ and $\mathcal{L}_2^{(1)} \cong \mathcal{L}_2^{(0)}(y_2)$ inducing the given isomorphism 
\[
s : \mathcal{K}^{(0)} ( y^{\tau_1} - y^{\tau_2} ) \iso \mathcal{K}^{(1)}.
\]
The line bundle $\mathcal{M}^{(i)} \in \Pic(Y_{2S})$, used to construct $\mathcal{L}_2^{(i)}$, is a descent of $\underline\Hom(\mathcal{K}^{(i)}, \mathcal{K}^{(i)\tau_1})$ via the isomorphism $\Nm_{Y/Y_3}(\mathcal{K}^{(i)}) \cong f_3^*\Delta$.  The isomorphism $s$ therefore induces a canonical isomorphism 
\begin{equation}\label{Ms}
\mathcal{M}^{(1)} \cong \mathcal{M}^{(0)}(y_2 - y_2^{\sigma_2}).
\end{equation} 
By construction, $\mathcal{L}_2^{(i)}$ is the unique $S$-point of  $\Pic_{Y_2}/\Pic_X$ such that 
\[
\mathcal{L}_2^{(i)} \cong \mathcal{M}^{(i)} \otimes \mathcal{L}_2^{(i)\sigma_2}.
\]
  It then follows from (\ref{Ms}) that $\mathcal{L}_2^{(0)}(y_2) \cong \mathcal{L}_2^{(1)}$.  One constructs the isomorphism  $\mathcal{L}_1^{(1)} \cong \mathcal{L}_1^{(0)}(y_1)$ in a similar fashion.            
\end{proof}

\begin{corollary}\label{prop:hk_d facts}
The stack $\Hk_{M_d}$ is a scheme, and $\gamma_0, \gamma_1 \colon \Hk_{M_d} \to M_d$ are finite and surjective. 
In particular, by Proposition $\ref{prop:Msmooth}$,  if $d \geq 2g_3 - 1$ then $\dim \Hk_{M_d} = 2d - g + 1$.  
\end{corollary}

The correspondence  $\Hk_{M_d}$ induces an endomorphism 
\[
[\Hk_{M_d}] :  \alpha_*\Q_\ell \to \alpha_*\Q_\ell
\] 
of sheaves on $A_d$, given by the composition
\[\alpha_*\Q_\ell \to \alpha_*\gamma_{0*}\gamma_0^*\Q_\ell \simeq \alpha_*\gamma_{0*}\Q_\ell \simeq \alpha_*\gamma_{1*}\Q_\ell \to \alpha_*\Q_\ell.\]
The first and last maps are induced by adjunction, using that $\gamma_0$ and $\gamma_1$ are finite.  Denote by $[\Hk_{M_d}]^r$ the $r$-fold composition of  this endomorphism with itself.  
 The remainder of  \S \ref{s:intersection theory} is devoted to the proof of the following proposition.


 \begin{proposition}\label{prop:geometric trace}
 Fix an effective divisor $D\in \Div(X)$ of degree $d \geq 2g_3 - 1$, and recall the closed subscheme $A_D \subset A_d$ and the inclusion $A_D(\kk) \subset K_3$ of Proposition $\ref{prop:invariant domain}$.  The intersection multiplicity (\ref{I distribution def}) satisfies
 \[
 \I_r(f_D) = \sum_{ \substack{  \xi \in K_3  \\  \Tr_{K_3/F}(\xi)=1   } }\I_r(\xi , f_D),
 \]
 where 
\[
\I_r(\xi , f_D)= 
\begin{cases}
\mathrm{Trace} \big(      [\Hk_{M_d}]_{\bar{\xi}} ^r  \circ \mathrm{Frob}_\xi   ;\,   ( \alpha_* \Q_\ell)_{\bar{\xi}}  \big)  
& \mbox{if } \xi\in A_D(\kk)\\
0 & \mbox{otherwise.}
\end{cases}
\]
Here $\bar{\xi}$ is any geometric point above $\xi:\Spec(\kk) \to A_D$.
 \end{proposition}


\subsection{Correspondences with multiple paws}

As a first step toward proving Proposition \ref{prop:geometric trace}, we want to interpret the $r$-fold iterated endomorphism 
$[\Hk_{M_d}]^r$ as the endomorphism associated with a single correspondence.

 To this end, we define a stack $\Hk_{M_d}^\mu$, sitting in a commutative diagram
\begin{equation}\label{hecke}
\xymatrix{
& { \Hk^\mu_{M_d} }  \ar[dr]^{\gamma_r} \ar[dl]_{\gamma_0}  \\
{  M_d  }  \ar[dr]_{ \alpha }  & &  {   M_d    }  \ar[dl]^{  \alpha } \\
& { A_d } 
}
\end{equation}
First define a stack $\widetilde\Hk^\mu_{M_d}$ whose $S$-points are given by: 
\begin{itemize}
\item For each $0 \leq i \leq r$, points $(\mathcal{L}_1^{(i)}, \mathcal{L}_2^{(i)}, \phi_i) \in \widetilde M_d(S)$.
\item For each $1 \leq i \leq r$, points $y^{(i)} = (y^{(i)}_1, y^{(i)}_2) \in Y(S) = (Y_1 \times_X Y_2)(S)$.
\item For each $1 \leq i \leq r$, rational maps 
\[s_1^{(i)} \colon \mathcal{L}_1^{(i-1)} \dashrightarrow \mathcal{L}_1^{(i)} \hspace{2mm} \mbox{ and } \hspace{2mm} s_2^{(i)} \colon \mathcal{L}_2^{(i-1)} \dashrightarrow \mathcal{L}_2^{(i)},\]
such that $(\mathcal{L}_1^{(i)}, s_1^{(i)}, y^{(i)}_1)$ and $(\mathcal{L}_2^{(i)}, s_2^{(i)}, y^{(i)}_2)$ give points of $\Hk_{T_1}^\mu(S)$ and $\Hk_{T_2}^\mu(S)$, respectively, and such that the following diagram commutes  
\begin{equation}\label{Hk^mu_d}
\xymatrix{
f_{1*}\mathcal{L}_{1}^{(0)}\ar[d]^{\phi_0}\ar@{-->}[r]^{s_1^{(1)}} 
& f_{1*}\mathcal{L}_{1}^{(1)}\ar[d]^{\phi_{1}} \ar@{-->}[r]^{s_1^{(2)}} 
& \cdots\ar@{-->}[r]^{ s_1^{(r)}} & f_{1*}\mathcal{L}_{1}^{(r)}\ar[d]^{\phi_{r}}\\
f_{2*}\mathcal{L}_{2}^{(0)}\ar@{-->}[r]^{s_2^{(1)}} 
& f_{2*}\mathcal{L}_{2}^{(2)}\ar@{-->}[r]^{ s_2^{(2)}} 
& \cdots\ar@{-->}[r]^{ s_1^{(r)}} & f_{2*}\mathcal{L}_{2}^{(r)}
}
\end{equation}

\end{itemize}
We then set 
\[
\Hk_{M_d}^\mu = \widetilde\Hk_{M_d}^\mu/\Pic_X.
\]  
For $0 \leq i \leq r$, we have morphisms $\gamma_i \colon \Hk_{M_d}^\mu \to M_d$ which remember the $i$th column in (\ref{Hk^mu_d}), and this gives the diagram (\ref{hecke}).   

Exactly as in \cite[Lemma 6.2]{YZ},  there is an isomorphism of $(M_d \times_{A_d} M_d)$-schemes
\begin{equation}\label{Hecke mu}
\Hk_{M_d}^\mu \iso  \underbrace{\Hk_{M_d} \times_{\gamma_1, \gamma_0} \Hk_{M_d} \times_{\gamma_1, \gamma_0} \cdots \times_{\gamma_1,\gamma_0} \Hk_{M_d}}_{r \, \mathrm{times}},
\end{equation}
where the fiber products are  with respect to the morphisms of (\ref{hecke one paw}).  
\begin{corollary}
If $d \geq 2g_3 - 1$, then $\dim \Hk_{M_d}^\mu = \dim \Hk_{M_d} = 2d - g + 1.$
\end{corollary}


Define a $\kk$-scheme $\Sht_{M_d}^\mu$ as the fiber product  
\begin{equation}\label{sht_d}
\xymatrix{
{  \Sht_{M_d}^\mu} \ar[rr]\ar[d]  & &  {  \Hk_{M_d}^\mu   }  \ar[d]^{(\gamma_0, \gamma_r)} \\
{  M_d } \ar[rr]_{(\mathrm{id}, \mathrm{Fr}_{M_d})}  &  & { M_d \times M_d }  .
}
\end{equation}

\begin{proposition}
The scheme $\Sht_{M_d}^\mu$ has dimension $0$, and the image of the composition
\[
\Sht_{M_d}^\mu(\kk^{\mathrm{alg}}) \to M_d(\kk^{\mathrm{alg}}) \map{\alpha} A_d(\kk^{\mathrm{alg}})
\]
is a subset of $A_d(\kk)$.
\end{proposition}
\begin{proof}
 The map $(\gamma_0, \gamma_r)$ is finite, by Proposition \ref{prop:hk_d facts} and (\ref{Hecke mu}),   and hence 
so is $\Sht_{M_d}^\mu \to M_d$.
From the cartesian diagrams (\ref{hecke}) and (\ref{sht_d}) we see that any $\kk^{\mathrm{alg}}$-point of $\Sht_{M_d}^\mu$ lies over a $\kk$-point of $A_d$.  As $\alpha : M_d \to A_d$ is finite, it follows that  $\Sht_{M_d}^\mu \to M_d$  factors through a 0-dimensional closed subscheme of $M_d$.  
All parts of the proposition follow from this.
\end{proof}

For any $\xi \in A_d(\kk)$ we form the fiber product
\[
\xymatrix{
{  \Sht_{M_d}^\mu(\xi) }  \ar[rr] \ar[d]  &  &  {   \Spec(\kk)  }    \ar[d]^{\xi}     \\
 {  \Sht_{M_d}^\mu  }   \ar[r]    & {  M_d  }   \ar[r]_{\alpha}      &   {  A_d ,}  
}
\]
and so obtain  a  decomposition
\[
\Sht_{M_d}^\mu = \bigsqcup_{\xi \in A_d(\kk)} \Sht_{M_d}^\mu(\xi)
\] 
into finitely many open and closed $0$-dimensional subschemes.  On the level of point sets there is a decomposition
$A_d(\kk) = \bigsqcup_D A_D(\kk)$,  where the disjoint union runs over all   effective degree $d$ divisors  
\[
D \in  \mathrm{Sym}^d(X) \iso \Sigma_d(X)(\kk) ,
\]
and $A_D$ is as in \S \ref{ss:geometric orbital}. Setting 
\begin{equation}\label{ShtD}
\Sht_{M_D}^\mu = \bigsqcup_{\xi \in A_D(\kk)} \Sht_{M_d}^\mu(\xi),
\end{equation}
we obtain a decomposition of the Chow group
\begin{equation}\label{Ddecomp}
\Ch_0(\Sht_{M_d}^\mu) = \bigoplus_{D \in \Sigma_d(X)(\kk)} \Ch_0(\Sht_{M_d}^\mu).
\end{equation}


\subsection{The refined Gysin map}
 

As $M_d$ is smooth (by Proposition \ref{prop:Msmooth}), the morphism  $(\mathrm{id}, \mathrm{Fr}_{M_d})$  of (\ref{sht_d}) is a regular local immersion.  We therefore have a refined Gysin map
\begin{equation}\label{basic gysin} 
(\mathrm{id}, \mathrm{Fr}_{M_d})^! \colon \Ch_{2d-g+1}(\Hk_{M_d}^\mu) \to \Ch_0(\Sht_{M_d}^\mu)
\end{equation}
defined as in \cite[\S3.1]{Kresch-cycle}.

\begin{proposition}\label{prop:fundamental}
Suppose $D$ is an effective divisor on $X$ of degree $d \geq 2g_3 - 1$.
The composition
 \[
 \Ch_{2d-g+1}(\Hk_{M_d}^\mu) \map{ (\ref{basic gysin}) }  \Ch_0(\Sht_{M_d}^\mu)
 \map{(\ref{Ddecomp})}   \Ch_0(\Sht_{M_d}^\mu) \map{\deg}\Q
\]
sends the fundamental class $ [\Hk_{M_d}^\mu] \in \Ch_{2d-g+1}(\Hk_{M_d}^\mu)$ to 
the intersection multiplicity  $\I_f(f_D)$ defined by (\ref{I distribution def}).
  \end{proposition}

\begin{proof}
 As in \cite[6.3]{YZ} we consider an octahedral diagram:
\begin{equation}\label{octahedron}
\xymatrix{\Hk^{\mu}_{T_1}\times\Hk^{\mu}_{T_2}\ar@<-3ex>[d]_{(\gamma_{0},\gamma_{r})}\ar@<3ex>[d]^{(\gamma_{0},\gamma_{r})}\ar[rr]^{\Pi_1^{\mu}\times\Pi_2^{\mu}} && \Hk^{r}_{G_0}\times\Hk^{r}_{G_0} \ar@<-3ex>[d]_{(\gamma_{0},\gamma_{r})}\ar@<3ex>[d]^{(\gamma_{0},\gamma_{r})} && \Hk^{r}_{G_0,d}\ar[d]^{(\gamma_{0},\gamma_{r})}\ar[ll]_{(\orr{s},\orr{t})}\\
(\mathrm{Bun}_{T_1})^{2}\times (\mathrm{Bun}_{T_2})^{2}\ar[rr]^{\Pi\times\Pi\times\Pi\times\Pi} && (\mathrm{Bun}_{G_0})^{2}\times (\mathrm{Bun}_{G_0})^{2} && H_{d}\times H_{d}\ar[ll]_{s^2 \times t^2}\\
\mathrm{Bun}_{T_1}\times \mathrm{Bun}_{T_2}\ar@<-3ex>[u]_{(\mathrm{id},\Fr)}\ar@<3ex>[u]^{(\mathrm{id},\Fr)}\ar[rr]^{\Pi_1\times\Pi_2} && \mathrm{Bun}_{G_0}\times\mathrm{Bun}_{G_0}\ar@<-3ex>[u]_{(\mathrm{id},\Fr)}\ar@<3ex>[u]^{(\mathrm{id},\Fr)} && H_{d}\ar[ll]_{(s, t)} \ar[u]_{(\mathrm{id},\Fr)}}
\end{equation}

The stack $H_d$ is defined exactly as in \cite{YZ}.  In particular, $H_d = \widetilde H_d/\Pic_X$, where $\widetilde H_d$ parameterizes colength $d$ injections $\phi \colon \mathcal{E} \hookrightarrow \mathcal{E}'$ of rank two vector bundles on $X$.  The map $(s, t) \colon H_d \to \Bun_{G_0}^2$ appearing in the bottom row of (\ref{octahedron}) takes the map $\phi$ to $(\mathcal{E}, \mathcal{E}')$.

The stack $\Hk^r_{G_0,d}$ is defined, as in \cite[6.3.3]{YZ}, to be $\widetilde \Hk_{G_0,d}^r/\Pic_X$.  Here, $\widetilde \Hk_{G_0,d}^r$ parameterizes colength $d$ injections $\phi \colon \mathcal{E} \hookrightarrow \mathcal{E}'$ of rank two vector bundles on $X$, together with $r$ modifications $f_i \colon \mathcal{E}_i \to \mathcal{E}_{i+1}$ and $f_i' \colon \mathcal{E}_i' \to \mathcal{E}_{i+1}'$ of $\mathcal{E}$ and $\mathcal{E}'$, compatible with $\phi$.  The modifications $f_i$ and $f_i'$ are required to be of type $\mu$ and above the same points $(x_1, \ldots, x_r) \in X^r$.  The isomorphism class of $\Hk^r_{G_0,d}$ is independent of $\mu$.  The map 
\[(\orr{s},\orr{t}) \colon \Hk_{G_0,d}^r \to \Hk_{G_0}^r \times \Hk_{G_0}^r\] in (\ref{octahedron}) is $\{\phi, f_i, f_i'\} \mapsto (\{\mathcal{E}, f_i\}, \{\mathcal{E}', f_i'\})$.

The following two lemmas follow immediately from the definitions.   

\begin{lemma}
The fiber product of the bottom row in (\ref{octahedron}) is $M_d$.  
\end{lemma}

\begin{lemma}
The fiber product of the top row in (\ref{octahedron}) is $\Hk^\mu_{M_d}$, i.e. 
\begin{equation}\label{Hk_d cartesian}
\xymatrix{
{  \Hk_{M_d}^\mu} \ar[r]\ar[d]  &  {  \Hk_{G_0,d}^r   }  \ar[d] \\
{  \Hk_{T_1}^\mu \times \Hk_{T_2}^\mu } \ar[r]^{\Pi_1^\mu \times \Pi_2^\mu}  & { \Hk_{G_0}^r \times \Hk_{G_0}^r}  .
}
\end{equation}
is cartesian.  
\end{lemma}

We will have to work around the singularities of $\Hk_{G_0,d}^r$ in order to compute various intersection pairings.  Let $\Hk^{r,\circ}_{G_0,d} \subset \Hk^r_{G_0,d}$ be the open substack consisting of those $\{\phi, f_i, f_i', x_i\}$ such that the support of the divisor of $\det(\phi)$ is disjoint from the $x_i$.  Then $\Hk^{r,\circ}_{G_0,d}$ is smooth of dimension $2d +2r + 3g -3$ \cite[6.10(1)]{YZ}.  Let $\Hk_{M_d}^{\mu,\circ}$ be the preimage of $\Hk^{r,\circ}_{G_0,d}$ in $\Hk_{M_d}^\mu$.  

\begin{lemma}\label{smooth part}
If $d \geq 2g_3 - 1$, then 
\[\dim(\Hk_{M_d}^\mu - \Hk_{M_d}^{\mu, \circ}) < 2d - g + 1 = \dim\, \Hk^\mu_{M_d}.\]  
\end{lemma}

\begin{proof}
From Proposition \ref{prop:rep one paw} and (\ref{Hecke mu}) we have 
\[\Hk_{M_d}^\mu \simeq A_d \times_{\Sigma_{2d}(Y_3)} H^r_{2d}(Y),\] where $H^r_{2d}(Y)$ parameterizes tuples $(E_0,E_1, \cdots, E_r)$ of effective divisors on $Y$ of degree $2d$, such that for each $1 \leq i \leq r$, $E_i$ is obtained from $E_{i-1}$ by changing a point $y_i \in E_{i-1}$ to $\tau_3(y_i)$.  If we write $g_3 \colon Y \to Y_3$ for the double cover, then the divisor $E = g_{3*}(E_i) \in \Div(Y_3)$ is independent of $i$.  

The locus $\Hk_{M_d}^\mu - \Hk_{M_d}^{\mu, \circ}$ consists of those triples 
\[(\Delta, \xi, (E_i)) \in A_d \times_{\Sigma_{2d}(Y_3)} H_{2d}^r(Y)\] such that the divisors $\mathrm{div}(\Tr(\xi)) = \mathrm{div}(\xi + \xi^{\sigma_3})$ and $\mathrm{div}(\xi \xi^{\sigma_3}) = f_{3*}(E)$ have a point in common.  For such triples, there exists $x \in |Y_3|$ such that $x$ and $x^{\sigma_3}$ are in $E$.  Thus, the image of $(\Delta, \xi, (E_i))$ in $A_d$ lies in the subscheme $C_d \subset A_d$ consisting of pairs $(\Delta, \xi)$ such that $\mathrm{div}(\xi)$ and $\mathrm{div}(\xi)^{\sigma_3}$ have a point in common.  Since there is a surjective map $X \times A_{d-1} \to C_d$, we have 
\[\dim \, C_d \leq \dim(X\times A_{d-1}) = 2d - g.\]  As the composition $\Hk_{M_d}^\mu \to M_d \to A_d$ is a finite morphism, we deduce the desired inequality: $\dim(\Hk_{M_d}^\mu - \Hk^{\mu, \circ}_d) \leq 2d - g$.    
\end{proof}

\begin{lemma}\label{fundamentalcycle}
The refined Gysin map 
\[
\left(\Pi_1^\mu \times \Pi_2^\mu\right)^! \colon \Ch_{2d+2r + 3g -3}(\Hk_{G_0,d}^r) \to \Ch_{2d-g+1}(\Hk_{M_d}^\mu)
\]
associated to the diagram (\ref{Hk_d cartesian}) is defined.  Moreover, 
\begin{equation}\label{eq:fundamental class}
\left(\Pi_1^\mu \times \Pi_2^\mu\right)^![\Hk_{G_0,d}^r] = [\Hk_{M_d}^\mu].
\end{equation}
\end{lemma}

\begin{proof}
For the first statement, it is enough to verify the two conditions in \cite[A.2.8]{YZ}.  The first condition is satisfied since $\Hk_{M_d}^\mu$ is a scheme (Corollary \ref{prop:hk_d facts}).  For the second condition, it is enough to show that for $i = 1,2$, the map $\Pi_i^\mu \colon \Hk_{T_i}^\mu \to \Hk_{G_0}^r$ can be factored as a regular local immersion followed by a smooth relative Deligne-Mumford type morphism.  Since $X^{'r} \to X^r$ is \'etale, it is enough to prove this for the base change $\Hk_{T_i}^\mu \to \Hk_{G_0}^r \times_{X^r} X^{'r}$, and this is proved in \cite[Lem.\ 6.11(1)]{YZ}.   

For the second statement, note that 
\[\Ch_{2d+2r + 3g - 3}\left(\Hk_{G_0,d}^r\right) \simeq \Ch_{2d+2r + 3g - 3}\left(\Hk_{G_0,d}^{r, \circ}\right).\]  By Lemma \ref{smooth part}, we also have
\[\Ch_{2d -g + 1}\left(\Hk_{M_d}^\mu\right) \simeq \Ch_{2d-g+1}\left(\Hk_{M_d}^{\mu, \circ}\right).\] Both of these isomorphisms preserve fundamental classes.  On the other hand, $\left(\Pi_1^\mu \times \Pi_2^\mu\right)^![\Hk_{G_0,d}^{r, \circ}] = [\Hk_{M_d}^{\mu, \circ}]$, since $\Hk_{G_0,d}^{r, \circ}$ is smooth.  The equality (\ref{eq:fundamental class}) follows.     
\end{proof}

The stack $\Sht_{G_0,d}^r$ is defined to be the fiber product of the third column in (\ref{octahedron}):
\begin{equation}\label{ShtGd definition}
\xymatrix{
{  \Sht_{G_0,d}^r} \ar[r]\ar[d]  &  {  \Hk_{G_0,d}^r   }  \ar[d]^{(\gamma_0, \gamma_r)} \\
{  H_d} \ar[r]^{(\mathrm{id},\Fr_{H_d})}  & { H_d \times H_d}  .
}
\end{equation}
By \cite[Lem.\ 6.12]{YZ}, there is a canonical isomorphism
\begin{equation}\label{ShtukaDdecomp}
\Sht_{G_0,d}^r \cong \bigsqcup_{D \in \Sigma_d(X)(\kk)} \Sht_{G_0}^r(f_D),
\end{equation}
where $\Sht_{G_0}^r(f_D)$ is the stack of Hecke correspondences between Shtukas with $r$ paws, defined in \cite[5.3.1]{YZ}.  Recall that it is these Hecke correspondences which give the action of $\mathscr{H}$ on $\Ch_{c,r}(\Sht_{G_0}^r)$, and which allowed us to define $\I_r(f_D)$.    

\begin{lemma}\label{lem:octahedron lemma 1}
Let $D$ be an effective divisor of degree $d$, and recall the definition of $\Sht_{M_D}^\mu$ in (\ref{ShtD}).  Then the following diagram is cartesian 
\begin{equation}\label{Shtuka_D cartesian}
\xymatrix{
{  \Sht_{M_D}^\mu} \ar[r]\ar[d]  &  {  \Sht_{G_0}^r(f_D)   }  \ar[d] \\
{  \Sht_{T_1}^\mu \times \Sht_{T_2}^\mu} \ar[r]^{\theta_1^\mu \times \theta_2^\mu}  & { \Sht_{G_0}^r \times \Sht_{G_0}^r}  .
}
\end{equation}
\end{lemma}

\begin{proof}
The fiber products of the three rows in the octahedron are
\begin{equation}\label{rows}
\Hk_{M_d}^\mu \xrightarrow{(\gamma_0, \gamma_r)} M_d \times M_d \xleftarrow{(\mathrm{id}, \Fr)}M_d,
\end{equation}
and the fiber product of this resulting diagram is by definition 
\begin{equation}\label{Ddecomposition}
\Sht_{M_d}^\mu \cong \bigsqcup_{D \in \Sigma_d(X)(\kk)} \Sht_{M_D}^\mu.
\end{equation}
By (\ref{Sht_T definition}), the fiber products of the three columns in the octahedron are
\begin{equation}\label{columns}
\Sht_{T_1}^\mu \times \Sht_{T_2}^\mu \xrightarrow{\theta_1^\mu \times \theta_2^\mu}\Sht_{G_0}^r \times \Sht_{G_0}^r \xleftarrow{(s,t)} \Sht_{G_0,d}^r.
\end{equation}
By \cite[Lem.\ A.9]{YZ}, the fiber product of (\ref{rows}) is canonically isomorphic, as an $A_d$-stack, to  
the fiber product of (\ref{columns}).  We therefore have the cartesian square
\begin{equation}\label{Shtuka_d cartesian}
\xymatrix{
{  \Sht_{M_d}^\mu} \ar[r]\ar[d]  &  {  \Sht_{G_0,d}^r   }  \ar[d] \\
{  \Sht_{T_1}^\mu \times \Sht_{T_2}^\mu} \ar[r]^{\theta_1^\mu \times \theta_2^\mu}  & { \Sht_{G_0}^r \times \Sht_{G_0}^r}  .
}
\end{equation}
Taking the fiber over $D \in \Sigma_{d}(X)(\kk)$ in (\ref{Shtuka_d cartesian}) and using the decompositions (\ref{ShtukaDdecomp}) and (\ref{Ddecomposition}), we obtain the desired cartesian square (\ref{Shtuka_D cartesian}).  
\end{proof}

Lemma \ref{lem:octahedron lemma 1} allows us to define two different maps 
\[\Ch_{2d+2r+3g-3}(\Hk_{G_0,d}^r) \to \Ch_0(\Sht_{M_d}^\mu),\] each obtained by composing two refined Gysin morphisms.  Specifically, the cartesian squares (\ref{ShtGd definition}) and (\ref{Shtuka_d cartesian}) induce the composition 
\[ \Ch_{2d+2r+3g-3}(\Hk_{G_0,d}^r) \map{(\mathrm{id}, \Fr_{H_d})^!} \Ch_{2r}(\Sht_{G_0}^d) \map{\left(\theta_1^\mu \times \theta_2^\mu\right)^!} \Ch_0(\Sht_{M_d}^\mu),\]
whereas the cartesian squares (\ref{Hk_d cartesian}) and (\ref{sht_d}) give 
\[ \Ch_{2d+2r+3g-3}(\Hk_{G_0,d}^r) \map{\left(\Pi_1^\mu \times \Pi_2^\mu\right)^!} \Ch_{2d-g+1}(\Hk_{M_d}^\mu)\map{(\mathrm{id}, \Fr_{M_d})^!}  \Ch_0(\Sht_{M_d}^\mu).\] 
This is assuming that $\left(\theta_1^\mu \times \theta_2^\mu\right)^!$ and $(\mathrm{id}, \Fr_{H_d})^!$ are well-defined, which we justify next. The other technical result that we need is that these two compositions agree on the fundamental cycle:      

\begin{lemma}\label{lem:octahedron lemma 2}
The refined Gysin maps $\left(\theta_1^\mu \times \theta_2^\mu\right)^!$ and $(\mathrm{id}, \Fr_{H_d})^!$ are well-defined. Moreover, we have
\[\left(\theta_1^\mu \times \theta_2^\mu\right)^!(\mathrm{id}, \Fr_{H_d})^![\Hk_{G_0,d}^r] =  (\mathrm{id}, \Fr_{M_d})^!\left(\Pi_1^\mu \times \Pi_2^\mu\right)^![\Hk_{G_0,d}^r].\]
\end{lemma}
\begin{proof}
This is the octahedron lemma \cite[Thm.\ A.10]{YZ} applied to the diagram (\ref{octahedron}).  We must show that hypotheses $(1)$-$(4)$ of that lemma are satisfied.  The smoothness of all stacks in (\ref{octahedron}) aside from $\Hk_{G_0,d}^r$ was proven in \cite{YZ}, so (1) is verified.  That the middle and left rows and middle and bottom columns have expected fiber product dimensions is an easy computation using Proposition \ref{prop:Msmooth} and that $\dim H_d = 2d + 3g - 3$, $\dim \Bun_{G_0} = 3g - 3$, and $\dim \Bun_{T_i} = g - 1$.  This proves (2).  Hypothesis $(3)$ was verified in the proof of Lemma \ref{fundamentalcycle} and \cite[Lem.\ 6.14]{YZ}.     

To verify (4), we must check that the two conditions in \cite[A.2.8]{YZ} are satisfied for the cartesian diagrams (\ref{sht_d}) and (\ref{Shtuka_d cartesian}).  First note that the fiber product in both of these diagrams is $\Sht_{M_d}^\mu$, which is a scheme by Proposition \ref{Hecke mu} and (\ref{sht_d}).  Thus, it remains to check that the bottom row in each of the cartesian squares (\ref{sht_d}) and (\ref{Shtuka_d cartesian}) can be factored as regular local immersion followed by a smooth relative Deligne-Mumford type morphism.  For (\ref{sht_d}), note that $(\mathrm{id}, \Fr_{M_d})$ is a regular immersion since $M_d$ is smooth (by Proposition \ref{prop:Msmooth}).  For (\ref{Shtuka_d cartesian}), this follows since $\Sht_{T_1}^\mu$, $\Sht_{T_2}^\mu$, and $\Sht_{G_0}^r$ are all smooth Deligne-Mumford stacks.         
\end{proof}

We finally put everything together to prove Proposition \ref{prop:fundamental}.  Recalling that $[\Sht_{T_i}^\mu]$ is the pushforward of the fundamental class along the map $\theta_i^\mu \colon \Sht_{T_i}^\mu \to \Sht_{G_0}^\mu$, we have 
\begin{align*}
\I_r(f_D) &= \left\langle [\Sht_{T_1}^\mu], f_D *  [\Sht_{T_2}^\mu]\right\rangle_{\Sht_{G_0}^r}\\
&= \deg\left(\left(\theta_1^\mu \times \theta_2^\mu\right)^![\Sht_{G_0}^r(f_D)]\right).
\end{align*}
By Lemma  \ref{lem:octahedron lemma 1}, the class $\left(\theta_1^\mu \times \theta_2^\mu\right)^![\Sht_{G_0}^r(f_D)]$ is the $D$th component of 
\[\left(\theta_1^\mu \times \theta_2^\mu\right)^![\Sht_{G_0,d}^r] \in \bigoplus_{D \in \Sigma_d(X)(\kk)} \Ch_0(\Sht_{M_D}^\mu).\]
So to prove Proposition \ref{prop:fundamental} for all $D$ of degree $d$, it is enough to show that 
\[\left(\theta_1^\mu \times \theta_2^\mu\right)^![\Sht_{G_0,d}^r] = (\mathrm{id}, \Fr_{M_d})^![\Hk_{M_d}^\mu].\]
For this we compute:
\begin{align*}
\left(\theta_1^\mu \times \theta_2^\mu\right)^![\Sht_{G_0,d}^r] &= \left(\theta_1^\mu \times \theta_2^\mu\right)^!(\mathrm{id}, \Fr_{H_d})^![\Hk_{G_0,d}^r] \hspace{6mm} (\mbox{\cite[Lem.\ 6.14(2)]{YZ}}) \\
&= (\mathrm{id}, \Fr_{M_d})^!\left(\Pi_1^\mu \times \Pi_2^\mu\right)^![\Hk_{G_0,d}^r] \hspace{3mm} (\mbox{Lemma \ref{lem:octahedron lemma 2}})\\
&= (\mathrm{id},\Fr_{M_d})^![\Hk_{M_d}^\mu] \hspace{25mm} (\mbox{Lemma \ref{fundamentalcycle}}).
\end{align*}
This completes the proof of Proposition \ref{prop:fundamental}.
\end{proof}


\subsection{Completion of the proof of Proposition \ref{prop:geometric trace}}


The proof of  Proposition \ref{prop:geometric trace} is obtained by combining  Proposition \ref{prop:fundamental} with the  Lefschetz-Verdier trace formula, as we now explain.

  \begin{proof}[Proof of   Proposition $\ref{prop:geometric trace}$]
  Consider the composition
 \begin{align*}
 \Ch_{2d-g+1}(\Hk_{M_d}^\mu)   & \map{ (\ref{basic gysin}) }  \Ch_0(\Sht_{M_d}^\mu)  \\
&  \map{(\ref{Ddecomp})}   \Ch_0(\Sht_{M_d}^\mu)  = \bigoplus_{\xi \in A_D(\kk)} \Ch_0(\Sht_{M_d}^\mu(\xi)),
  \end{align*}
  where the final decomposition is induced by (\ref{ShtD}).  If the image of the fundamental class $[\Hk_{M_d}^\mu] \in \Ch_{2d-g+1}(\Hk_{M_d}^\mu)$
 in the component indexed by $\xi$ is denoted $C_\xi$,  the trace formula of  \cite[A.12]{YZ} implies 
\[
\deg(C_\xi ) = \mathrm{Trace}\big   ( [\Hk_{M_d}^\mu]_{\bar{\xi}} \circ \mathrm{Frob}_\xi ; (\alpha_*\Q_\ell)_{\bar\xi}  \big),
\]
The isomorphism  (\ref{Hecke mu}) implies the equality $ [\Hk_{M_d}^\mu]  = [\Hk_{M_d}]^r$ of  endomorphisms of $\alpha_* \Q_\ell$, and so
 \[
\deg(C_\xi ) = \mathrm{Trace}\big(  [\Hk_{M_d}]^r_{\bar{\xi}} \circ \mathrm{Frob}_\xi ; (\alpha_*\Q_\ell)_{\bar\xi}  \big).
\]
On the other hand,  Proposition \ref{prop:fundamental}  gives the first equality in
 \[
 \I_r(f_D)  = \sum_{\xi \in A_D(\kk)} \deg(C_\xi)
 = \sum_{\xi \in A_D(\kk)}  \mathrm{Trace}\big(  [\Hk_{M_d}]^r_{\bar{\xi}} \circ \mathrm{Frob}_\xi  ;  ( \alpha_*\Q_\ell)_{\bar\xi}  \big),
 \]
    completing the proof.
\end{proof}


\section{Completion of the proofs}


Fix an  integer $r\ge 0$, and an $r$-tuple $\mu=(\mu_1,\ldots, \mu_r) \in \{ \pm 1\}^r$ satisfying  the parity condition
$\sum_{i=1}^r \mu_i =0.$  In particular,  $r$ is even.

 
 \subsection{Representations of symmetric groups}
 \label{ss:symmetric representations}
 

Fix a positive integer $d$. 
We will prove some elementary facts about the representation theory of the finite group  
\[
\Gamma_{2d} = \{ \pm 1\}^{2d} \rtimes S_{2d}.
\]
These facts will be used in the proof of Proposition \ref{prop:fundamental lemma} below.

Denote by $\bm{1}$ the trivial representation of $S_{2d}$ on $\Q_\ell$, so that 
\begin{equation}\label{symmetric induced}
\mathrm{Ind}_{ S_{2d}  }^{\Gamma_{2d} } \bm{1}  = \{ \Phi  : S_{2d} \backslash \Gamma_{2d}  \to \Q_\ell \}.
\end{equation}
For each $ x  \in \{ \pm 1\}^{2d}$ denote by $\Phi_x$ the characteristic function of the coset
$S_{2d}  \cdot x \subset \Gamma_{2d}$.   As $x$ varies these form a basis for (\ref{symmetric induced}).
Setting
 \[
e_i = ( 1,\ldots, 1, \underbrace{-1}_{i^\mathrm{th} \mathrm{\, place} }, 1,\ldots,1) \in \{\pm 1\}^{2d},
\]  
define  a    $\Gamma_{2d}$-linear endomorphism   of (\ref{symmetric induced})  by 
$
H \cdot  \Phi_x = \sum_{i=1}^{2d} \Phi_{ e_i  x}.
$

\begin{proposition}\label{prop:rep decomp}
There is a $\Gamma_{2d}$-stable decomposition 
\begin{equation}\label{rep decomp}
\mathrm{Ind}_{ S_{2d}  }^{\Gamma_{2d} } \bm{1}  = \bigoplus_{ \substack{  d_1,d_2 \ge 0 \\ d_1+d_2=2d   } } V_{(d_1,d_2)}
\end{equation}
where $V_{(d_1,d_2)}$ is the subspace of (\ref{symmetric induced}) on which $H$ acts as the scalar $d_1-d_2$.  
Moreover,  each $V_{(d_1,d_2)}$ is an irreducible representation of $\Gamma_{2d}$.
\end{proposition}

\begin{proof}
For each character $\chi : \{ \pm 1\}^{2d} \to \{\pm 1\}$ define 
\[
\Psi_\chi = \sum_{ x \in \{ \pm 1\}^{2d} } \chi (x) \Phi_x .
\]
As $\chi$ varies these  form a basis for (\ref{symmetric induced}), and   an easy calculation shows that  
\[
H \cdot \Psi_\chi = \big( \mathrm{pos}(\chi) - \mathrm{neg}(\chi) \big)  \cdot   \Psi_\chi,
\]
where  $\mathrm{pos}(\chi) = \# \{ e_i : \chi(e_i)= 1 \}$ and $\mathrm{neg}(\chi) = \# \{ e_i : \chi(e_i)= - 1 \}$. 
It is now easy to see that 
\[
V_{(d_1,d_2)}  = \mathrm{Span} \left\{ \Psi_\chi :  \begin{array}{c} \mathrm{pos}(\chi) = d_1 \\ \mathrm{neg}(\chi)=d_2 \end{array} \right\}
\]
is $\Gamma_{2d}$-stable and irreducible, that  $H$ acts by $d_1-d_2$, and that    (\ref{rep decomp}) holds.
\end{proof}

Now fix a pair $(d_1,d_2)$ of non-negative integers  such that $d_1+d_2 = 2d$.  For $i\in \{1,2\}$, set
$
\Gamma_{d_i} = \{ \pm 1\}^{d_i} \rtimes S_{d_i},
$    
and define a character 
\[
\eta_{d_i}: \{ \pm 1\}^{d_i} \to \{ \pm 1\}
\] 
by $(x_1,\ldots, x_{d_i}) \mapsto x_1\cdots x_{d_i}$.  It extends uniquely to  a character of $\Gamma_{d_i}$, 
trivial on the subgroup $S_{d_i}$.

\begin{proposition}\label{prop:induced constituent}
There are isomorphisms of $\Gamma_{2d}$-representations
\begin{align*}
 \mathrm{Ind}^{\Gamma_{2d}}_{  \Gamma_{d_1} \times \Gamma_{d_2}  } ( \eta_{d_1}   \boxtimes \bm{1} )  \iso V_{(d_1,d_2)} 
\end{align*}
\end{proposition}

\begin{proof}
Define a character $\chi : \{ \pm 1 \}^{2d} \to \{ \pm 1\}$ by 
\[
\chi(x_1,\ldots, x_{2d}) = \eta_{d_1}( x_1 , \ldots, x_{d_1} ).
\]
Using the notation of the  proof of Proposition \ref{prop:rep decomp}, it is easy to see that the subgroup $\Gamma_{d_1}\times \Gamma_{d_2} \subset \Gamma_{2d}$ acts on the vector $\Psi_\chi$ via the character $\eta_{d_1}   \boxtimes \bm{1}$, and that $\Psi_\chi$ generates $V_{(d_1,d_2)}$ as a $\Gamma_{2d}$-representation.
Thus Frobenius reciprocity provides a surjection
\[
 \mathrm{Ind}^{\Gamma_{2d}}_{  \Gamma_{d_1} \times \Gamma_{d_2}  } ( \eta_{d_1}   \boxtimes \bm{1})  \to V_{(d_1,d_2)} ,
\]
which is an isomorphism by dimension counting.  
\end{proof}


\subsection{Comparison of \'etale sheaves}


The constructions of \S \ref{ss:geometric intersection} provide us with a constructible  $\ell$-adic sheaf $\alpha_* \Q_\ell$ on $A_d$, endowed with an endomorphism $[\Hk_{M_d}]$.
On the other hand, for every pair of non-negative integers $(d_1,d_2)$ with $d_1+d_2=2d$ the constructions of \S \ref{ss:geometric orbital} provide us with a constructible  $\ell$-adic sheaf $\beta_* L_{(d_1,d_2)}$ on $A_d$.
The goal of this section is to prove the following result, relating these \'etale sheaves.

\begin{proposition}\label{prop:fundamental lemma}
Assume that $d\ge 2g_3-1$.  There is an isomorphism
\begin{equation}\label{perverse continuation}
\alpha_* \Q_\ell \iso 
\bigoplus_{  \substack{  d_1,d_2 \ge 0  \\ d_1+d_2=2d }}  \beta_* L_{ (d_1,d_2) }
\end{equation}
of $\ell$-adic sheaves on $A_d$.  Each summand  is stable under the Hecke correspondence $[\Hk_{M_d}]$, which 
acts on $ \beta_* L_{ (d_1,d_2) }$  via the scalar $d_1-d_2$.
\end{proposition}

\begin{proof}
 Let  $U_{2d}(Y_3)  \subset Y_3^{2d}$  be the open subscheme parametrizing  $2d$-tuples of \emph{distinct} points on $Y_3$,
and let $U'_{2d}(Y) \subset Y^{2d}$ be its preimage under the morphism $Y^{2d} \to Y_3^{2d}$. 
Thus we have a cartesian diagram
\[
\xymatrix{
{  U'_{2d}(Y) } \ar[r] \ar[d] &  { Y^{2d}  }  \ar[d] \\
{  U_{2d}(Y_3)  }  \ar[r]  & {  Y_3^{2d} },
}
\]
in which the horizontal arrows are open immersions with dense image, and the vertical arrows are finite \'etale.  Taking the GIT quotients by the action of $S_{2d}$ throughout, and using the isomorphisms of (\ref{symmetric id}), we obtain a cartesian diagram
\begin{equation}\label{unramified stratum}
\xymatrix{
{ S_{2d} \backslash U'_{2d}(Y) } \ar[r] \ar[d]_a &  {   \Sigma_{2d}(Y) }  \ar[d]^{\Nm} \\
{   S_{2d} \backslash U_{2d}(Y_3)  }  \ar[r]_u  & {   \Sigma_{2d}(Y_3) ,}
}
\end{equation}
in which the horizontal arrows are open immersions, the vertical arrows are finite, and $a$ is \'etale.

By identifying   $\{\pm 1\}^{2d} \iso \Aut(Y/Y_3)^{2d}$,  the semi-direct product $\Gamma_{2d}$ of  \S \ref{ss:symmetric representations}
acts  on $U'_{2d}(Y)$.  In fact $\Gamma_{2d}$ is the automorphism group of the Galois cover
\[
U'_{2d}(Y) \to S_{2d} \backslash U_{2d}(Y_3),
\]
and the local system $a_*\Q_\ell$ on $S_{2d}\backslash U_{2d}(Y_3)$  corresponds to the induced representation (\ref{symmetric induced}).  
Each summand on the right hand side of (\ref{rep decomp})  also determines a local system, denoted the same way, and so we have a decomposition 
\begin{equation}\label{unr stratum decomp}
a_*\Q_\ell = \bigoplus_{ \substack{  d_1,d_2 \ge 0 \\ d_1+d_2=2d } } V_{(d_1,d_2)}
\end{equation}
of local systems on $S_{2d}\backslash U_{2d}(Y_3)$.

Define  $A_d^\circ$ as the cartesian product
\[
\xymatrix{
{  A_d^\circ } \ar[r]^v \ar[d]_\pi &   {  A_d }   \ar[d]^{f_3^\sharp}  \\
{      S_{2d} \backslash U_{2d}(Y_3)    }  \ar[r]_u&  {  \Sigma_{2d}(Y_3) . } 
}
\]
We interrupt the proof of Proposition \ref{prop:fundamental lemma} for two  lemmas.

\begin{lemma}
On $A_d^\circ$, there are isomorphisms of constructible sheaves
\begin{align}
 v^*\alpha_* \Q_\ell     & \iso    \pi^*a_*\Q_\ell   \label{first unr restriction} \\
v^* \beta_*L_{(d_1,d_2)}  &  \iso  \pi^*V_{(d_1,d_2)} .  \label{second unr restriction}
 \end{align}
\end{lemma}

\begin{proof}
Applying the proper base change theorem   to (\ref{fundamental 1}) and  (\ref{unramified stratum}) provides the 
first and third isomorphisms, respectively,  in 
\[
v^*  \alpha_* \Q_\ell    \iso   ( f_3^\sharp \circ v )^* \mathrm{Nm}_*\Q_\ell    
 \iso (  u \circ  \pi )^* \mathrm{Nm}_*\Q_\ell   \iso \pi^* a_*\Q_\ell .
\]
This proves (\ref{first unr restriction}).

The proof of (\ref{second unr restriction}) is very similar. 
Consider the canonical morphism
\[
\Sigma_{d_1}(Y_3) \times_\kk \Sigma_{d_2}(Y_3) \map{g} \Sigma_{2d}(Y_3).
\]
(This is the morphism labeled $\otimes$ in (\ref{fundamental 2}), but we temporarily change the name to avoid the awkward  notation $\otimes_*$ for the pushforward.)
By Proposition \ref{prop:induced constituent}, the local system $L_{d_1}\boxtimes \Q_\ell$ on $\Sigma_{d_1}(Y_3) \times_\kk \Sigma_{d_2}(Y_3)$ defined in  \S \ref{ss:orbital local} 
satisfies $ u^* g_*( L_{d_1}\boxtimes \Q_\ell ) \iso V_{(d_1,d_2)}.$
On the other hand, applying proper base change to the diagram (\ref{fundamental 2}) and recalling (\ref{local system}), there is an canonical isomorphism
$
\beta_*L_{(d_1,d_2)} \iso (f_3^\sharp)^* g_* ( L_{d_1}\boxtimes \Q_\ell ).
$
The desired isomorphism (\ref{second unr restriction}) is the composition
\[
v^*\beta_*L_{(d_1,d_2)}     \iso (f_3^\sharp \circ v)^* g_*( L_{d_1}\boxtimes \Q_\ell ) 
  \iso  (u\circ \pi )^* g_*( L_{d_1}\boxtimes \Q_\ell )   
 \iso \pi^*V_{(d_1,d_2)}.
\]
\end{proof}

\begin{lemma}
On $A_d$, there are isomorphisms of constructible sheaves
\begin{align}
 v_*v^*\alpha_* \Q_\ell     & \iso   \alpha_*  \Q_\ell   \label{first no perversity} \\
v_*v^* \beta_*L_{(d_1,d_2)}  &  \iso   \beta_* L_{(d_1,d_2)} .  \label{second no perversity}
 \end{align}
\end{lemma}

\begin{proof}
If $\mathscr{F}$ is a locally constant \'etale sheaf on a normal scheme $S$,  the natural map
$\mathscr{F} \to j_* j^* \mathscr{F}$ is an isomorphism for any open immersion  $j:S^\circ\to S$  with dense image.  
Indeed, normality guarantees that the natural map $\pi_0( U\times_S S^\circ)  \to \pi_0(U)$
is bijective for any \'etale morphism $U\to S$, which implies the claim whenever $\mathscr{F}$ is constant.
The locally constant case then follows using a descent argument.

Now define $M_d^\circ$ as the cartesian product
\[
\xymatrix{
{M_d^\circ} \ar[d]_{a} \ar[r]^{j}  &  { M_d } \ar[d]^{\alpha} \\
{ A_d^\circ } \ar[r]_{v}  & { A_d.} 
}
\]
Proposition \ref{prop:Msmooth} implies that $M_d$ is normal and $\alpha$ is finite.
Hence the discussion above and the proper base change theorem provide us with  isomorphisms 
\[
\alpha_* \Q_\ell \iso \alpha_*j_*j^*\Q_\ell \iso v_*a_* j^*\Q_\ell \iso v_* v^* \alpha_* \Q_\ell.
\]
   This proves  (\ref{first no perversity}). 
   The proof of (\ref{second no perversity}) is entirely similar, using Proposition \ref{prop:Nsmooth} in place of Proposition \ref{prop:Msmooth}.
\end{proof}

Combining (\ref{unr stratum decomp}) with the isomorphisms (\ref{first unr restriction}) and (\ref{second unr restriction}) yields an isomorphism of local systems
\begin{equation}\label{precontinuation}
v^* \alpha_*\Q_\ell \iso \bigoplus_{  \substack{  d_1,d_2 \ge 0  \\ d_1+d_2=2d }}  v^*\beta_* L_{ (d_1,d_2) }
\end{equation}
 over the open subscheme $A_d^\circ \subset A_d$.   The isomorphism (\ref{perverse continuation})  follows by applying $v_*$ to both sides and using the isomorphisms (\ref{first no perversity}) and  (\ref{second no perversity}).

We now turn to the endomorphism $[\Hk_{M_d}] : \alpha_*\Q_\ell \to \alpha_*\Q_\ell$.  
The endomorphism $H$ of (\ref{symmetric induced})  induces an endomorphism $H : a_*\Q_\ell \to a_*\Q_\ell$ of the local system (\ref{unr stratum decomp}).
After restricting to the open subscheme $A_d^\circ \subset A_d$  there is a commutative diagram
\[
\xymatrix{
{ v^*\alpha_* \Q_\ell    }  \ar[r]^{ [\Hk_{M_d}] }  \ar[d]_{ (\ref{first unr restriction}) }  &   {  v^*\alpha_* \Q_\ell    }  \ar[d]^{ (\ref{first unr restriction}) }   \\
{   \pi^*a_*\Q_\ell    }   \ar[r]_{H} &  {    \pi^*a_*\Q_\ell   } .
}
\]
Commutativity of the diagram follows by direct comparison of the definition of $H$ with Proposition \ref{prop:rep one paw}, which characterizes  $\Hk_{M_d}$ in terms of the correspondence (\ref{simple hecke}).

By Proposition \ref{prop:rep decomp}   the endomorphism $H$  acts via the scalar $d_1-d_2$ on  the summand $V_{(d_1,d_2)}$  in (\ref{unr stratum decomp}).  This implies that  the endomorphism $[\Hk_{M_d}]$ in the top row  of the above diagram acts  via the scalar $d_1-d_2$ on the summand  $v^*\beta_* L_{ (d_1,d_2) }$ in  (\ref{precontinuation}).
Using (\ref{first no perversity}) and (\ref{second no perversity}), it  also  on the summand $\beta_*L_{(d_1,d_2)}$ in (\ref{perverse continuation}).
\end{proof}


\subsection{The intersection pairing in cohomology}


Fix an auxiliary prime $\ell \neq \mathrm{char} (\kk)$.  
Before stating and proving our main results, we must summarize some results from  \cite{YZ} on various quotients of the $\ell$-adic analogue $\mathscr{H}_\ell=\mathscr{H}\otimes\Q_\ell$ of the $\Q$-algebra $\mathscr{H}$ of  \S \ref{ss:basic automorphic}.

The Hecke algebra $\mathscr{H}_\ell$ acts on the $\ell$-adic cohomology group
\[
V = H_c^{2r}(\Sht^\mu_{G_0} \otimes_k \bar k, \Q_\ell)(r),
\]
as in \cite[\S 7.1]{YZ}.   The cycle class map $\mathrm{cl} \colon \Ch_{c,r}(\Sht_{G_0}^r) \to V$
 is $\mathscr{H}$-equivariant, and  the cup product   
 \begin{equation}\label{cup}
 \langle\cdot,\cdot\rangle:  V\times V \to \Q_\ell
 \end{equation}
 pulls back to  the intersection pairing on the Chow group.

Recalling the map $\mathscr{H} \to \Q[\Pic_X(\kk)]^{\iota_\Pic}$ appearing in (\ref{eisenstein def}), define
\begin{align*}
\widetilde{\mathscr{H}}_\ell 
& = \mathrm{Image}  \big(  \mathscr{H}_\ell \to \End_{\Q_\ell} (V) \times \End_{\Q_\ell}(\mathscr{A}_\ell) \times \Q_\ell[\Pic_X(\kk)]^{\iota_\Pic}  \big)   \\
\overline{\mathscr{H}_\ell} 
& = \mathrm{Image}  \big(  \mathscr{H}_\ell \to \End_{\Q_\ell} (V) \times  \Q_\ell[\Pic_X(\kk)]^{\iota_\Pic}  \big) \\
\mathscr{H}_{\mathrm{aut},\ell} 
& = \mathrm{Image}  \big(  \mathscr{H}_\ell \to \End_{\Q_\ell}(\mathscr{A}_\ell) \times \Q_\ell[\Pic_X(\kk)]^{\iota_\Pic}  \big).
\end{align*}
These are  finite type $\Q_\ell$-algebras, related by surjections
\[
\xymatrix{
& { \widetilde{\mathscr{H}}_\ell } \ar[dl] \ar[dr] \\
{  \overline{\mathscr{H}_\ell}  }  \ar[dr] &  &  {  \mathscr{H}_{\mathrm{aut},\ell}  }  \ar[dl] \\
& {   \Q_\ell[\Pic_X(\kk)]^{\iota_\Pic}  .  }   
}
\]
Recalling the $\Q$-algebra $\mathscr{H}_\mathrm{aut}$ of \S \ref{ss:basic automorphic}, there is a canonical isomorphism 
\[
 \mathscr{H}_\mathrm{aut} \otimes\Q_\ell \iso \mathscr{H}_{\mathrm{aut},\ell} .
\]

For any $f\in \mathscr{H}$ the function $\J(f,s)$ of (\ref{J distribution def}) is a Laurent polynomial in $q^s$ with rational coefficients.
Setting 
\[
\J_r(f) = ( \log q)^{-r}  \frac{d^r}{ds^r}  \J(f ,s)  \big|_{s=0},
\]
we obtain a linear functional $\J_r : \mathscr{H} \to \Q$.  The following result shows that this agrees with the linear functional 
$\I_r : \mathscr{H} \to \Q$ defined by (\ref{I distribution def}).

\begin{proposition}\label{prop:key identity}
The equality \[\I_r(f) = \J_r(f)\] holds for every $f\in \mathscr{H}$.  
Moreover, the  $\Q_\ell$-linear extensions  of $\I_r$ and $\J_r$ to $\mathscr{H}_\ell \to \Q_\ell$   factor through  $\widetilde{\mathscr{H}}_\ell$.
\end{proposition}

\begin{proof}
The compatibility of the cup product pairing (\ref{cup}) with the intersection pairing on the Chow group implies that 
the $\Q_\ell$-linear extension $\I_r : \mathscr{H}_\ell \to \Q_\ell$ factors through  $\overline{\mathscr{H}_\ell}$. 
The final claim of Proposition \ref{J pi decomp} implies that the $\Q_\ell$-linear extension $\J_r : \mathscr{H}_\ell \to \Q_\ell$ factors through  $\mathscr{H}_{\mathrm{aut}, \ell}$.  It follows that both   $\I_r$ and $\J_r$  factor through the quotient 
$ \widetilde{\mathscr{H}}_\ell$.

It  remains to prove that $\I_r(f) = \J_r(f)$ for all $f\in \mathscr{H}$.
Assume first that $f=f_D$ for some effective divisor $D \in \mathrm{Div}(X)$ of degree $d\ge 2 g_3-1$.  
Combining the decomposition (\ref{J decomp}) with Proposition \ref{prop:geometric orbital}, we find
\[
\J_r(f_D)=   \sum_{ \xi\in A_D(\kk)} \sum_{  \substack{  d_1,d_2 \ge 0  \\ d_1+d_2=2d } } 
  (d_1-d_2)^r    \cdot   \mathrm{Trace} \big(\Frob_\xi  ; \,   \big( \beta_* L_{ (d_1,d_2) } \big)_{\bar \xi} \big)  .
\]
On the other hand,  Proposition   \ref{prop:geometric trace} tells us that 
\[
 \I_r(f_D) =   \sum_{\xi\in A_D(\kk) }
 \mathrm{Trace} \big(      [\Hk_{M_d}]_{\bar{\xi}} ^r  \circ \mathrm{Frob}_\xi   ;\,   ( \alpha_* \Q_\ell)_{\bar{\xi}}  \big)  .
\]
These two expressions are equal, by Proposition \ref{prop:fundamental lemma}.  
 
 The proof of \cite[Theorem 9.2]{YZ} shows that the image of $\mathscr{H}_\ell \to \widetilde{\mathscr{H}}_\ell$  is generated as $\Q_\ell$-vector space by the images of $f_D\in \mathscr{H}$ as $D$ ranges over all effective divisors on $X$ of degree $d \ge 2 g_3-1$.  Therefore $\I_r=\J_r$.
\end{proof}

According to  \cite[(9.5)]{YZ}, there is a canonical $\Q_\ell$-algebra decomposition
\begin{equation}\label{ell hecke decomp}
\widetilde{\mathscr{H}}_\ell = 
\widetilde{\mathscr{H}}_{ \ell, \mathrm{Eis}}  \oplus
 \big( \bigoplus_\mathfrak{m}   \widetilde{\mathscr{H}}_{ \ell,  \mathfrak{m} }   \big),
\end{equation}
where $\mathfrak{m}$ runs over the finitely many maximal ideals $\mathfrak{m} \subset \widetilde{\mathscr{H}}_\ell$ that do not contain the 
kernel of the projection  
\begin{equation}\label{ell satake}
\widetilde{\mathscr{H}}_\ell  \to \Q_\ell[\Pic_X(\kk)]^{\iota_{\Pic}} .
\end{equation}
For each such $\m$ the localization  $ \widetilde{\mathscr{H}}_{\ell,\mathfrak{m}}$ is a finite (hence Artinian) $\Q_\ell$-algebra.
If we denote by $E_\m$ its residue field, then Hensel's lemma implies that the quotient map $\widetilde{\mathscr{H}}_{\ell,\mathfrak{m}} \to E_\mathfrak{m}$ admits a unique section, which makes $\widetilde{\mathscr{H}}_{\ell, \mathfrak{m}}$ into an Artinian local $E_\mathfrak{m}$-algebra.

The decomposition (\ref{ell hecke decomp}) induces a decomposition  of $\widetilde{\mathscr{H}}_\ell$-modules
\begin{equation}\label{cohomological decomposition}
V = V_\Eis \oplus \big(\bigoplus_\m V_\m\big),
\end{equation}
in which each localization $V_\m$ is a finite-dimensional $E_\m$-vector space.
It follows from \cite[Corollary 7.15]{YZ} that this decomposition is orthogonal with respect to the cup product pairing.   
Moreover, the self adjointness of the action of $\mathscr{H}_\ell$ with respect to the cup product pairing (\ref{cup})
implies that there is a unique symmetric $E_\m$-bilinear pairing 
\[
\langle \cdot , \cdot \rangle_{E_\m} : V_\m \times V_\m \to E_\m
\]
such that $\mathrm{Trace}_{E_\m/\Q_\ell} \langle \cdot , \cdot \rangle_{E_\m}= \langle\cdot,\cdot\rangle$.

For $i\in \{1,2\}$,   we define
$
[\Sht_{T_i}^\mu]_\m \in V_\m
$ 
to be the projection  of the cycle class $\mathrm{cl}( [\Sht_{T_i}^\mu]) \in V$,  and form the intersection pairing
\[
\langle [\Sht_{T_1}^\mu]_\m, [\Sht_{T_2}^\mu]_\m \rangle_{E_\m} \in E_\m.
\]

Some of the maximal ideals $\mathfrak{m} \subset \widetilde{\mathscr{H}}_\ell$ appearing in (\ref{ell hecke decomp}) are attached to cuspidal automorphic forms, as we now explain.  Fix an  unramified cuspidal automorphic representation $\pi \subset \mathcal{A}_\cusp(G_0)$. As 
in \S \ref{ss:basic automorphic}, such a representation determines a homomorphism 
\[
\mathscr{H}_\mathrm{aut} \to \mathscr{H}_\cusp \map{\lambda_\pi} \C
\] 
whose image is a number field $ E_\pi$.  The induced map
\[
\widetilde{ \mathscr{H} }_\ell \to  \mathscr{H}_{\mathrm{aut},\ell} \map{\lambda_\pi} E_\pi \otimes \Q_\ell \iso \prod_{ \mathfrak{l} \mid \ell} E_{\pi ,\mathfrak{l}},
\]
determines,  for every prime  $\mathfrak{l}\mid \ell$ of $E_\pi$,  a surjection
$\lambda_{\pi,\mathfrak{l}}   :    \widetilde{\mathscr{H}}_\ell \to  E_{\pi,\mathfrak{l}}$ whose kernel is one of those maximal ideals 
\begin{equation}\label{cuspidal points}
\mathfrak{m}=\ker( \lambda_{\pi,\mathfrak{l}} )
\end{equation}
appearing in the decomposition (\ref{cohomological decomposition}).  This is a consequence of the isomorphism (\ref{automorphic decomp}).

Recalling the period integrals $\mathscr{P}_0$ and $\mathscr{P}_3$ of \S \ref{ss:spectral}, for every cuspidal automorphic representation $\pi \subset\mathcal{A}_\cusp(G_0)$ define 
\[
C(\pi,s) =   \frac{   \mathscr{P}_0 (  \phi ,s) \mathscr{P}_3(\overline{\phi} ,\eta)}{   \langle \phi, \phi \rangle_\mathrm{Pet} } .
\]
Here $\phi\in \pi^{U_0}$ is any nonzero vector.   
Recall from Remark \ref{rem:auto galois} that $\Aut(\C/\Q)$ acts on the set of all unramified cuspidal automorphic representations, in such a way that stabilizer of $\pi$ is the subgroup $\Aut(\C/E_\pi)$.

\begin{proposition}\label{prop:period reciprocity}
The complex number 
\[
C_r(\pi)  =  ( \log q)^{-r} \cdot  \frac{d^r}{ds^r}  C (  \pi ,s) \big|_{s=0}
\]  
satisfies $C_r(\pi)^\sigma= C_r(\pi^\sigma)$ for all $\sigma \in \Aut(\C/\Q)$.  In particular, it lies in  $E_\pi$.
\end{proposition}

\begin{proof}
 Proposition \ref{J pi decomp} implies that
\begin{equation}\label{period spectrum}
\J_r(f) = \sum_{\mathrm{unr.\, cusp.\,}\pi }  C_r(\pi)\cdot  \lambda_\pi(f)  
\end{equation}
for all $f\in \mathcal{I}^\Eis$, and both sides factor through the quotient 
\[
\mathscr{H}\to \mathscr{H}_\mathrm{aut}\iso \mathscr{H}_\cusp \times \Q[\Pic_X(\kk)]^{\iota_\Pic}
\]
appearing in (\ref{automorphic decomp}).  In other words, we may view (\ref{period spectrum}) as an equality of linear functionals on the cuspidal subalgebra
$\mathscr{H}_\cusp \subset \mathscr{H}_\mathrm{aut}$.

It follows from what was said in \S \ref{ss:basic automorphic} that $\mathscr{H}_\cusp$ is a finite product of number fields, where the factors are indexed by the $\Aut(\C/\Q)$-orbits of unramified cuspidal automorphic representations.  Restricting (\ref{period spectrum}) to the factor $\mathscr{H}_\cusp(\pi)\iso E_\pi$ indexed by the Galois orbit of $\pi$ yields 
the equality 
\[
\J_r(f) = \sum_{ \sigma : E_\pi \to \C  }  C_r(\pi^\sigma)\cdot \sigma(f)
\]
for all  $f\in E_\pi$.  The sum is over all  $\Q$-algebra embeddings $\sigma : E_\pi \to \C$, and for each such embedding we  fix an extension to $\Aut(\C/\Q)$.

On the other hand, we know that $\J_r(f)$ is $\Q$-valued, so the right hand side must be fixed by the action of $\Aut(\C/\Q)$.  The claim follows easily from this and the linear independence of $\{ \sigma : E_\pi \to \C\}$.
\end{proof}

\begin{theorem}\label{main}
Let $\mathfrak{m} \subset \widetilde{\mathscr{H}}_\ell$ be a maximal ideal that does not contain the kernel of (\ref{ell satake}).
\begin{enumerate}
\item
If $\mathfrak{m}$ is of the form (\ref{cuspidal points}) for an unramified cuspidal automorphic representation $\pi$ and a place $\mathfrak{l}\mid \ell$ of $E_\pi$, the equality
\[
\langle [\Sht_{T_1}^\mu]_\m , \,   [\Sht_{T_2}^\mu]_\m \rangle_{E_\m} = C_r(\pi)
\]
holds  in $E_\m=E_{\pi ,\mathfrak{l}}$.
\item
If $\mathfrak{m}$ is not of the form (\ref{cuspidal points}) then 
\[
\langle [\Sht_{T_1}^\mu]_\m , \,   [\Sht_{T_2}^\mu]_\m \rangle_{E_\m} =0.
\]
\end{enumerate}
\end{theorem}

\begin{proof}
Given Proposition \ref{prop:key identity}, the proof is essentially the same as that of  \cite[Theorem 1.6]{YZ}.
Briefly, restrict both
\[
\I_r,\J_r : \widetilde{\mathscr{H}}_\ell \to \Q_\ell
\]
to the  $E_\m$-algebra  $\widetilde{\mathscr{H}}_{\ell,\m}$ in (\ref{ell hecke decomp}), and then further restrict to $E_\m$ itself.  
Directly from its definition (\ref{I distribution def}), the resulting $\I_r : E_\m \to \Q_\ell$ satisfies
\begin{align*}
\I_r(f) & = \langle [\Sht_{T_1}^\mu]_\m , \,  f *   [\Sht_{T_2}^\mu]_\m \rangle \\
& = \mathrm{Trace}_{E_\m/\Q_\ell} \big( f\cdot  \langle [\Sht_{T_1}^\mu]_\m , \,   [\Sht_{T_2}^\mu]_\m \rangle_{E_\m}  \big),
\end{align*}
where the first pairing is (\ref{cup}). 
 As for $\J_r : E_\m \to \Q_\ell$, an argument similar to that used in Proposition \ref{prop:period reciprocity} shows that
\[
\J_r(f) = \mathrm{Trace}_{E_\m/\Q_\ell} \big( f\cdot C_r(\pi) \big)
\]
if $\m$ has the form (\ref{cuspidal points}), and otherwise $\J_r(f) =0$.   
The claim now follows from  $\I_r=\J_r$ and the nondegeneracy of the trace pairing.
\end{proof}


\subsection{The proofs of Theorems \ref{thm:main intro} and \ref{thm:nonvanishing}}
\label{ss:main proofs}


As in the introduction,  for $i\in \{1,2\}$ we let $[\Sht_{T_i}^r]$ be the pushforward of the fundamental class under    
\[
\theta_i^\mu \colon \Sht_{T_i}^r \to \Sht_G^r,
\] 
and let 
$
\tilde W_i \subset \Ch_{c,r}(\Sht_G^r)
$ 
be the  $\mathscr{H}$-submodule generated by it.   Define quotients  
\begin{align*}
W_1 & = \tilde{W}_1/  \{ c \in \tilde{W}_1 : \langle c, \tilde{W}_2 \rangle =0 \}  \\
W_2 & = \tilde{W}_2/  \{ c \in \tilde{W}_2 : \langle c, \tilde{W}_1\rangle =0 \},
\end{align*}
so that the intersection pairing  descends to  $\langle \cdot\, , \cdot\rangle :  W_1 \times W_2 \to \Q$.

\begin{proposition}\label{prop:automorphic chow}
The actions of $\mathscr{H}$ on $W_1$ and $W_2$ factor through the quotient 
\[
\mathscr{H} \to \mathscr{H}_\mathrm{aut} \iso \mathscr{H}_\mathrm{cusp} \times \Q[\Pic_X(\kk)]^{\iota_\Pic}
\]
defined in \S \ref{ss:basic automorphic}.
\end{proposition}

\begin{proof}
By Proposition \ref{J pi decomp} the distribution $\J_r(f)$ only depends on the image of $f$ under $\mathscr{H} \to \mathscr{H}_\mathrm{aut}$.
By Proposition \ref{prop:key identity} the same is true of the distribution $\I_r(f)$ defined by (\ref{I distribution def}), and the claim follows exactly as in \cite[Corollary 9.4]{YZ}. 
\end{proof}

It follows from the discussion of \S \ref{ss:basic automorphic} that $\mathscr{H}_{\cusp,\R}= \mathscr{H}_\cusp\otimes_\Q\R$ is isomorphic to a product of copies of $\R$,  indexed by the unramified cuspidal automorphic representations $\pi$.  
For each such $\pi$, let $e_\pi \in \mathscr{H}_{\cusp,\R}$ be the corresponding idempotent.  
Using Proposition \ref{prop:automorphic chow}, these idempotents induce a decomposition
\[
W_i (\R) =   W_{i,\mathrm{cusp}} \oplus W_{i,\mathrm{Eis}}  = \left(\bigoplus _\pi W_{i, \pi}\right) \oplus W_{i,\mathrm{Eis}} 
\] 
where the sum is over all  unramified cuspidal  $\pi$, and $W_{i, \pi} \subset W_i(\R)$ is the $\lambda_\pi$-eignespace of $\mathscr{H}$.

The following is Theorem \ref{thm:main intro} of the introduction.

\begin{theorem}\label{thm:main chow}
If $[\Sht_{T_i}^r]_\pi$ denotes the projection of the image of $[\Sht_{T_i}^r]$ to the summand $W_{i,\pi}$, then 
\[
\langle   [\Sht_{T_1}^r]_\pi, [\Sht_{T_2}^r]_\pi \rangle  = C_r(\pi). 
\]  
\end{theorem}

\begin{proof}
It follows from the discussion of \S \ref{ss:basic automorphic} that $\mathscr{H}_\cusp$ decomposes as a product of totally real fields, 
indexed by the $\Aut(\C/\Q)$-orbits of unramified cuspidal automorphic representations.  Let $\Pi$ denote the $\Aut(\C/\Q)$-orbit of $\pi$, and let $\mathscr{H}_\Pi \subset \mathscr{H}_\cusp$ be the corresponding summand.

For each $\pi^\sigma \in \Pi$, the  corresponding $\lambda_{\pi^\sigma} : \mathscr{H}_\cusp \to \C$ restricts to an isomorphism $\lambda_{\pi^\sigma} : \mathscr{H}_\Pi \to E_{\pi^\sigma}$.   As $\sigma\in \Aut(\C/\Q)$ varies,  it follows from Proposition \ref{prop:period reciprocity} that we may collect together the constants $C_r(\pi^\sigma)$ into a single 
\[
C_r(\Pi)\in \mathscr{H}_\Pi
\]
such that $\lambda_{\pi^\sigma}( C_r(\Pi) ) = C_r(\pi^\sigma)$.

The idempotent $e_\Pi\in \mathscr{H}_\Pi \subset \mathscr{H}_\mathrm{aut}$ cuts out an $\mathscr{H}_\Pi$-vector space
\[
e_\Pi  W_i  \subset  W_i.
\]
The action of $\mathscr{H}$ on the Chow group is self-adjoint relative to the $\Q$-bilinear intersection pairing, and it follows that
there is a unique $\mathscr{H}_\Pi$-bilinear pairing  
\[
\langle \cdot , \cdot\rangle_\Pi \colon e_\Pi  W_1  \times e_\Pi  W_2 \to \mathscr{H}_\Pi
\]
whose trace  is the $\Q$-valued intersection form.

Using Theorem \ref{main} and  the compatibility of the cycle class map  with intersection pairings, we obtain the equality
\[
\langle  e_\Pi   [\Sht_{T_1}^r] , e_\Pi   [\Sht_{T_2}^r]   \rangle_\Pi  = C_r(\Pi)
\]
in $\mathscr{H}_\Pi$, and applying the isomorphism $\lambda_\pi : \mathscr{H}_\Pi \to E_\pi$ to both sides  proves the claim.
\end{proof}

Let $\mathscr{L}(\pi,s)$ be the normalized $L$-function defined in the introduction and recall that $\chi_1$ and $\chi_2$ are quadratic characters corresponding to $K_1$ and $K_2$, respectively.      
The following is Theorem \ref{thm:nonvanishing} of the introduction.

\begin{theorem}
In the notation of Theorem \ref{thm:main chow},
\[
\langle   [\Sht_{T_1}^r]_\pi, [\Sht_{T_2}^r]_\pi \rangle_{\pi}  = 0
\] 
if and only if 
\[
\mathscr{L}^{(r)}\left(\pi, 1/2 \right)\mathscr{L}\left(\pi \otimes \chi_1, 1/2\right)\mathscr{L}\left(\pi \otimes \chi_2, 1/2\right) = 0.
\]
\end{theorem}

\begin{proof}
By Theorem \ref{thm:main chow}, it suffices to show that
\begin{equation}\label{P0}
\mathscr{P}_0^{(r)}(\phi,0) = 0 \hspace{2mm} \mbox{ if and only if } \hspace{2mm} \mathscr{L}^{(r)}(\pi,1/2) = 0  
\end{equation}
and
\begin{equation}\label{P3}
\mathscr{P}_3(\bar\phi,\eta) = 0 \hspace{2mm} \mbox{ if and only if }  \hspace{2mm} \mathscr{L}(\pi \otimes \chi_1,1/2)\mathscr{L}(\pi \otimes \chi_2,1/2) = 0.  
\end{equation}
The equivalence (\ref{P0}) follows since we may choose a spherical vector $\phi$ so that 
\[\mathscr{P}_0(\phi,s) = q^{1-g} \mathscr{L}(\pi,2s + 1/2),\] 
as follows from the calculations in \cite[pg.\ 805-6]{YZ}.  Note that our definition of $\mathscr{P}_0$ differs slightly from the one appearing in \cite{YZ},  since we insert $2s$ instead of $s$ in the exponent.    
The claim (\ref{P3}) follows from applying the Yun-Zhang formula \cite{YZ} (or even the original work of Walspurger \cite{Waldspurger}, as extended to function fields in \cite{CW}) to the twist $\pi \otimes \chi_1$ and the cover $Y_3/X$.   
\end{proof}


\subsection{The proof of Theorem \ref{thm:bare intersection}}


The goal of this section is to prove that 
\[
\langle [\Sht_{T_1}^r ]  ,   [\Sht_{T_2}^r ] \rangle = 0
\]
when $r>0$.  This will be deduced  from the following result.

\begin{proposition}\label{f_0 computation}
Let $f  \in \mathscr{H}$ be the characteristic function of $U_0\subset G_0(\A)$, and recall the function  $\J(f,s)$ of (\ref{J distribution def}).  
\begin{enumerate}
\item
If $\mathrm{char}(\kk)=2$ then $\J(f ,s)=0$ for all $s\in \C$.
\item
If $\mathrm{char}(\kk)\neq 2$ then $\J(f ,s)=1$ for all $s\in \C$.
\end{enumerate}
\end{proposition}

\begin{proof}
Using Lemma \ref{lem:coset transfer}, we view $f$ as a compactly supported function 
\[
f: U_0 \backslash J(\A) / U_3 \to \Q.
\]
In other words, $f$ is the characteristic function of the image of
\[
\mathrm{Iso}(\mathbb{O}_3,\mathbb{O}_0) \subset \mathrm{Iso}(\A_3,\A_0) = \tilde{J}(\A)
\]
under $\tilde{J}(\A) \to J(\A)$.

\begin{lemma}\label{lem:optimal}
Fix $\gamma \in J(F)$, and let $\xi \in K_3$ be its image under (\ref{orbit invariant}).
If there exist   $t_0\in T_0(\A)$  and $t_3\in T_3(\A)$ such that $f( t_0^{-1} \gamma t_3) \neq 0$, then $\xi\in \kk$ and 
$2\xi =1$.
\end{lemma}

\begin{proof}
 By hypothesis there is some $\phi \in \tilde{J}(\A)$  lifting $t_0^{-1} \gamma t_3 \in J(\A)$,  and satisfying the integrality condition $\phi(\mathbb{O}_3) = \mathbb{O}_0$.
   As in the proof of Proposition \ref{prop:split coset regular}, there is a canonical bijection 
\[
\A_0^\times   \backslash     \mathrm{Iso}(\A_3,\A_0)   / \A_3^\times
\iso 
 \GL_2(\A) \backslash    \left\{  \begin{array}{c}  \mbox{pairs of embeddings }  \\     \alpha_0: \A_0 \to M_2(\A)  \\   \alpha_3 :\A_3 \to M_2(\A)  \end{array} \right\} ,
\]
and the image of $\phi$ under this bijection is represented by a pair of $\mathbb{O}$-algebra embeddings
$ \alpha_0: \mathbb{O}_0 \to M_2(\mathbb{O})$ and   $\alpha_3 :\mathbb{O}_3 \to M_2(\mathbb{O})$.

As both $K_0$ and $K_3$ are unramified over $F$, the  quartic $\mathbb{O}$-algebra 
\[
R= \mathbb{O}_0\otimes_\mathbb{O} \mathbb{O}_3
\]
 is self-dual with respect to the  bilinear form $(x,y)\mapsto \mathrm{Tr}_{R/\mathbb{O}}(xy)$.
 If we define an $\mathbb{O}$-linear map 
$\alpha : R \to M_2(\mathbb{O})$ by $\alpha(x_0\otimes x_3) = \alpha_0(x_0)\alpha_3(x_3)$, then 
tracing the construction of the invariant (\ref{orbit invariant}) all the way back to Proposition \ref{prop:xi construct} shows that $\xi\in K_3$ satisfies 
\[
\mathrm{Trd}_{M_2(\mathbb{O})}( \alpha(x) \alpha(y)^\iota )= \mathrm{Tr}_{ R / \mathbb{O}}( \xi x \overline{y} )
\]
for all $x,y\in R$.   Here $\iota$ is the main involution on the quaternion order $M_2(\mathbb{O})$, and $y\mapsto \overline{y}$ is the involution on $R$ defined by $x_0\times x_3\mapsto x_0^{\sigma_0}\otimes x_3^{\sigma_3}$.  
  The left hand side clearly lies in $\mathbb{O}$ for all choices of $x$ and $y$, and hence $\xi\in R$, by   the self-duality of $R$ noted above.

Recalling that $K_3=\kk(Y_3)$ is the field of rational functions on a projective and geometrically connected curve,  
\[
 \xi \in K_3 \cap R = K_3\cap \mathbb{O}_3=\kk,
\]
and the condition $2\xi =1$ then follows from $\mathrm{Tr}_{K_3/F}(\xi)=1$.
\end{proof}

Returning to the main proof, fix a $\gamma \in J(F)$ and recall from \S \ref{ss:orbital notation} the notation
\begin{equation}\label{partial J}
\J (\gamma,f,s) = \int_{ T_0(\A)  \times T_3(\A) } f(t_0^{-1} \gamma t_3) \, |t_0|^{2s} \eta(t_3)\, dt_0\, dt_3.
\end{equation}
If (\ref{partial J}) nonzero, the lemma implies that the invariant $\xi = \inv(\gamma) \in K_3$  lies in the field of constants $\kk$ and satisfies  $2\xi =1.$  If $\mathrm{char}(\kk)=2$ there is no such $\xi$,   and so  (\ref{partial J}) vanishes for all  $\gamma\in J(F)$.   
The first claim of Theorem \ref{f_0 computation} follows from this and the decomposition   (\ref{J decomp}).

From now on we assume that $\mathrm{char}(\kk)>2$, and  let 
\[
\gamma \in T_0(F) \backslash J(F) / T_3(F)
\] 
be the unique element with $\inv(\gamma)=1/2$.  Thus, by the discussion above,
\begin{equation}\label{lonely sum}
\J( f,s) =  \int_{ T_0(\A)  \times T_3(\A) } f(t_0^{-1} \gamma t_3) \, |t_0|^{2s} \eta(t_3)\, dt_0\, dt_3.
\end{equation}

Fix an   $\epsilon \in K_3^\times$ satisfying $\mathrm{Tr}_{K_3/F}(\epsilon)=0$, and define an $F$-linear isomorphism $\phi :  K_3 \to  K_0$ by $\phi  ( x+y\epsilon) = (x,y).$
By carefully unwinding the definition of the invariant (\ref{orbit invariant}), one can see that $\phi \mapsto \gamma$ under the canonical bijection
\[
 K_0^\times \backslash  \mathrm{Iso}( K_3 ,  K_0 )  / K_3^\times = T_0(F) \backslash J(F) /T_3(F).
\]

\begin{lemma}
If we factor $f=\prod_{x\in |X|} f_x$ and $\eta=\prod_{x\in |X|} \eta_x$,  then 
\begin{equation}\label{local orbital}
 \int_{ T_0(F_x)  \times T_3(F_x) } f_x(t_0^{-1} \gamma t_3) \, |t_0|^{2s} \eta_x(t_3)\, dt_0\, dt_3 = |  \epsilon |_x^{-2s}
\end{equation}
for every place $x$ of $F$.
\end{lemma}

\begin{proof}
As $K_3/F$ is unramified, we may  choose  $c \in F_x^\times$ so that $c \epsilon \in \co_{K_{3, x}}^\times$.  For any such choice we have 
$
\co_{K_3,x} = \co_{F,x} \oplus c\epsilon \co_{F,x},
$
and hence  \[ \phi(\co_{K_{3,x}}) = (1,c) \cdot \co_{K_{0,x}}.\]

Suppose first that $x$ is inert in $K_3$.   The integral over $T_3(F_x)$ can be replaced by a sum over the singleton set
$F_x^\times \backslash K^\times_{3,x} / \co^\times_{K_{3,x}} =\{1\}$, while the integral over $T_0(F_x)$ can be replaced by a sum over
\[
F_x^\times \backslash (F_x^\times \times F_x^\times) / ( \co_{F_x}^\times \times \co_{F_x}^\times ) =
\{ (1,\varpi^k) : k\in \Z\}
\]
for any uniformizer $\varpi \in F_x$.   Moreover,  $f_x( ( 1, \varpi^{-k}) \cdot  \gamma )$ is equal to $1$ if the $\co_{F_x}$-lattices $\phi(\co_{K_{3,x} } )$ and $( 1, \varpi^{k}) \cdot  \co_{ K_{0,x}}$  agree up to scaling by $F_x^\times$, and is $0$ otherwise. 
 In other words
\[
f_x( ( 1, \varpi^{-k}) \cdot  \gamma ) = \begin{cases}
1 & \mbox{if } | \varpi |^k = |c|  \\
0 & \mbox{otherwise,}
\end{cases}
\]
and  the integral (\ref{local orbital})  reduces to 
\[
\sum_{k\in \Z}  | \varpi | ^{2ks}      f_x \big( ( 1, \varpi^{-k}) \cdot  \gamma \big)  = |c|^{2s} = |\epsilon|_x^{-2s}  .
\]

Now suppose that $x$ is split in $K_3$. In this case we can choose $c$ in such a way that $(c\epsilon)^2=1$, and define orthogonal idempotents
\[
e = \frac{1+c\epsilon}{2} ,\quad f = \frac{1-c\epsilon}{2}
\]
in $\co_{K_{3,x}}$.  The integral over $T_0(F_x)$ can then be replaced by a sum over
\[
F_x^\times \backslash K_{3,x}^\times / \co_{K_{3,x}}^\times = \{  e+ \varpi^\ell  f   : \ell \in \Z   \} .
\]
Moreover, $f_x\big( ( 1, \varpi^{-k}) \cdot  \gamma \cdot (e+\varpi^\ell f) \big)$ is equal to $1$ if the $\co_{F_x}$-lattices
\[
\phi \big(  (e+\varpi^\ell f)  \cdot  \co_{K_{3,x}} \big) = \left\{ (x + \varpi^\ell y , cx -  \varpi^\ell c y) \in K_{0,x} : x,y  \in   \co_{F_x} \right\} 
\]
and $( 1, \varpi^{k}) \cdot  \co_{ K_{0,x}}$  agree up to scaling by $F_x^\times$, and is $0$ otherwise.   
After some elementary linear algebra,  this simplifies to
\[
f_x\big( ( 1, \varpi^{-k}) \cdot  \gamma \cdot (e+\varpi^\ell f) \big)     = \begin{cases}
1 & \mbox{if } | \varpi |^k = |c|  \mbox{ and } \ell=0 \\
0 & \mbox{otherwise,}
\end{cases}
\]
and the integral (\ref{local orbital}) again  reduces to 
\[
\sum_{k,\ell \in \Z}  | \varpi | ^{2ks}   f_x\big( ( 1, \varpi^{-k}) \cdot  \gamma \cdot (e+\varpi^\ell f) \big) 
\, \eta_x(  e+\varpi^\ell f )     = |c|^{2s} = |\epsilon|_x^{-2s}  .
\]
This proves the lemma.
\end{proof}

Combining (\ref{lonely sum}) with the preceding lemma yields
$
\J(f,s) = |\epsilon|^{-2s}=1,
$
completing the proof of the second claim of Theorem \ref{f_0 computation}.
\end{proof}

The following is Theorem \ref{thm:bare intersection} of the introduction.

\begin{theorem}\label{thm:unprojected}
Let  $\langle\cdot,\cdot\rangle :\Ch_{c,r}(\Sht_{G_0}^r) \times \Ch_{c,r}(\Sht_{G_0}^r) \to \Q$
be the intersection pairing.
\begin{enumerate}[$(a)$]
\item If $r>0$ then $\langle  [\Sht_{T_1}^r ]  ,   \, [\Sht_{T_2}^r ] \rangle  = 0$.
\item If $r=0$ then  
\[
\langle  [\Sht_{T_1}^0 ]  , \,   [\Sht_{T_2}^0 ] \rangle = 
\begin{cases} 1 &  \mbox{ if }  \mathrm{char}(\kk)>2 \\
0 & \mbox{ if } \mathrm{char}(\kk)=2 .
\end{cases}
\]
\end{enumerate}
%
%
\end{theorem}

\begin{proof}
If  $f \in \mathscr{H}$ denotes the characteristic function of $U_0$, then  Proposition \ref{prop:key identity} implies
\[
\langle [\Sht_{T_1}^r ]  ,   [\Sht_{T_2}^r ] \rangle = \I_r(f) = 
(\log q)^{-r} \frac{d^r}{ds^r} \J( f,s) \big|_{s=0} .
\]
The claim follows from this and Proposition \ref{f_0 computation}.
\end{proof}


\subsection{The proof of Theorem \ref{thm:periods}}
\label{ss:r=0}


Throughout \S \ref{ss:r=0} we assume that $r=0$, and omit the $r$ superscripts in our moduli spaces  $\Sht_{G_0}$, $\Sht_{T_1}$, and $\Sht_{T_2}$.  All three of these are  disjoint unions of stack quotients of $\Spec(\kk)$, and the latter two have only finitely many connected components.

On $\kk^\alg$-points (or, equivalently, $\kk$-points), the morphisms 
\[
\theta_i:\Sht_{T_i} \to \Sht_{G_0}
\]
introduced in \S \ref{ss:cycles} are identified with
\[
\xymatrix{
{  \Sht_{T_1} (\kk^\alg) }  \ar@{=}[d] \ar[dr]^{\theta_1} &   &    { \Sht_{T_2} (\kk^\alg) }  \ar@{=}[d] \ar[dl]_{\theta_2}   \\
  { T_1(F) \backslash T_1(\A) / T_1(\mathbb{O}) } \ar[d]   &  { \Sht_{G_0} (\kk^\alg)} \ar@{=}[d]  &   { T_2(F) \backslash T_2(\A) / T_2(\mathbb{O}) }   \ar[d]   \\
   { G_1(F) \backslash G_1(\A) / U_1 }  \ar@{=}[r] & { G_0(F) \backslash G_0(\A) / U_0 } \ar@{=}[r] & { G_2(F) \backslash G_2(\A) / U_2 } .
}
\]
where each ``='' is a canonical isomorphism.
The isomorphisms of the bottom row come from Lemma \ref{lem:naive transfer}, while the others follow by examination of the definition of a Shtuka with $0$ modifications.

The Chow group of 0-cycles with proper support $\Ch_{c,0}(\Sht_{G_0})_\R$ is identified with the space 
\[
\mathscr{A}_\R = C_c^{\infty}(G_0(F)\backslash G_0(\A)/U_0, \R)
\] 
of compactly supported, $\R$-valued, unramified automorphic forms.  Under this identification, the intersection pairing becomes 
\begin{align*}
\langle \phi_1, \phi_2 \rangle 
&= \int_{ G_0(F) \backslash G_0(\A)} \phi_1(g)\phi_2(g)\,  dg 
\end{align*}
where the Haar measure on $G_0(\A)$ is normalized as in \S \ref{ss:notation}.

Fix $i \in \{1,2\}$.  As in Lemma \ref{lem:naive transfer}, fix an $F$-linear isomorphism $\rho_i : K_0 \to K_i$.
Denote in the same way the induced isomorphism 
\[
\rho_i : G_{0/F} \to G_{i/F},
\]
 and by 
 \begin{equation}\label{torus fixed}
 \alpha_i : T_{i/F} \to G_{0/F}
 \end{equation}
  the composition of the inclusion  $T_{i/F} \subset G_{i/F}$ with $\rho_i^{-1}$.
If we  choose  a $b_i \in G_0(\A)$ such that 
$
\rho_i(  b_i U_0 b_i^{-1}  ) =   U_i   , 
$
the injection
\begin{equation}\label{simple class}
 T_i(\A) / T_i(\mathbb{O})  \map{  t \mapsto \alpha_i(t) b_i  } G_0(\A) /U_0
\end{equation}
 induces the function denoted $\theta_i$ in the diagram above.
It follows that the class $[\Sht_{T_i}]  \in \Ch_{c,0}(\Sht_{G_0})_\R$ is identified with the 
$\psi_i \in \mathscr{A}_\R$ defined by
\[
\psi_i (g ) =  \sum_{ \gamma \in  G_0(F) / \alpha_i(T_i(F))  } \bm{1}_{   \theta_i    } ( \gamma^{-1}  g ).
\]
Here $ \bm{1}_{   \theta_i  }$ is the characteristic function of the image of (\ref{simple class}).

The following is Theorem \ref{thm:periods} of the introduction.

\begin{theorem}
Let $\pi$ be an unramified cuspidal automorphic representation of $G_0(\A)$.
For  any $\phi \in \pi^{U_0}$ we have
\begin{multline*}
   \left( \int_{T_1(F)\backslash T_1(\A)} \phi(t_1) dt_1 \right)  \left( \int_{T_2(F) \backslash T_2(\A)} \overline{\phi}(t_2) dt_2\right) \\
    = 
    \left( \int_{T_0(F)\backslash T_0(\A)} \phi(t_0) dt_0 \right)  \left( \int_{T_3(F) \backslash T_3(\A)} \overline{\phi}(t_3)\eta(t_3)dt_3\right).
 \end{multline*}
\end{theorem} 

\begin{proof}
By Remark \ref{rem:real valued}, we may assume that  $\phi \in \pi^{U_0}$  is $\R$-valued, so that $\phi \in \mathscr{A}_\R$.
The projection of $[\Sht_{T_i}] \in \mathscr{A}_\R$ onto the $\pi$-isotypic component is 
\[
[\Sht_{T_i}]_\pi = \dfrac{\langle \psi_i, \phi \rangle}{\langle \phi, \phi\rangle } \phi
= \dfrac{1}{\langle \phi, \phi\rangle_{\mathrm{Pet}}}   \left( \int_{[T_i]} \phi(t_i) dt_i \right) \cdot    \phi,
\] 
from which we deduce
\[
\left\langle [\Sht_{T_1}]_\pi, [\Sht_{T_2}]_\pi \right\rangle =    \dfrac{1}{\langle \phi, \phi \rangle_{\mathrm{Pet}}}\left( \int_{[T_1]} \phi(t_1) dt_1 \right)  \left( \int_{[T_2]} \overline{\phi}(t_2) dt_2\right).
\]
On the other hand,  Theorem \ref{thm:main intro} implies 
\[
\left\langle [\Sht_{T_1}]_\pi, [\Sht_{T_2}]_\pi \right\rangle =   \dfrac{1}{\langle \phi, \phi \rangle_{\mathrm{Pet}}} \left( \int_{[T_0]} \phi(t_0) dt_0 \right)  \left( \int_{[T_3]} \overline{\phi}(t_3)\eta(t_3)dt_3\right),
\]
and we are done.
\end{proof}

Finally, we give a reinterpretation of the $r=0$ case of Theorem \ref{thm:unprojected} (a.k.a.~Theorem \ref{thm:bare intersection} of the introduction).

\begin{definition}\label{def:optimal}
A pair of $F$-algebra embeddings 
\[
\alpha_1 : K_1 \to M_2(F) ,\qquad \alpha_2: K_2 \to M_2(F)
\]
is \emph{optimal} if there is a maximal $\mathbb{O}$-order in $M_2(\A)$  containing $\alpha_1(\mathbb{O}_1) \cup \alpha(\mathbb{O}_2)$.
\end{definition}

Note that the set of all optimal pairs $(\alpha_1,\alpha_2)$ is stable under the conjugation action of $\GL_2(F)$.

\begin{proposition}\label{prop:optimal}
If $(\alpha_1,\alpha_2)$ is an optimal pair, then the maximal order of $M_2(\A)$  containing $\alpha_1(\mathbb{O}_1) \cup \alpha_2(\mathbb{O}_2)$ is unique.  Moreover,  the number of $\GL_2(F)$-conjugacy classes of optimal pairs $(\alpha_1,\alpha_2)$ is equal to
\[
\langle  [\Sht_{T_1} ]  , \,   [\Sht_{T_2} ] \rangle = 
\begin{cases} 1 &  \mbox{ if }  \mathrm{char}(\kk)>2 \\
0 & \mbox{ if } \mathrm{char}(\kk)=2 .
\end{cases}
\]
\end{proposition}

\begin{proof}
Using the notation and discussion of \S \ref{ss:r=0}, one can easily check that
\begin{align}
\langle  [\Sht_{T_1} ]  , \,   [\Sht_{T_2} ] \rangle &=  
\int_{ G_0(F) \backslash G_0(\A)}  \psi_1(g)  \psi_2  (g)\,  dg  \nonumber \\
&= \sum_{     (\gamma_1,\gamma_2) }  \sum_{  g\in   G_0(\A)/U_0    }  \mathbf{1}_{\theta_1}(\gamma_1^{-1} g)  \mathbf{1}_{\theta_2}(  \gamma_2^{-1} g) , \label{optimal count}
\end{align}
where the outer summation is over all pairs
\[
(\gamma_1,\gamma_2)  \in    G_0(F) \backslash   \big(  G_0(F)/  \alpha_1(T_1(F)) \times G_0(F) / \alpha_2( T_2(F) )   \big).
\]

There is a bijection from $G_0(\A) /U_0=\mathrm{PGL}_2(\A) / \mathrm{PGL}_2(\mathbb{O})$ to the set of maximal orders in $M_2(\A)$, sending 
 $g \mapsto gM_2(\mathbb{O})g^{-1}$.
Recalling the fixed embeddings (\ref{torus fixed}), every  $(\gamma_1,\gamma_2)$ in the outer summation gives rise to a pair of embeddings 
\[
\gamma_1 \alpha_1\gamma_1^{-1} : T_{1/F} \to G_{0/F} ,\quad \gamma_2 \alpha_2 \gamma_2^{-1} : T_{2/F} \to G_{0/F},
\]
which arise from $F$-algebra maps 
\[
K_1\to \End_F(K_0) \iso M_2(F) ,\qquad  K_2\to \End_F(K_0) \iso M_2(F).
\]
It is an exercise in linear algebra to check that 
   $\mathbf{1}_{\theta_1}(\gamma_1^{-1} g)  \mathbf{1}_{\theta_2}(  \gamma_2^{-1} g)=1$  if and only if these embeddings take both $\mathbb{O}_1$ and $\mathbb{O}_2$ into   $gM_2(\mathbb{O})g^{-1}$.

In other words, we can rewrite (\ref{optimal count}) as
\[
\langle  [\Sht_{T_1} ]  , \,   [\Sht_{T_2} ] \rangle 
= \sum_{ (\alpha_1,\alpha_2)} \sum_{ R \subset M_2(\A) } \mathrm{opt}(\alpha_1,\alpha_2,R)
\]
where the outer sum is over all $\GL_2(F)$-conjugacy classes of  pairs $(\alpha_1,\alpha_2)$ as in Definition \ref{def:optimal},
the inner sum is over all maximal orders $R\subset M_2(\A)$, and 
\[
\mathrm{opt}(\alpha_1,\alpha_2,R) = \begin{cases}
1 &\mbox{if } \alpha_1(\mathbb{O}_1) \cup \alpha_2(\mathbb{O}_2) \subset R \\
0 & \mbox{otherwise.}
\end{cases}
\]
Both claims of the proposition follow from this and Theorem \ref{thm:unprojected}.
\end{proof}

\begin{remark}\label{rem:no optimal}
Suppose $(\alpha_1,\alpha_2)$ is an  optimal pair as in Definition \ref{def:optimal}.
In the notation of \S \ref{ss:invariants}, this data determines a quaternion embedding
\[
(M_2(F) , \alpha_1,\alpha_2) \in \mathcal{Q}(K_1,K_2).
\]
As in the proof of Lemma \ref{lem:optimal},   optimality forces  
$\xi = \inv(M_2(F) , \alpha_1,\alpha_2)$ to be everywhere locally integral, so that  $\xi\in K_3\cap \mathbb{O}_3$ is a rational function on $Y_3$ having no poles.
As $Y_3$ is  geometrically connected, this forces  $\xi\in \kk$, and hence
\[
1 = \mathrm{Tr}_{K_3/F}(\xi)  = 2\xi.
\]
This gives a more direct proof that no optimal pair can exist when $\mathrm{char}(\kk)=2$.
It also provides a strengthening of  Proposition \ref{prop:optimal}:
if $\mathrm{char}(\kk)>2$ the unique $\GL_2(F)$-conjugacy class of optimal pairs is characterized by 
\[
\inv(M_2(F) , \alpha_1,\alpha_2) = \frac{1}{2} \in K_3.
\]
\end{remark}



\bibliographystyle{amsalpha}


\providecommand{\bysame}{\leavevmode\hbox to3em{\hrulefill}\thinspace}
\providecommand{\MR}{\relax\ifhmode\unskip\space\fi MR }
\providecommand{\MRhref}[2]{%
  \href{http://www.ams.org/mathscinet-getitem?mr=#1}{#2}}

\end{document}